\documentclass[twoside,10pt]{article}
\usepackage{geometry}
\usepackage{graphicx}
\usepackage{indentfirst}                                
\usepackage[colorlinks]{hyperref}
\hypersetup{linkcolor=blue,filecolor=black,urlcolor=blue, citecolor=black}   
\usepackage[numbers,sort&compress]{natbib}         
\usepackage[perpage,symbol*,marginal]{footmisc}   
\usepackage{amsmath}                                
\usepackage{amsfonts}
\usepackage{amssymb}                                
\usepackage{bm}                                            
\usepackage{mathrsfs}                                 
\usepackage{amsthm}                                   
\usepackage{amsxtra}
\usepackage[numbers]{natbib}
\ifxetex
\usepackage{letltxmacro}
\LetLtxMacro\SavedIncludeGraphics\includegraphics
\def\includegraphics#1#{
	\IncludeGraphicsAux{#1}%
}%
\newcommand*{\IncludeGraphicsAux}[2]{%
	\XeTeXLinkBox{%
		\SavedIncludeGraphics#1{#2}%
}}
\fi
\newcommand\orcidicon[1]{\href{https://orcid.org/#1}{\includegraphics[scale=0.02]{orcid.pdf}}}
\usepackage[titletoc,title]{appendix}
\usepackage{titlesec}
\usepackage{authblk}
\titleformat{\section}{\large\bfseries\centering}{\thesection}{1em}{}
\titleformat{\subsection}{\centering\bfseries}{\thesubsection}{0.5em}{}
\titleformat{\subsubsection}[runin]{\bfseries}{\thesubsubsection.}{0.4em}{}[.]
\numberwithin{equation}{section}
\newtheorem{lemma}{Lemma}[section]
\newtheorem{proposition}[lemma]{Proposition}
\newtheorem{theorem}{Theorem}[section]

\newtheorem{remark}{Remark}[section]
\usepackage{color}
\definecolor{DarkRed}{RGB}{139,0,0}
\definecolor{Purple}{RGB}{128,0,128}
\usepackage{ulem}
	
	\usepackage{accents}
	\makeatletter
	\def\wideubar{\underaccent{{\cc@style\underline{\mskip10mu}}}}
	\def\Wideubar{\underaccent{{\cc@style\underline{\mskip8mu}}}}
	\makeatother
	\makeatletter
	\def\widebar{\accentset{{\cc@style\underline{\mskip10mu}}}}
	\def\Widebar{\accentset{{\cc@style\underline{\mskip8mu}}}}
	\makeatother
	\newcommand{\VERTiii}[1]{{\left\vert\kern-0.3ex\left\vert\kern-0.3ex\left\vert #1
			\right\vert\kern-0.3ex\right\vert\kern-0.3ex\right\vert}}
	\newcommand{\VERT}{\vert\kern-0.3ex\vert\kern-0.3ex\vert}
	\newcommand{\VERTl}{\left\vert\kern-0.3ex\left\vert\kern-0.3ex\left\vert}
	\newcommand{\VERTr}{\right\vert\kern-0.3ex\right\vert\kern-0.3ex\right\vert}
	\newcommand{\VERTbig}{\big\vert\kern-0.3ex\big\vert\kern-0.3ex\big\vert}
	\newcommand{\VERTBig}{\Big\vert\kern-0.3ex\Big\vert\kern-0.3ex\Big\vert}
\begin{document}
		\title{\bf A Rigorous Derivation of the Vlasov-Navier-Stokes Model from
			Multicomponent Boltzmann Equations
		}
		
		\author[a]{Shuping Li\thanks{ Email: lisp@hnu.edu.cn}}
		\author[a,b]{Linjie Xiong\thanks{Corresponding author. Email: xlj@hnu.edu.cn}}
		\affil[a]{\small School of Mathematics, Hunan University, Changsha 410082, China}
		\affil[b]{\small Hunan Provincial Key Laboratory of Intelligent Information Processing and Applied Mathematics, Hunan University, Changsha 410082, China}
		
		\date{\empty}
		\maketitle
		
		\begin{abstract}
			The rigorous justification of the Vlasov–Navier–Stokes system remains an outstanding open problem, both as a mean-field limit of an ‭$N$‬-particle fluid-interacting system and as a hydrodynamic limit derived from multiphase Boltzmann equations. Inspired by the work of Bernard, Desvillettes, Golse, and Ricci [{\it Commun. Math. Sci.}, {\bf 15}(6), 1703–1741, 2017], we present a rigorous derivation of the incompressible Vlasov–Navier–Stokes system from the two-component Boltzmann system for elastic hard-sphere collisions. To this end, we establish the well-posedness of the rescaled multicomponent Boltzmann system for small initial data. Specifically, we obtain estimates for the solution that hold uniformly with respect to both the thermal speed ratio ‭$\varepsilon$‬ and the mass ratio ‭$\eta$‬. Furthermore, under the Vlasov–Navier–Stokes scaling assumption ‭$O(1)\varepsilon^3 \leqslant \eta \leqslant o(1)\varepsilon^2$‬‭‬‭‬‭‬‭‬‭‬‭‬‭‬‭‬‭‬‭‬ as ‭$\varepsilon \rightarrow 0$‬‭‬, we establish the weak convergence of the solution to the fluid–kinetic coupled limit system.
		\end{abstract}
		
		\vspace*{2mm}
		{\small

			\vspace*{2mm}
			
			

			\tableofcontents
		}
		
		\bigskip
		
		\section{Introduction}
		\subsection{IVNS System and Multicomponent Boltzmann Equations}
		When describing the motion of a dispersed phase of small particles flowing in a surrounding incompressible fluid—such as solid particles or liquid droplets immersed in a gas—physicists have introduced fluid-particle models, most notably the incompressible Vlasov–Navier–Stokes (IVNS) system:
		\begin{equation}\label{V-N-S*}
			\left\{\begin{aligned}
				&\partial_t F+v\cdot \nabla_x F=\kappa {\text{div}}_v((v-u)F),\\
				&	\partial_t u+u\cdot \nabla_x u+\nabla_x p=\widetilde{\nu} \Delta_x u+\kappa\int_{\mathbb{R}^3}(v-u)F\mathrm{d}v,\\
				&	\partial_t \theta+u \cdot \nabla_x\theta=\widetilde{\kappa}\Delta_x\theta,\\
				&	{\text{div}}_x u=0,\quad	\rho+\theta=0,
			\end{aligned} \right.
		\end{equation}
		where ‭$F(t,x,v)\geqslant 0$‬‭‬‭‬‭‬‭‬‭‬‭‬‭‬‭‬ is the number density of particles or droplets with velocity ‭$v$‬ at position ‭$x$‬ and time ‭$t$‬, ‭$\rho(t,x)$‬‭‬‭‬‭‬‭‬‭‬ is the fluid mass density, ‭$u(t,x)$‬‭‬‭‬‭‬‭‬‭‬ is the fluid velocity, ‭$\theta(t,x)$‬‭‬‭‬‭‬‭‬‭‬ is the fluid temperature, and ‭$p(t,x)$‬‭‬‭‬‭‬‭‬‭‬ is the fluid pressure. The physical parameters ‭$\kappa$‬, ‭$\widetilde{\nu}$‬, and ‭$\widetilde{\kappa}$‬ represent the gas friction coefficient on the dispersed phase, the kinematic viscosity, and the thermal diffusivity of the gas, respectively (see \cite{MR3668954} for details). For further discussion on its physical relevance, we refer the reader to Caflisch–Papanicolaou \cite{MR0709743} and Williams \cite{Williams-85} (see also Goudon–Jabin–Vasseur \cite{MR2106333}).
		
		Driven by both physical importance and mathematical complexity, the global existence of solutions to this kinetic system coupled with incompressible fluid dynamics has been extensively studied. Anoshchenko, and Boutet de Monvel-Berthier \cite{Ano-Bou} as well as Hamdache \cite{MR1610309} pioneered the investigation of the Cauchy problem for Vlasov–Navier–Stokes models. Boudin, Desvillettes, Grandmont, and Moussa \cite{MR2555647} established the global existence of weak solutions on the torus ‭$\mathbb{T}^3$‬. This result was subsequently extended to general bounded domains in ‭$\mathbb{R}^3$‬ by Wang–Yu \cite{MR3369269} and to time-dependent bounded domains by Boudin–Grandmont–Lorz–Moussa \cite{MR3582196}. Yu \cite{MR3073216} proved the global existence of weak solutions in two and three spatial dimensions, alongside uniqueness in 2D. Choi and Kwon studied the global solvability and large-time behavior for inhomogeneous Vlasov–Navier–Stokes equations under small, regular initial data \cite{MR3403400}, while also exploring finite-time blow-up scenarios \cite{MR3723164}. Han-Kwan and collaborators investigated asymptotic stability \cite{MR3842054} and uniqueness \cite{MR4061981} in a 2D pipe with a non-trivial kinetic phase, as well as the large-time behavior of Fujita–Kato type solutions on the 3D torus \cite{MR4076066} and the whole space ‭$\mathbb{R}^3$‬ \cite{MR4420295}. Very recently, Danchin \cite{MR5041103} established the unique existence and optimal decay rates of global solutions initialized in a velocity space of critical regularity under a Fujita–Kato type framework. A vast literature exists on strong and weak solutions for both homogeneous and inhomogeneous fluid regimes; see, e.g., \cite{MR2904273,MR3465376,MR2825335,Springerlink-ews}.
		
		Regarding the hydrodynamic limit of the Vlasov–Navier–Stokes system itself, Goudon, Jabin, and Vasseur first explored various asymptotic regimes using a relative entropy method for light particles \cite{MR2106333} and fine particles \cite{MR2106334}. El Ghani and Mejri \cite{MR4591940} established the hydrodynamic limit for the inhomogeneous Vlasov–Navier–Stokes system in a bounded 2D domain subject to homogeneous Dirichlet boundary conditions for the fluid and Maxwell boundary conditions for the kinetic distribution. Su, Wu, Yao, and Zhang \cite{MR4493149} analyzed this limit on the 3D torus. More recently, Han-Kwan and Michel \cite{Han-Kwan-Michel-23} developed a unified framework to justify hydrodynamic limits under high-friction regimes. For other asymptotic properties, see \cite{MR1869641,MR4344262,MR2415460}.
		
		Rigorously deriving the Vlasov–Navier–Stokes system as a mean-field model—where the empirical phase-space measure of solid particles (or droplets) converges as the particle number approaches infinity and the particle radius vanishes under a suitable scaling—is notoriously challenging and remains a major open problem (see \cite{MR1079189,MR2398959,MR4172441,MR2094523,MR0565234}).
		Nevertheless, in the spirit of the Bardos–Golse–Levermore program \cite{MR1115587,MR1213991} and the full justification of Boltzmann hydrodynamic limits within DiPerna–Lions' renormalized solution framework \cite{MR2025302,MR2517786,MR1014927}, Bernard, Desvillettes, Golse, and Ricci \cite{MR3668954} proposed a theoretical program to formally derive the IVNS system from multiphase Boltzmann equations. Furthermore, based on Hilbert or Chapman–Enskog expansions near equilibrium with uniform Knudsen-number estimates, foundational works such as \cite{MR1115292,MR0503305,MR0553964,MR1029125,MR3875244,MR3894734} obtained the Euler or Navier–Stokes limits of sufficiently smooth solutions to scaled Boltzmann equations.
		
		As noted by Han-Kwan and Michel \cite{Han-Kwan-Michel-23}, a complete rigorous justification of this IVNS limit remains open. In this paper, we provide a rigorous proof deriving the IVNS model \eqref{V-N-S*} from a microscopic aerosol flow description. Specifically, we investigate the hydrodynamic limit of the following Boltzmann system arising in \cite{MR3668954,MR4172441}, which models a binary mixture of gas molecules and substantially larger dust particles or liquid droplets:
		\begin{equation}\label{bi-Boltz}
			\left\{ \begin{aligned}
				& \partial_t F+v \cdot \nabla_x F=\mathcal{D}(F, f)+\mathcal{B}(F,F), \\
				& \partial_tf+w \cdot \nabla_x f=\mathcal{R}(f, F)+\mathcal Q (f, f),
			\end{aligned}  \right.
		\end{equation}
		where $F(t, x, v) \geqslant 0$ is the distribution function of dust particles or droplets with velocity $v\in \mathbb R^3$ located at position $x\in \mathbb R^3$ at time $t>0$, and $ f(t, x, w) \geqslant 0$ is the distribution function of gas molecules with velocity $w\in \mathbb R^3$ located at the position $x\in \mathbb R^3$ at time $t>0$.
		The terms $\mathcal{B}(F,F)$ and $\mathcal Q ( f, f)$ are the classical Boltzmann collision integrals for pairs of dust particles or liquid droplets and for pairs of gas molecules respectively.
		The terms $\mathcal{D}(F, f)$ and $\mathcal{R}(f, F)$ are Boltzmann type collision integrals describing the deflection of dust particles or liquid droplets due to the impingement of gas molecules, and the slowing down of gas molecules caused by collisions with dust particles or liquid droplets respectively.
		The exact collision operators $\mathcal Q, \mathcal B, \mathcal D$ and $\mathcal R$ are tabulated below:
		\begin{align*}
			\mathcal Q(f,g)(w)&= \int_{\mathbb{R}^3\times\mathbb{S}^2} \mathfrak{c}(w-w_*,\omega) \{f(w')g(w_*')-f(w)g(w_*)\}\mathrm{d}\omega\mathrm{d}w_*,\\
			\mathcal B(F,G)(v)&= \int_{\mathbb{R}^3\times\mathbb{S}^2} \mathfrak{d}(v-v_*,\omega) \{F(v')G(v_*')-F(v)G(v_*)\}\mathrm{d}\omega\mathrm{d}v_*,\\
			\mathcal D(F,f)(v) & =\int_{\mathbb{R}^3\times\mathbb{S}^2}\mathfrak{b}(v-w,\omega)\{F(v^{\prime\prime})f(w^{\prime\prime})-F(v)f(w)\}\mathrm{d}\omega\mathrm{d}w,\\
			\mathcal R(f,F)(w) & =\int_{\mathbb{R}^3\times\mathbb{S}^2}\mathfrak{b}(v-w,\omega) \{F(v^{\prime\prime})f(w^{\prime\prime})-F(v)f(w)\}\mathrm{d}\omega\mathrm{d}v,
		\end{align*}
		where $ \mathbb{S}^2=\{\omega\in \mathbb{R}^3 \mid |\omega|=1\}$.
		Here, $w'$ and $w_*'$ are the velocities of two gas molecules after an elastic collision, which are related to the pre-collision velocities $w$ and $w_*$ by
		\[
		\begin{aligned}
			w'=w+\omega(\omega\cdot (w_*-w)),\quad
			w_*'=w_*-\omega(\omega\cdot (w_*-w)).
		\end{aligned}
		\]
		Similarly, $v'$ and $v_*'$ are the velocities of two particles after an elastic collision, which are related to the pre-collision velocities $v$ and $v_*$ by
		\[
		\begin{aligned}
			v'=v+\omega(\omega\cdot (v_*-v)),\quad
			v_*'=v_*-\omega(\omega\cdot (v_*-v)).
		\end{aligned}
		\]
		Finally, ‭$v^{\prime\prime}$‬ and ‭$w^{\prime\prime}$‬ represent post-collision velocities resulting from an elastic collision between a particle and a gas molecule with initial velocities ‭$v$‬ and ‭$w$‬:
		\begin{equation*}
			v^{\prime\prime}=v-\frac{2m_g}{m_g+m_p}\omega(\omega\cdot (v-w)),\quad
			w^{\prime\prime}=w+\frac{2m_p}{m_g+m_p}\omega(\omega\cdot (v-w)),
		\end{equation*}
		where ‭$m_g$‬ and ‭$m_p$‬ denote the masses of a gas molecule and a particle/droplet, respectively. The collision kernels ‭$\mathfrak{c}$‬, ‭$\mathfrak{d}$‬, and ‭$\mathfrak{b}$‬ are defined by
	\begin{align*}
		\mathfrak{c}(w-w_*,\omega)&=|w-w_*|\sigma_{gg}(|w-w_*|,|\cos\widehat{(w-w_*,\omega)}|),\\
		\mathfrak{d}(v-v_*,\omega)&=|v-v_*|\sigma_{pp}(|v-v_*|,|\cos\widehat{(v-v_*,\omega)}|),\\
		\mathfrak{b}(v-w,\omega)&=|v-w|\sigma_{pg}(|v-w|,|\cos\widehat{(v-w,\omega)}|),
	\end{align*}
		where ‭$\sigma_{gg}$‬, ‭$\sigma_{pp}$‬, and ‭$\sigma_{pg}$‬ denote the differential cross-sections for gas–gas, particle–particle, and particle–gas interactions, respectively.
		
		\subsection{Dimensionless Analysis}
		Our aim in the present work is to capture the asymptotic dynamical behavior of aerosol flows starting from the multicomponent Boltzmann system \eqref{bi-Boltz}. To this end, we must first formulate the dimensionless version of the system. This dimensional analysis follows the framework established in \cite{MR3668954}, which we summarize below for completeness.
		
		We begin by introducing a characteristic length unit, $L$, and defining the dimensionless spatial variable:
		\[
		\widehat{x}=\frac{x}{L}.
		\]
		Let $V_p$ and $V_g$ denote the thermal speeds of the particles and gas molecules, respectively. Measuring the velocity of each species against its respective thermal speed, we define the dimensionless velocity variables:
		\[
		\widehat{v}=\frac{v}{V_p},\quad \widehat{w}=\frac{w}{V_g}.
		\]
		Based on these choices, the dimensionless post-collision velocities $(\widehat{w'}, \widehat{w_*'}, \widehat{v'}, \widehat{v_*'}, \widehat{v^{\prime\prime}}, \widehat{w^{\prime\prime}})$ take the form:
		\[
		\begin{aligned}
			\widehat{w'}&=\widehat{w}+\omega(\omega\cdot (\widehat{w_*}-\widehat{w})),&
			\widehat{w_*'}&=\widehat{w_*}-\omega(\omega\cdot (\widehat{w_*}-\widehat{w})),\\
			\widehat{v'}&=\widehat{v}+\omega(\omega\cdot (\widehat{v_*}-\widehat{v})),&
			\widehat{v_*'}&=\widehat{v_*}-\omega(\omega\cdot (\widehat{v_*}-\widehat{v})),\\
			\widehat{v^{\prime\prime}}&=\widehat{v}-\frac{2m_g}{m_g+m_p}\omega\left(\omega\cdot \left(\widehat{v}-\frac{V_g}{V_p}\widehat{w}\right)\right), &\widehat{w^{\prime\prime}}&=\widehat{w}+\frac{2m_p}{m_g+m_p}\omega\left(\omega\cdot \left(\frac{V_p}{V_g}\widehat{v}-\widehat{w}\right)\right).
		\end{aligned}
		\]
		
		Next, we define a time variable adapted to the motion of the typical particle in the slower species (the dust particles or droplets):
		\[
		\widehat{t}=\frac{tV_p}{L}.
		\]
		Let $N_p$ and $N_g$ denote the number densities (particles per unit volume) of the dust particles and gas molecules, respectively. We introduce the dimensionless distribution functions for each species:
		\[
		\widehat{F}(\widehat{t},\widehat{x},\widehat{v})=\frac{V_p^3}{N_p}F(t,x,v),\quad \widehat{f}(\widehat{t},\widehat{x},\widehat{w})=\frac{V_g^3}{N_g}f(t,x,w).
		\]
		
		To non-dimensionalize the collision integrals, we denote by $S_{gg}$, $S_{pp}$, and $S_{pg}$ the average molecular, particle--particle, and particle--gas cross-sections, respectively. The dimensionless differential cross-sections are given by:
		\[
		\begin{aligned}
			\widehat{\sigma_{gg}}(|\widehat{w}|,|\cos\widehat{(\widehat{w},\omega)}|)&=\frac{1}{S_{gg}}\sigma_{gg}(V_g|\widehat{w}|,|\cos\widehat{(\widehat{w},\omega)}|),\\
			\widehat{\sigma_{pp}}(|\widehat{v}|,|\cos\widehat{(\widehat{v},\omega)}|)&=\frac{1}{S_{pp}}\sigma_{pp}(V_p|\widehat{v}|,|\cos\widehat{(\widehat{v},\omega)}|),\\
			\widehat{\sigma_{pg}}(|\widehat{z}|,|\cos\widehat{(\widehat{z},\omega)}|)&=\frac{1}{S_{pg}}\sigma_{pg}(V_g|\widehat{z}|,|\cos\widehat{(\widehat{z},\omega)}|).
		\end{aligned}
		\]
		Accordingly, the dimensionless collision kernels are defined as:
		\begin{align*}
			\widehat{\mathfrak{c}}(\widehat{w}-\widehat{w_*},\omega)&=|\widehat{w}-\widehat{w_*}|\widehat{\sigma_{gg}}(|\widehat{w}-\widehat{w_*}|,|\cos\widehat{(\widehat{w}-\widehat{w_*},\omega)}|),\\
			\widehat{\mathfrak{d}}(\widehat{v}-\widehat{v_*},\omega)&=|\widehat{v}-\widehat{v_*}| \widehat{\sigma_{pp}}(|\widehat{v}-\widehat{v_*}|,|\cos\widehat{(\widehat{v}-\widehat{v_*},\omega)}|),\\
			\widehat{\mathfrak{b}}\left(\frac{V_p}{V_g}\widehat{v}-\widehat{w},\omega\right)&=\left|\frac{V_p}{V_g}\widehat{v}-\widehat{w}\right|\widehat{\sigma_{pg}}\left(\left|\frac{V_p}{V_g}\widehat{v}-\widehat{w}\right|,\left|\cos\widehat{\left(\frac{V_p}{V_g}\widehat{v}-\widehat{w},\omega\right)}\right|\right),
		\end{align*}
		which lead to the dimensionless collision integrals:
		\begin{align*}
			\widehat{\mathcal Q}(\widehat{f},\widehat{g})(\widehat{w})&= \int_{\mathbb{R}^3\times\mathbb{S}^2} 	\widehat{\mathfrak{c}}(\widehat{w}-\widehat{w_*},\omega) \{\widehat{f}(\widehat{w'})\widehat{g}(\widehat{w_*'})-\widehat{f}(\widehat{w})\widehat{g}(\widehat{w_*})\}\mathrm{d}\omega\mathrm{d}\widehat{w_*},\\
			\widehat{\mathcal B}(\widehat{F},\widehat{G})(\widehat{v})&= \int_{\mathbb{R}^3\times\mathbb{S}^2}\widehat{\mathfrak{d}}(\widehat{v}-\widehat{v_*},\omega) \{\widehat{F}(\widehat{v'})\widehat{G}(\widehat{v_*'})-\widehat{F}(\widehat{v})\widehat{G}(\widehat{v_*})\}\mathrm{d}\omega\mathrm{d}\widehat{v_*},\\
			\widehat{\mathcal D}(\widehat{F},\widehat{f})(\widehat{v}) & =\int_{\mathbb{R}^3\times\mathbb{S}^2}\widehat{\mathfrak{b}}\left(\frac{V_p}{V_g}\widehat{v}-\widehat{w},\omega\right)
			\{\widehat{F}(\widehat{v^{\prime\prime}})\widehat{f}(\widehat{w^{\prime\prime}})-\widehat{F}(\widehat{v})\widehat{f}(\widehat{w})\}\mathrm{d}\omega\mathrm{d}\widehat{w},\\
			\widehat{\mathcal R}(\widehat{f},\widehat{F})(\widehat{w}) & =\int_{\mathbb{R}^3\times\mathbb{S}^2}\widehat{\mathfrak{b}}\left(\frac{V_p}{V_g}\widehat{v}-\widehat{w},\omega\right)
			\{\widehat{F}(\widehat{v^{\prime\prime}})\widehat{f}(\widehat{w^{\prime\prime}})-\widehat{F}(\widehat{v})\widehat{f}(\widehat{w})\}\mathrm{d}\omega\mathrm{d}\widehat{v}.
		\end{align*}
		
		Collecting the above definitions, we obtain the dimensionless multicomponent Boltzmann system:
		\begin{equation*}
			\left\{\begin{aligned}
				& \partial_{\widehat{t}} \widehat{F}+\widehat{v} \cdot \nabla_{\widehat{x}} \widehat{F}=N_gS_{pg}L\frac{V_g}{V_p} \widehat{\mathcal D}(\widehat{F},\widehat{f})+N_pS_{pp}L\widehat{\mathcal B}(\widehat{F},\widehat{F}), \\
				& \partial_{\widehat{t}}\widehat{f}+\frac{V_g}{V_p}\widehat{w} \cdot \nabla_{\widehat{x}}\widehat{f}=N_pS_{pg}L\frac{V_g}{V_p}\widehat{\mathcal R}(\widehat{f},\widehat{F})+N_gS_{gg}L\frac{V_g}{V_p}\widehat{\mathcal Q}(\widehat{f},\widehat{f}).
			\end{aligned} \right.
		\end{equation*}
		
		Following \cite{MR3668954}, we assume throughout that
		\[
		N_pS_{pp}L\ll 1,
		\]
		meaning that self-collisions among dust particles/droplets, $N_pS_{pp}L\widehat{\mathcal B}(\widehat{F},\widehat{F})$, are formally negligible. Moreover, because dust particles are typically much larger than gas molecules and have a lower thermal speed, we introduce two key dimensionless scaling parameters: the mass ratio
		\begin{equation}\label{DefParm1}
			\eta = \frac{m_g}{m_p}\in [0,1],
		\end{equation}
		which quantifies the disparity between the gas molecular mass $m_g$ and the particle mass $m_p$, and the thermal speed ratio
		\begin{equation}\label{DefParm2}
			\varepsilon = \frac{V_p}{V_g}\in [0,1].
		\end{equation}
		In aerosol dynamics, the number density of dispersed particles is far lower than that of gas molecules. We thus assume the number density ratio satisfies
		$\frac{N_p}{N_g} = \eta \in [0,1]$. This scaling ensures that the mass densities of the gas phase and the particle phase remain of the same order of magnitude. Furthermore, we adopt the scaling relations
		\begin{equation}\label{DefParm4}
			N_p S_{pg}L=\varepsilon,\quad N_gS_{gg}L=\frac{1}{\varepsilon}.
		\end{equation}
		Combining relations \eqref{DefParm1}, \eqref{DefParm2}, and \eqref{DefParm4} yields the prefactors:
		\[
		\begin{aligned}
			N_gS_{pg}L\frac{V_g}{V_p}&=\frac{N_g}{N_p}(N_pS_{pg}L)\frac{V_g}{V_p}=\frac{1}{\eta},\\
			N_pS_{pg}L\frac{V_g}{V_p}&=1,\\
			N_gS_{gg}L\frac{V_g}{V_p}&=\frac{1}{\varepsilon^2}.
		\end{aligned}
		\]
		
		Omitting the caret notation $(\,\widehat{\cdot}\,)$ for brevity, the rescaled distribution functions $F_{\varepsilon,\eta}(t,x,v)$ and $f_{\varepsilon,\eta}(t,x,w)$ satisfy the nondimensionalized system:
		\begin{equation}\label{Boltz.sys.}
			\left\{\begin{aligned}
				\partial_t F_{\varepsilon,\eta} + v\cdot \nabla_x F_{\varepsilon,\eta}&=\frac{1}{\eta} \mathcal D(F_{\varepsilon,\eta},f_{\varepsilon,\eta}), \\
				\partial_t f_{\varepsilon,\eta} +\frac{1}{\varepsilon}w\cdot \nabla_x f_{\varepsilon,\eta}&=\mathcal R(f_{\varepsilon,\eta},F_{\varepsilon,\eta})+\frac{1}{\varepsilon^2}\mathcal Q(f_{\varepsilon,\eta},f_{\varepsilon,\eta}),
			\end{aligned} \right.
		\end{equation}
		subject to initial data
		\begin{equation*}
			F_{\varepsilon,\eta}\vert_{t=0}=F_{\varepsilon,\eta,0}(x,v),\quad f_{\varepsilon,\eta}\vert_{t=0}=f_{\varepsilon,\eta,0}(x,w).
		\end{equation*}
		
		For simplicity, we restrict our attention to elastic hard-sphere interactions among gas molecules as well as between molecules and particles. The collision kernels $\mathfrak{c}$, $\mathfrak{d}$, and $\mathfrak{b}$ thus reduce to:
		\begin{align*}
			\mathfrak{c}(w-w_*,\omega)&=|(w-w_*)\cdot\omega|,\\
			\mathfrak{d}(v-v_*,\omega)&=|(v-v_*)\cdot\omega|,\\
			\mathfrak{b}\left(\varepsilon v-w,\omega\right)&=|(\varepsilon v-w)\cdot\omega|.
		\end{align*}
		Under this hard-sphere assumption, the collision operators in \eqref{Boltz.sys.} take the explicit form:
		\begin{subequations}
			\begin{align}
				\mathcal Q(f,g)(w)&= \int_{\mathbb{R}^3\times\mathbb{S}^2} |(w-w_*)\cdot \omega| \{f(w')g(w_*')-f(w)g(w_*)\}\mathrm{d}\omega\mathrm{d}w_*,\label{Q}\\
				\mathcal D(F,f)(v) & =\int_{\mathbb{R}^3\times\mathbb{S}^2}|(\varepsilon v-w)\cdot{\omega}|\{F(v^{\prime\prime})f(w^{\prime\prime})-F(v)f(w)\}\mathrm{d}\omega\mathrm{d}w, \label{D,R,defa}\\
				\mathcal R(f,F)(w) & =\int_{\mathbb{R}^3\times\mathbb{S}^2}|(\varepsilon v-w)\cdot{\omega}| \{F(v^{\prime\prime})f(w^{\prime\prime})-F(v)f(w)\}\mathrm{d}\omega\mathrm{d}v,\label{D,R,defb}
			\end{align}
		\end{subequations}
		where $\mathbb{S}^2=\{\omega\in \mathbb{R}^3 \mid |\omega|=1\}$.
		The post-collision velocities $w'$ and $w_*'$ for inter-molecular collisions are given by
		\[
		\begin{aligned}
			w'=w+\omega(\omega\cdot (w_*-w)),\quad
			w_*'=w_*-\omega(\omega\cdot (w_*-w)),
		\end{aligned}
		\]
		and the post-collision velocities $v^{\prime\prime}$ and $w^{\prime\prime}$ for molecule--particle collisions satisfy:
		\begin{subequations}
			\begin{align}
				v^{\prime\prime}&=v-\frac{2\eta}{1+\eta}\omega\left(\omega\cdot \left(v-\frac{1}{\varepsilon}w\right)\right),\label{velocitya}\\
				w^{\prime\prime}&=w+\frac{2}{1+\eta}\omega(\omega\cdot (\varepsilon v-w)).\label{velocityb}
			\end{align}
		\end{subequations}
		
		Inter-molecular collisions are elastic and thus conserve mass, momentum, and energy in the standard sense (see \cite{MR1379589}). Collisions between gas molecules and particles are likewise elastic, preserving species identities. Consequently, the collision integrals $\mathcal D(F,f)$ and $\mathcal R(f,F)$ independently satisfy mass conservation:
		\begin{subequations}
			\begin{equation}\label{Jonit.Cons.M0}
				\int_{\mathbb{R}^3}\mathcal D(F,f)(v)\mathrm{d}v=\int_{\mathbb{R}^3}\mathcal R(f,F)(w)\mathrm{d}w=0,
			\end{equation}
			and jointly satisfy momentum and energy conservation:
			\begin{align}
				&\varepsilon\int_{\mathbb{R}^3}\mathcal D(F,f)(v) v\mathrm{d}v+\eta\int_{\mathbb{R}^3}\mathcal R(f,F)(w)w\mathrm{d}w=0, \label{Jonit.Cons.M,E.a}\\
				&\varepsilon^2\int_{\mathbb{R}^3}\mathcal D(F,f)(v)| v|^2\mathrm{d}v+\eta\int_{\mathbb{R}^3}\mathcal R(f,F)(w)|w|^2\mathrm{d}w=0 \label{Jonit.Cons.M,E.b}.
			\end{align}
		\end{subequations}
		Properties \eqref{Jonit.Cons.M,E.a} and \eqref{Jonit.Cons.M,E.b} follow directly from the micro-level conservation laws implied by \eqref{velocitya}--\eqref{velocityb}:
		\begin{align}
			\varepsilon v+\eta w&=\varepsilon v''+\eta w'',\label{vela}\\
			\varepsilon^2 |v|^2+\eta |w|^2&=\varepsilon^2 |v''|^2+\eta |w''|^2\label{velb}.
		\end{align}

		\subsection{The Fluctuation Equations}
		In \cite[Sec.~4.1, pp.~1723--1725]{MR3668954}, Bernard et al.\ assumed that the scaling parameters $\varepsilon$ and $\eta$ satisfy the Vlasov--Navier--Stokes scaling regime:
		\[
		\varepsilon \rightarrow 0 \quad \text{and}\quad \frac{\eta}{\varepsilon^2}\rightarrow 0.
		\]
		Let $f_{\varepsilon,\eta}=\mu(1+\varepsilon g_{\varepsilon,\eta})$, where $\mu$ denotes the standard global Maxwellian:
		\begin{equation}\label{gloal-max}
			\mu(w)=\Big(\frac{1}{2\pi}\Big)^{\frac{3}{2}}\exp\left\{-\frac{|w|^2}{2}\right\}.
		\end{equation}
		Under the assumptions that $F_{\varepsilon,\eta}\longrightarrow F$ weak-$*$ in $L^\infty_{\mathrm{loc}}$ and $g_{\varepsilon,\eta}\longrightarrow g$ weakly in $L^2_{\mathrm{loc}}$, there exist $L^\infty$ scalar functions $\rho\equiv\rho(t,x)\in \mathbb{R}$ and $\theta \equiv\theta(t,x)\in \mathbb{R}$, along with an $L^\infty$ vector field $u\equiv u(t,x)\in \mathbb{R}^3$, such that
		\[
		g(t,x,w)=\rho(t,x)+u(t,x)\cdot w+\frac{1}{2}\theta(t,x)(|w|^2-3), \quad \text{a.e.},
		\]
		and the tuple $(F, \rho, u, \theta)$ satisfies the incompressible Vlasov--Navier--Stokes system \eqref{V-N-S*} with transport coefficients
		\begin{equation}\label{kappa nu}
			\kappa=\frac{2\pi}{3}\int_{\mathbb{R}^3}|w|^3\mu(w)\mathrm{d}w,\quad
			\widetilde{\nu}=\frac{1}{10} \langle\sqrt{\mu}A,\sqrt{\mu}\widehat{A} \rangle_{L^2_w},\quad
			\widetilde{\kappa}=\frac{2}{15} \langle\sqrt{\mu}B,\sqrt{\mu}\widehat{B} \rangle_{L^2_w}.
		\end{equation}
		Here, the non-equilibrium functions $A$ and $B$ are given by
		\begin{equation}\label{AB}
			A:=w\otimes w-\frac{|w|^2}{3}\mathbb{I}_3,\quad B:=w \left(\frac{|w|^2}{2}-\frac{5}{2}\right),
		\end{equation}
		where $\mathbb{I}_3$ denotes the $3\times 3$ identity matrix, and $\widehat{A}, \widehat{B}$ satisfy
		\begin{equation}\label{ABhat}
			\sqrt{\mu} \widehat{A}=\mathcal L^{-1}(\sqrt{\mu}A), \quad\sqrt{\mu} \widehat{B}=\mathcal L^{-1}(\sqrt{\mu}B).
		\end{equation}
		
		Unfortunately, the above derivation of the IVNS model remains purely \emph{formal} due to the lack of uniform $L^{\infty}$ bounds on $F_{\varepsilon,\eta}$ and uniform $L^2$ bounds on $g_{\varepsilon,\eta}$ in general settings. Our goal is to provide a \emph{rigorous} derivation of the IVNS system starting from the rescaled Boltzmann equations \eqref{Boltz.sys.}.
		
		It is well known (see, e.g., \cite{Cercignani-88}) that Maxwellian distributions describe the global thermodynamic equilibrium for the classical single-component Boltzmann equation. Here, we define the parameter-dependent normalized Maxwellian by
		\begin{equation}\label{mu-para}
			\mu_{\varepsilon,\eta}(v)=\Big(\frac{\varepsilon^2}{2\pi\eta}\Big)^{\frac{3}{2}}\exp\left\{-\frac{|\varepsilon v|^2}{2\eta}\right\}.
		\end{equation}
      In particular, when $\varepsilon = \eta = 1$, $\mu_{\varepsilon,\eta}$ reduces to the standard Maxwellian $\mu$ in \eqref{gloal-max}.

		To connect the kinetic dynamics with the hydrodynamic limit, we consider fluctuations of $(F_{\varepsilon,\eta},f_{\varepsilon,\eta})$ around the trivial/global equilibrium state $(0,\mu)$:
		\begin{align}
			F_{\varepsilon,\eta}(t,x,v)&=\sqrt{\mu_{\varepsilon,\eta}(v)}h_{\varepsilon,\eta}(t,x,v),\label{purturbationa}\\
			f_{\varepsilon,\eta}(t,x,w)&=\mu(w)+\varepsilon\sqrt{\mu(w)}g_{\varepsilon,\eta}(t,x,w).\label{purturbationb}
		\end{align}
		Substituting the expansions \eqref{purturbationa}--\eqref{purturbationb} into system \eqref{Boltz.sys.} yields the coupled fluctuation equations:
		\begin{align}
			&\partial_t h_{\varepsilon,\eta} + v\cdot \nabla_x h_{\varepsilon,\eta} +\frac{1}{\eta} \mathcal L_{\scriptscriptstyle \mathcal D}h_{\varepsilon,\eta}=\frac{\varepsilon}{\eta} \Gamma_{\scriptscriptstyle \mathcal D} (h_{\varepsilon,\eta},g_{\varepsilon,\eta}), \label{equ:h} \\
			&\partial_tg_{\varepsilon,\eta} +\frac{1}{\varepsilon} w\cdot \nabla_x g_{\varepsilon,\eta}+\frac{1}{{\varepsilon}^2} \mathcal{L}g_{\varepsilon,\eta}=\frac{1}{\varepsilon}\Gamma(g_{\varepsilon,\eta},g_{\varepsilon,\eta})+\frac{1}{\varepsilon}\mathcal L_{\scriptscriptstyle \mathcal R} h_{\varepsilon,\eta}+\Gamma_{\scriptscriptstyle \mathcal R} (g_{\varepsilon,\eta},h_{\varepsilon,\eta}), \label{equ:g}
		\end{align}
		supplemented with the initial data
		\begin{equation}\label{equ:hg0}
			h_{\varepsilon,\eta}(0,x,v)=h_{\varepsilon,\eta,0}(x,v),\quad g_{\varepsilon,\eta}(0,x,w)=g_{\varepsilon,\eta,0}(x,w),
		\end{equation}
		where $h_{\varepsilon,\eta,0}$ and $g_{\varepsilon,\eta,0}$ are defined as
		\begin{align*}
			h_{\varepsilon,\eta,0}(x,v)=\frac{F_{\varepsilon,\eta,0}(x,v)}{\sqrt{\mu_{\varepsilon,\eta}(v)}},\quad
			g_{\varepsilon,\eta,0}(x,w)=\frac{f_{\varepsilon,\eta,0}(x,w)-\mu(w)}{\varepsilon\sqrt{\mu(w)}}.
		\end{align*}
		
		In the gas fluctuation equation \eqref{equ:g}, the linearized Boltzmann operator $\mathcal{L}$ is given by
		\begin{equation}\label{Ope.L.}
			\mathcal L g=\nu(w) g-\mathcal K g,
		\end{equation}
		where the collision frequency $\nu(w)$ is
		\begin{equation}\label{Ope.nu.}
			\nu(w)=\int_{\mathbb{R}^3\times \mathbb{S}^2}|(w-w_*)\cdot{\omega}|\mu(w_*) \mathrm{d}\omega \mathrm{d}w_*,
		\end{equation}
		and the compact integral operator $\mathcal K$ is defined by
		\begin{equation}\label{Ope.K.}
			\begin{aligned}
				\mathcal K g&=\int_{\mathbb{R}^3\times \mathbb{S}^2}\left\{\sqrt{\mu(w)}g(w_*)-\sqrt{\mu(w_*^\prime)}g(w^\prime)-\sqrt{\mu(w^\prime)}g(w_*^\prime)\right\}\\
				&\quad\times|(w-w_*)\cdot{\omega}|\sqrt{\mu(w_*)} \mathrm{d}\omega \mathrm{d}w_*.
			\end{aligned}
		\end{equation}
		The bilinear collision operators $\Gamma(g,h)$, $\Gamma_{\scriptscriptstyle \mathcal R} (g,h)$, and $\Gamma_{\scriptscriptstyle \mathcal D} (h,g)$ appearing in \eqref{equ:h}--\eqref{equ:g} are defined by
		\begin{equation}\label{def,gamma,R,D}
			\begin{aligned}
				\Gamma(g,h)=& \frac{1}{\sqrt{\mu}}Q(\sqrt{\mu}g,\sqrt{\mu}h),\\
				\Gamma_{\scriptscriptstyle \mathcal R} (g,h)=\frac{1}{\sqrt{\mu}}\mathcal R( \sqrt{\mu}g,\sqrt{\mu_{\varepsilon,\eta}}h),&\quad
				\Gamma_{\scriptscriptstyle \mathcal D} (h,g)=\frac{1}{\sqrt{\mu_{\varepsilon,\eta}}} \mathcal D(\sqrt{\mu_{\varepsilon,\eta}}h,\sqrt{\mu}g).
			\end{aligned}
		\end{equation}
		Furthermore, the linear drag operator $\mathcal L_{\scriptscriptstyle \mathcal D}$ in \eqref{equ:h} takes the form
		\begin{equation}\label{Ope.LD.}
			\mathcal L_{\scriptscriptstyle \mathcal D} h=-\Gamma_{\scriptscriptstyle \mathcal D} (h,\sqrt{\mu})=\nu_1(v)h-\mathcal K_ { \scriptscriptstyle \mathcal D} h,
		\end{equation}
		where
		\begin{equation}\label{Ope.nu1.}
			\nu_1(v)=\int_{\mathbb{R}^3\times \mathbb{S}^2}|(\varepsilon v-w)\cdot{\omega}|\mu(w)\mathrm{d}\omega \mathrm{d}w,
		\end{equation}
		and
		\begin{equation}\label{Ope.K1.}
			\mathcal K_ { \scriptscriptstyle \mathcal D}h  =\frac{1}{\sqrt{\mu_{\varepsilon,\eta}(v)}} \int_{\mathbb{R}^3\times\mathbb{S}^2}\sqrt{\mu_{\varepsilon,\eta}(v^{\prime\prime})}h(v^{\prime\prime})\mu(w^{\prime\prime})|(w-\varepsilon v)\cdot{\omega}|\mathrm{d}\omega\mathrm{d}w.
		\end{equation}
		Finally, the linear coupling operator $\mathcal L_{\scriptscriptstyle \mathcal R}$ in \eqref{equ:g} is given by
		\begin{equation}\label{Ope.LR.}
			\mathcal L_{\scriptscriptstyle \mathcal R} h =\Gamma_{\scriptscriptstyle \mathcal R} (\sqrt{\mu},h).
		\end{equation}
		
		Since our formulation involves two distinct scaling parameters, $\varepsilon$ and $\eta$, the resulting fluctuation system \eqref{equ:h}--\eqref{equ:g} presents analytical challenges significantly beyond those found in standard single-species hydrodynamic limits:
		\begin{itemize}
			\item \textbf{Asymmetric Scaling Structures}: The particle kinetic equation \eqref{equ:h} lacks advective rescaling and involves solely inter-species collisions, whereas the gas equation \eqref{equ:g} exhibits fast advection at order $\frac{1}{\varepsilon}$ along with dominant self-collisions at order $\frac{1}{\varepsilon^2}$. The relative behavior of these two scaling parameters can induce non-trivial singular limits. Through our analysis, we establish the precise parameter regime (see Theorem~\ref{Theom.1.}) required to obtain uniform energy estimates.
			\item \textbf{Singular Linear Inter-species Coupling}: The linear source term $\frac{1}{\varepsilon}\mathcal L_{\scriptscriptstyle \mathcal R} h_{\varepsilon,\eta}$ on the right-hand side of \eqref{equ:g} introduces severe analytical difficulty. To handle this singularity, the joint conservation laws play a fundamental role, enabling us to re-express problematic inter-component coupling terms via the particle transport equation and thereby prevent potential energy losses.
		\end{itemize}
		Ultimately, the key to rigorously justifying the IVNS limit for the multicomponent Boltzmann system lies in establishing energy estimates for the fluctuation pair $(h_{\varepsilon,\eta},g_{\varepsilon,\eta})$ solving \eqref{equ:h}--\eqref{equ:hg0} that hold uniformly in both $\varepsilon$ and $\eta$.

		\subsection{Notations and Main Results}
		We first introduce the notation used throughout this paper. For simplicity, $a \sim b$ means that there exist generic constants $c_1, c_2 > 0$ such that $c_1 b \leqslant a \leqslant c_2 b$. We abbreviate $a \leqslant C b$ to $a \lesssim b$, where $C > 0$ is a constant depending only on fixed parameters and independent of $\varepsilon$ and $\eta$, whose precise value may change from line to line. 
		For $p \geqslant 1$, we define the standard Lebesgue spaces:
		\[
		L^p_x = L^p(\mathbb{R}^3_x), \quad L^p_v = L^p(\mathbb{R}^3_v), \quad L^p_{x,v} = L^p(\mathbb{R}^3_x \times \mathbb{R}^3_v).
		\]
		In particular, when $p = 2$, $\langle\cdot,\cdot\rangle_{L^2_x}$, $\langle\cdot,\cdot\rangle_{L^2_v}$, and $\langle\cdot,\cdot\rangle_{L^2_{x,v}}$ denote the inner products on the Hilbert spaces $L^2_x$, $L^2_v$, and $L^2_{x,v}$, respectively, with corresponding norms $\|\cdot\|_{L^2_x}$, $\|\cdot\|_{L^2_v}$, and $\|\cdot\|_{L^2_{x,v}}$. We also introduce the following weighted velocity norms:
		\begin{equation*}
			\begin{alignedat}{2}
				\|h\|_{\nu_1} &= \sqrt{\langle\nu_1 h,h\rangle_{L^2_{x,v}}}, &\quad |h|_{\nu_1} &= \sqrt{\langle\nu_1 h,h\rangle_{L^2_v}}, \\
				\|g\|_{\nu} &= \sqrt{\langle\nu g,g\rangle_{L^2_{x,w}}}, &\quad |g|_{\nu} &= \sqrt{\langle\nu g,g\rangle_{L^2_w}}.
			\end{alignedat}
		\end{equation*}
		
		For a fixed $t \geq 0$ and integer $N \in \mathbb{N}$, the hybrid Sobolev space $L^2_v(H^N_x) = L^2(\mathbb{R}^3_v; H^N(\mathbb{R}^3_x))$ is endowed with the norm
		\[
		\|h\|^2_{L^2_v(H^N_x)} = \sum_{|\alpha|\leqslant N} \|\partial^\alpha_x h\|^2_{L^2_{x,v}},
		\]
		where $\alpha = (\alpha_1, \alpha_2, \alpha_3)$ is a multi-index with $|\alpha| = \sum_{i=1}^{3} \alpha_i$, and $\partial_x^\alpha = \partial_{x_1}^{\alpha_1} \partial_{x_2}^{\alpha_2} \partial_{x_3}^{\alpha_3}$ with $\partial_i = \partial_{x_i}$ ($i = 1, 2, 3$).
		
		For fixed $(t,x)$, let $\mathcal{N} = \ker \mathcal{L}$ and $\mathcal{N}_1 = \ker \mathcal{L}_{\scriptscriptstyle \mathcal D}$ denote the null spaces of the linear operators $\mathcal{L}$ and $\mathcal{L}_{\scriptscriptstyle \mathcal D}$, respectively. The corresponding orthogonal projections are denoted by
		\[
		\mathbf{P} : L^2(\mathbb{R}^3_w) \to \mathcal{N}, \quad \mathbf{P}_1 : L^2(\mathbb{R}^3_v) \to \mathcal{N}_1.
		\]
		
		Fix an integer $N \geqslant 2$. We define the instantaneous energy functionals by
		\[
		\mathcal{E}^2_h(t) = \|h(t,\cdot,\cdot)\|_{L^2_v(H^N_x)}^2, \quad \mathcal{E}^2_g(t) = \|g(t,\cdot,\cdot)\|_{L^2_w(H^N_x)}^2,
		\]
		and the corresponding dissipation rate functionals by
		\begin{equation*}
			\begin{alignedat}{2}
				\mathcal{D}^2_h(t) &= \frac{1}{\eta}\sum_{|\alpha|\leqslant N} \|\partial^\alpha_x (\mathbf{I} - \mathbf{P}_1) h\|^2_{\nu_1}, &\quad 
				\mathcal{C}^2_h(t) &= \varepsilon^2 \sum_{0 < |\alpha| \leqslant N} \|\partial_x^\alpha \mathbf{P}_1 h\|^2_{L^2_{x,v}}, \\
				\mathcal{D}^2_g(t) &= \frac{1}{\varepsilon^2}\sum_{|\alpha|\leqslant N} \|\partial^\alpha_x (\mathbf{I} - \mathbf{P}) g\|^2_\nu, &\quad 
				\mathcal{C}^2_g(t) &= \sum_{0 < |\alpha| \leqslant N} \|\partial_x^\alpha \mathbf{P} g\|^2_{L^2_{x,w}}.
			\end{alignedat}
		\end{equation*}
		We write $\mathcal{E}^2(h,g) = \mathcal{E}^2_h + \mathcal{E}^2_g$, $\mathcal{C}^2(h,g) = \mathcal{C}^2_h + \mathcal{C}^2_g$, and $\mathcal{D}^2(h,g) = \mathcal{D}^2_h + \mathcal{D}^2_g$. These functionals are carefully designed to perform micro-macro decompositions and control all singular scaling terms uniformly in $\varepsilon$ and $\eta$.
		
		\medskip
		
		In the first part of this work, we establish the global existence of solutions to the perturbed Boltzmann system \eqref{equ:h}--\eqref{equ:hg0} that holds uniformly with respect to both $\varepsilon$ and $\eta$.
		
		\begin{theorem}[Global Existence]\label{Theom.1.}
			Let $C_0$ and $C_1$ be two positive constants. For any $\varepsilon \in (0,1]$ and scaling parameters satisfying $C_0 \varepsilon^3 \leqslant \eta \leqslant C_1 \varepsilon^2$, there exists a threshold $\epsilon_0 > 0$ independent of $\varepsilon$ and $\eta$ such that if
			\[
			\mathcal{E}(h_{\varepsilon,\eta,0}, g_{\varepsilon,\eta,0}) \leqslant \epsilon_0,
			\]
			then the Cauchy problem \eqref{equ:h}--\eqref{equ:hg0} admits a unique global solution $(h_{\varepsilon,\eta}, g_{\varepsilon,\eta})$ satisfying
			\[
			h_{\varepsilon,\eta} \in L^\infty([0,\infty); L^2_v(H^N_x)), \quad g_{\varepsilon,\eta} \in L^\infty([0,\infty); L^2_w(H^N_x)),
			\]
			together with the global energy estimate:
			\begin{equation}\label{Glob.Enery.1}
				\sup_{t \geqslant 0} \mathcal{E}^2(h_{\varepsilon,\eta}, g_{\varepsilon,\eta})(t) + \int_{0}^{\infty} (\mathcal{D}^2 + \mathcal{C}^2)(h_{\varepsilon,\eta}, g_{\varepsilon,\eta})(t) \,\mathrm{d}t \lesssim \epsilon_0^2.
			\end{equation}
		\end{theorem}
		
		In the second part of this work, we investigate the rigorous hydrodynamic limit toward the IVNS system \eqref{V-N-S*}.
		
		\begin{theorem}[Incompressible Vlasov--Navier--Stokes Limit]\label{Theom.2}
			Let $\epsilon_0 > 0$ be the constant given in Theorem~\ref{Theom.1.}. Assume that the scaling parameters $\varepsilon, \eta \in (0,1]$ satisfy the Vlasov--Navier--Stokes scaling condition:
			\begin{equation}\label{scaling assumption}
				O(1)\varepsilon^3 \leqslant \eta \leqslant o(1)\varepsilon^2 \quad \text{as } \varepsilon \to 0.
			\end{equation}
			Suppose further that the initial data $(h_{\varepsilon,\eta,0}, g_{\varepsilon,\eta,0})$ satisfy:
			\begin{enumerate}
				\item[\rm (i)] $\mathcal{E}(h_{\varepsilon,\eta,0}, g_{\varepsilon,\eta,0}) \leqslant \epsilon_0$.
				\item[\rm (ii)] There exists a profile $\rho_{\scriptscriptstyle F_0} \in H^N_x$ such that
				\[
				\langle F_{\varepsilon,\eta,0}, 1 \rangle_{L^2_v} \longrightarrow \rho_{\scriptscriptstyle F_0}\quad \text{strongly in } H^N_x \quad \text{as } \varepsilon \to 0.
				\]
				\item[\rm (iii)] There exist scalar fields $\rho_0, \theta_0 \in H^N_x$ and a vector field $u_0 \in H^N_x$ such that
				\[
				g_{\varepsilon,\eta,0} \longrightarrow g_0 \quad \text{strongly in } L^2_w(H^N_x) \quad \text{as } \varepsilon \to 0,
				\]
				where $g_0(x,w)$ takes the macroscopic form:
				\[
				g_0(x,w) = \left\{\rho_0(x) + u_0(x)\cdot w + \frac{\theta_0(x)}{2}(|w|^2 - 3)\right\} \sqrt{\mu(w)}.
				\]
			\end{enumerate}
			Let $(h_{\varepsilon,\eta}, g_{\varepsilon,\eta})$ be the unique family of global solutions constructed in Theorem~\ref{Theom.1.}. Then, as $\varepsilon \to 0$:
			\begin{enumerate}
				\item[\rm (i)] $F_{\varepsilon,\eta} \longrightarrow F$ in the sense of distributions, where $F = \rho_{\scriptscriptstyle F} \delta_{v=0}$ with
				\[
			\rho_{\scriptscriptstyle F} \in C(\mathbb{R}^+; H^{N-\sigma}_x) \cap L^\infty(\mathbb{R}^+; H^N_x) \quad \text{for any } \sigma \in (0,1].
				\]
				\item[\rm (ii)] $g_{\varepsilon,\eta} \longrightarrow \left\{\rho + u\cdot w + \frac{\theta}{2}(|w|^2 - 3)\right\} \sqrt{\mu}$ weakly-$*$ in $L^\infty([0,\infty); L^2_w(H^N_x))$, where
				\[
				(\rho, u, \theta) \in C(\mathbb{R}^+; H^{N-\sigma}_x) \cap L^\infty(\mathbb{R}^+; H^N_x) \quad \text{for any } \sigma \in (0,1].
				\]
				\item[\rm (iii)] The limit tuple $(F, \rho, u, \theta)$ is a weak solution (in the distributional sense) to the IVNS system \eqref{V-N-S*} subject to initial data:
				\[
				F|_{t=0} = \\rho_{\scriptscriptstyle F_0}(x)\delta_{v=0}, \quad u|_{t=0} = \mathcal{P} u_0(x), \quad \theta|_{t=0} = \frac{3}{5}\theta_0(x) - \frac{2}{5}\rho_0(x),
				\]
				where $\mathcal{P}$ denotes the Leray projection operator.
			\end{enumerate}
		\end{theorem}
		
		\begin{remark}
			Due to the omission of collisions among heavy particles, the particle phase lacks a thermalization mechanism. As a consequence, the asymptotic state of the particle distribution is not a Maxwellian equilibrium, which aligns with the physical fact that particle dynamics are governed by a Vlasov-type equation. Therefore, the gas temperature $T_g$ and particle temperature $T_p$ cannot be expected to remain of the same order of magnitude. This physical feature is precisely reflected in the Vlasov--Navier--Stokes scaling condition. Indeed, from kinetic energy relations $T_E = m_E V_E^2 / (3 k_B)$ for $E \in \{p,g\}$, one observes that
			\[
			O(1)\varepsilon \leqslant \frac{T_g}{T_p} = \left(\frac{V_g}{V_p}\right)^2 \frac{m_g}{m_p} = \frac{\eta}{\varepsilon^2} \leqslant o(1) \quad \text{as } \varepsilon \to 0,
			\]
			which enforces $T_g \ll T_p$.
		\end{remark}
		
		\begin{remark}
			This asymptotic regime can also be viewed as a small mass ratio limit ($\eta \to 0$) in the context of dust particles propagating through a rarefied gas, a setting previously analyzed rigorously in the spatially homogeneous case \cite{MR2564289}. In the present work, however, we simultaneously treat the setting where the thermal speed ratio $\varepsilon \to 0$ (which plays the role of the Knudsen number in hydrodynamic limits), making the spatio-temporal dynamics substantially more complex and subtle. We also refer to related literature on massless electron limits in kinetic theory. For instance, Herda \cite{MR3479195} investigated the massless limit for a coupled Vlasov/Vlasov--Fokker--Planck system, while Flynn and Guo \cite{MR4698662} recently analyzed the massless electron limit for the Vlasov--Poisson--Landau system as a step toward deriving two-fluid hydrodynamics.
		\end{remark}
		
		\begin{remark}
			We deliberately avoid taking velocity derivatives in our energy estimates. This choice is necessitated by the factor $\frac{\varepsilon^2}{\eta}$ inherent in the velocity derivatives of the normalized global Maxwellian $\mu_{\varepsilon,\eta}$. Specifically, differentiating $\mu_{\varepsilon,\eta}(v) = \big(\frac{\varepsilon^2}{2\pi\eta}\big)^{3/2} \exp\left\{-\frac{|\varepsilon v|^2}{2\eta}\right\}$ with respect to $v$ generates unbounded coefficients proportional to $\frac{\varepsilon^2}{\eta}$, which would destroy our uniform energy estimates. This technical constraint explains why convergence can only be achieved in the weak topology.
		\end{remark}
		
		\medskip
		
		This work provides the first rigorous derivation of the IVNS system from a two-component Boltzmann model under the joint limit of vanishing mass ratio $(\eta \to 0)$ and thermal velocity ratio $(\varepsilon \to 0)$. This result resolves a long-standing mathematical challenge and sheds light on the asymptotic behavior of multiscale kinetic systems.
		
		The core of our proof rests on the parameter scaling regime identified in Theorem~\ref{Theom.1.}:
		\[
		\varepsilon^3 \lesssim \eta \lesssim \varepsilon^2,
		\]
		which allows us to establish global energy estimates that hold uniformly in $\varepsilon$ and $\eta$. We emphasize that this parameter regime is optimal for global existence within this framework, as it ensures that all coefficients balancing nonlinear and dissipative effects remain uniformly bounded. For instance, in estimate \eqref{Ceqn4}, we encounter the combination of parameter ratios:
		\[
		\left(1 + \frac{\eta^{1/2}}{\varepsilon} + \frac{\eta}{\varepsilon} + \frac{\varepsilon^4}{\eta} + \frac{\varepsilon^2}{\eta^{1/2}} + \frac{\varepsilon^3}{\eta}\right) \leqslant C,
		\]
		where $C > 0$ is independent of $\varepsilon$ and $\eta$. Crucially, the two dominant ratios $\frac{\eta^{1/2}}{\varepsilon}$ and $\frac{\varepsilon^3}{\eta}$ prescribe the upper and lower bounds $\eta \lesssim \varepsilon^2$ and $\varepsilon^3 \lesssim \eta$, respectively. The mathematical origins of these two key ratios stem directly from the microscopic energy estimates \eqref{Mic,h,g.} and macroscopic energy estimates \eqref{Mac.bar.a}.
		
		Furthermore, a major analytical obstacle arises from the loss of classical orthogonality in the inter-species collision operator $\mathcal{R}$, namely:
		\[
		\frac{1}{\varepsilon}\mathcal{L}_{\scriptscriptstyle \mathcal R} h + \Gamma_{\scriptscriptstyle \mathcal R}(g,h) \notin \mathcal{N}^\perp.
		\]
		This loss makes it exceptionally challenging to control the nonlinear inner product
		\begin{equation}\label{Key1}
			\left\langle \frac{1}{\varepsilon}\mathcal{L}_{\scriptscriptstyle \mathcal R} h + \Gamma_{\scriptscriptstyle \mathcal R}(g,h), \mathbf{P}g \right\rangle_{L^2_{x,w}}.
		\end{equation}
		The difficulty lies in the fact that \eqref{Key1} contains low-regularity components that cannot be directly absorbed by traditional dissipation rates. For example, the term $\big\langle \Gamma_{\scriptscriptstyle \mathcal R}(\mathbf{P}g, \mathbf{P}_1 h), \mathbf{P}g \big\rangle_{L^2_{x,w}}$ cannot be matched with macroscopic dissipation rates, which involve only high-order derivatives.
		
		To overcome this, we exploit the joint conservation laws \eqref{Jonit.Cons.M0}, \eqref{Jonit.Cons.M,E.a}, and \eqref{Jonit.Cons.M,E.b} to design a regularity-improving transformation scheme that re-expresses this troublesome term in the form:
		\begin{equation}\label{Key2}
			-\left\langle (\partial_t + v \cdot \nabla_x) F_{\varepsilon,\eta}, \, b(t,x)\cdot v + c(t,x)|v|^2 \right\rangle_{L^2_{x,v}},
		\end{equation}
		where $b(t,x)$ and $c(t,x)$ are the macroscopic coefficients of $\mathbf{P}g$ defined in \eqref{Def.abc.}. The transport operator in \eqref{Key2} provides essential regularity. More specifically, by performing integration by parts on the inner product \eqref{Key2}, we shift temporal and spatial derivatives onto the macroscopic profiles $b(t,x)$ and $c(t,x)$. Subsequently, using the macroscopic conservation laws \eqref{CONSE.bar.a}--\eqref{CONSE.a.b.c3}, the troublesome inner product is transformed into a coupling between the microscopic component $(\mathbf{I} - \mathbf{P}_1)h$ and the derivatives $\partial_t \mathbf{P}g, \nabla_x \mathbf{P}g$. This reveals a key intrinsic connection between \eqref{Key1} and the dissipation products $\mathcal{D}_h \mathcal{C}_g$ or $\mathcal{D}_h \mathcal{D}_g$. In Lemma~\ref{LemmaR}, this term is denoted by $I_1$ in \eqref{key3}, and through a meticulous term-by-term analysis of $I_{1,i}$ ($i = 1, 2, 3, 4$), we achieve the required control.
		
		Finally, owing to the rescaled speed variables $v$ and $w$ in system \eqref{equ:h}--\eqref{equ:hg0}, derivatives of the collision kernel $\Gamma_{\scriptscriptstyle \mathcal D}$ with respect to $v$ exhibit severe velocity singularities as $\varepsilon \to 0$. As a result, $v$-derivatives of $F_{\varepsilon,\eta}$ and $w$-derivatives of $f_{\varepsilon,\eta}$ cannot be propagated. Consequently, uniform bounds in $\varepsilon$ and $\eta$ can only be established in weighted Sobolev spaces involving spatial derivatives alone, yielding convergence exclusively in the weak topology.
		
		\bigskip
		\noindent\textbf{Organization of the Paper.} 
		In Section~\ref{Preliminaries}, we collect preliminary estimates for the operators $\mathcal{L}$, $\mathcal{L}_{\scriptscriptstyle \mathcal D}$, $\mathcal{L}_{\scriptscriptstyle \mathcal R}$, $\Gamma_{\scriptscriptstyle \mathcal D}$, and $\Gamma_{\scriptscriptstyle \mathcal R}$. Section~\ref{U.E.G.S} is devoted to establishing uniform energy estimates and proving the global existence result (Theorem~\ref{Theom.1.}). In the final section, we present the rigorous proof of the hydrodynamic limit toward the IVNS system (Theorem~\ref{Theom.2}).
		
		\section{Preliminaries}\label{Preliminaries}
		This section is devoted to the preliminary technical estimates required for the energy bounds in subsequent sections, encompassing both microscopic and macroscopic properties. We first summarize the basic operator properties of $\mathcal{L}$, $\mathcal{L}_{\scriptscriptstyle \mathcal D}$, and $\mathcal{L}_{\scriptscriptstyle \mathcal R}$, followed by a discussion of the macroscopic conservation laws. Finally, we establish precise estimates for the nonlinear and coupling operators $\Gamma_{\scriptscriptstyle \mathcal D}$, $\Gamma_{\scriptscriptstyle \mathcal R}$, and $\mathcal{L}_{\scriptscriptstyle \mathcal R}$.
		
		\subsection{Linear Operators and Macroscopic Conservation Laws}	
		Standard results for the operator $\mathcal L$ can be found in \cite{MR1307620, MR1379589}:
		\begin{proposition}\label{PropL}
			For the operator $\mathcal L = \nu - \mathcal K$ defined by \eqref{Ope.L.}-\eqref{Ope.K.}, we have the following properties:
			\begin{itemize}
				\item[(i)] There exist two constants $C_1, C_2 > 0$ such that
				\begin{equation*}
					C_1 (1+|w|) \leqslant \nu(w) \leqslant C_2 (1+|w|).
				\end{equation*}
				
				\item[(ii)] The linearized collision operator is symmetric and non-negative. Its null space is the five-dimensional space of collision invariants
				\begin{equation}\label{span-1}
					\mathcal{N} = \mathrm{span} \Big\{ \sqrt{\mu}, \, w_i\sqrt{\mu}, \, {|w|}^2\sqrt{\mu} \Big\}, \quad 1 \leqslant i \leqslant 3,
				\end{equation}
				which leads to the macro-micro decomposition
				\begin{equation}\label{Decom-g}
					g = {\bf P}g + {\bf (I-P)}g.
				\end{equation}
				Here, ${\bf P}g$ denotes the macroscopic projection of $g$, and ${\bf (I-P)}g$ denotes the kinetic (microscopic) component of $g$ (see \cite{MR2043729, MR2095473}). By \eqref{span-1}, ${\bf P}g$ can be expanded as
				\begin{equation}\label{Def.abc.}
					{\bf P}g = \left\{ a(t,x) + \sum_{i=1}^{3} b_i(t,x)w_i + c(t,x)|w|^2 \right\} \sqrt{\mu(w)}.
				\end{equation}	
				
				\item[(iii)] The operator $\mathcal L$ is locally coercive in the sense that there exists a constant $\lambda > 0$ such that
				\begin{equation}\label{Coev-1}
					\langle \mathcal L g, g \rangle_{L^2_w} \geqslant \lambda |{\bf (I-P)}g|_\nu^2.
				\end{equation}
			\end{itemize}
		\end{proposition}	
		
		\begin{remark}
			Since velocity polynomials multiplied by $\mu(w)$ are bounded in $w$, it follows directly from \eqref{Def.abc.} that
			\[
			\|\partial^\alpha_x {\bf P}g\|_{L^2_{x,w}} \sim \|\partial^\alpha_x (a,b,c)\|_{L^2_x},
			\]
			an equivalence relation that will be used repeatedly throughout this paper.
		\end{remark}	
		
		Next, we establish the key properties of the operator $\mathcal L_{\scriptscriptstyle \mathcal D}$.
		
		\begin{proposition}
			For the operator $\mathcal L_{\scriptscriptstyle \mathcal D} = \nu_1 - \mathcal K_{\scriptscriptstyle \mathcal D}$ defined by \eqref{Ope.LD.}-\eqref{Ope.K1.}, the following properties hold:
			\begin{itemize}
				\item[(i)] There exist two constants $C_1, C_2 > 0$ such that
				\begin{equation}\label{nu1-eqiv}
					C_1 (1+\varepsilon|v|) \leqslant \nu_1(v) \leqslant C_2 (1+\varepsilon|v|).
				\end{equation}
				
				\item[(ii)] The linearized collision operator is symmetric and non-negative with the null space
				\begin{equation}\label{span-2}
					\mathcal N_1 = \mathrm{span} \left\{ \sqrt{\mu_{\varepsilon,\eta}} \right\},
				\end{equation}
				which induces the macro-micro decomposition
				\begin{equation}\label{Decom-h}
					h = {\bf P_{\scriptscriptstyle 1}}h + {\bf (I-P_{\scriptscriptstyle 1})}h.
				\end{equation}
				Here, ${\bf P_{\scriptscriptstyle 1}}h$ and ${\bf (I-P_{\scriptscriptstyle 1})}h$ represent the macroscopic and microscopic components of $h$, respectively. Furthermore, in view of \eqref{span-2}, ${\bf P_{\scriptscriptstyle 1}}h$ can be expanded as
				\begin{equation}\label{Def.bar.a.}
					{\bf P_{\scriptscriptstyle 1}}h = a_{\scriptscriptstyle 1}(t,x) \sqrt{\mu_{\varepsilon,\eta}(v)}.
				\end{equation}	
				
				\item[(iii)] The operator $\mathcal L_{\scriptscriptstyle \mathcal D}$ is locally coercive; that is, there exists a constant $\lambda_1 > 0$ such that
				\begin{equation}\label{Coev-2}
					\langle \mathcal L_{\scriptscriptstyle \mathcal D} h, h \rangle_{L^2_v} \geqslant \lambda_1 |{\bf (I-P_{\scriptscriptstyle 1})}h|_{\nu_1}^2.
				\end{equation}
			\end{itemize}
		\end{proposition}
		
		\begin{remark}
			The relation \eqref{nu1-eqiv} is obtained by following arguments similar to those for the classical Boltzmann equation \cite[Sec.~3.3.1, pp.~47--48]{MR1379589}. The proofs of the remaining properties are deferred to Appendix \ref{Appendix} (see Lemmas \ref{Ker.K.d.}--\ref{L.D.j}).
		\end{remark}
		
		\begin{remark}
			It is worth noting that in the proof of \eqref{span-2} (see Appendix \ref{Appendix}, Lemma \ref{L.D.sym}, Steps 1–2), for any smooth functions $h(v)$ and $g(w)$ decaying rapidly at infinity, we have
			\[
			\big\langle \mathcal D(\sqrt{\mu_{\varepsilon,\eta}}h, \sqrt{\mu}g), a_{\scriptscriptstyle 1}(t,x) \big\rangle_{L^2_v} = 0,
			\]
			which implies
			\[
			\Gamma_{\scriptscriptstyle \mathcal D} (h,g) \in \mathcal N_1^\perp.
			\]
		\end{remark}
		
		\begin{remark}
			Since the normalized global Maxwellians $\mu_{\varepsilon,\eta}(v)$ satisfy
			\[
			\langle \sqrt{\mu_{\varepsilon,\eta}}, \sqrt{\mu_{\varepsilon,\eta}} \rangle_{L^2_v} \sim 1,
			\]
			it immediately follows from \eqref{Def.bar.a.} that
			\[
			\|\partial^\alpha_x {\bf P_{\scriptscriptstyle 1}}h\|_{L^2_{x,v}} \sim \|\partial^\alpha_x a_{\scriptscriptstyle 1}\|_{L^2_x}.
			\]
			This equivalence relation will be used henceforth without further mention.
		\end{remark}
		
		Finally, we address the basic properties of the operator $\mathcal L_{\scriptscriptstyle \mathcal R}$.
		
		\begin{proposition}\label{LR}
			If $h = a_{\scriptscriptstyle 1}(t,x)\sqrt{\mu_{\varepsilon,\eta}(v)}$, then $\mathcal L_{\scriptscriptstyle \mathcal R}h = 0$.
		\end{proposition}
		
		\begin{proof}
			Recalling the definition of $\mathcal L_{\scriptscriptstyle \mathcal R}h$ in \eqref{Ope.LR.}, we have
			\[
			\mathcal L_{\scriptscriptstyle \mathcal R}\sqrt{\mu_{\varepsilon,\eta}} = \Gamma_{\scriptscriptstyle \mathcal R} (\sqrt{\mu}, \sqrt{\mu_{\varepsilon,\eta}}) = \frac{1}{\sqrt{\mu}} \mathcal R(\mu, \mu_{\varepsilon,\eta}).
			\]
			By \eqref{velb}, the identity
			\[
			\mu_{\varepsilon,\eta}(v'') \mu(w'') = \mu_{\varepsilon,\eta}(v) \mu(w)
			\]
			holds, which implies that $\mathcal R(\mu, \mu_{\varepsilon,\eta}) = 0$. Consequently, we obtain
			\[
			\mathcal L_{\scriptscriptstyle \mathcal R}h = 0 \quad \text{for} \quad h = a_{\scriptscriptstyle 1}(t,x)\sqrt{\mu_{\varepsilon,\eta}(v)}.
			\]
		\end{proof}
		
		
		On the one hand, $a_{\scriptscriptstyle 1}(t,x)$ satisfies the local macroscopic mass conservation law. Specifically, multiplying \eqref{equ:h} by $\sqrt{\mu_{\varepsilon,\eta}}$ and integrating over $\mathbb R^3_v$, we yield
		\begin{equation}\label{CONSE.bar.a}
			\partial_t a_{\scriptscriptstyle 1} + \nabla_x \cdot \langle v\sqrt{\mu_{\varepsilon,\eta}}, {\bf (I-P_{\scriptscriptstyle 1})}h \rangle_{L^2_v} = 0.
		\end{equation}
		
		On the other hand, the macroscopic quantities $a(t,x)$, $b(t,x) = (b_1, b_2, b_3)(t,x)$, and $c(t,x)$ satisfy local conservation laws. Multiplying equation \eqref{equ:g} by $\sqrt{\mu}$, $w_i\sqrt{\mu}$, and $|w|^2\sqrt{\mu}$ respectively, and integrating each over $\mathbb R^3_w$, we arrive at:
		\begin{align}
			&\partial_t a = \frac{1}{2\varepsilon} \nabla_x \cdot \langle |w|^2 w \sqrt\mu, {\bf (I-P)}g \rangle_{L^2_w} - \frac{1}{2} \left\langle \frac{1}{\varepsilon}\mathcal L_{\scriptscriptstyle \mathcal R} h + \Gamma_{\scriptscriptstyle \mathcal R}(g,h), |w|^2 \sqrt\mu \right\rangle_{L^2_w}, \label{CONSE.a.b.c1} \\
			&\partial_t b_i + \frac{1}{\varepsilon} \partial_i (a+5c) + \frac{1}{\varepsilon} \nabla_x \cdot \langle w w_i \sqrt\mu, {\bf (I-P)}g \rangle_{L^2_w} = \left\langle \frac{1}{\varepsilon}\mathcal L_{\scriptscriptstyle \mathcal R} h + \Gamma_{\scriptscriptstyle \mathcal R}(g,h), w_i \sqrt\mu \right\rangle_{L^2_w}, \label{CONSE.a.b.c2} \\
			&\partial_t c + \frac{1}{3\varepsilon} \nabla_x \cdot b + \frac{1}{6\varepsilon} \nabla_x \cdot \langle |w|^2 w \sqrt\mu, {\bf (I-P)}g \rangle_{L^2_w} = \frac{1}{6} \left\langle \frac{1}{\varepsilon}\mathcal L_{\scriptscriptstyle \mathcal R} h + \Gamma_{\scriptscriptstyle \mathcal R}(g,h), |w|^2 \sqrt\mu \right\rangle_{L^2_w}. \label{CONSE.a.b.c3}
		\end{align} 
		These macroscopic conservation laws will be extensively utilized throughout the remainder of this paper.

		\subsection{Estimates for the Nonlinear Terms}
		In this subsection, we establish estimates for the nonlinear terms $\frac{\varepsilon}{\eta}\Gamma_{\scriptscriptstyle \mathcal D}(h,g)$ and $\frac{1}{\varepsilon}\mathcal L_{\scriptscriptstyle \mathcal R} h + \Gamma_{\scriptscriptstyle \mathcal R}(g,h)$. We begin with the following technical proposition, whose proof is straightforward.
		
		\begin{proposition}
			Let $\mu_{\varepsilon,\eta}$ be defined as in \eqref{mu-para}, and let $k \in \mathbb{N}$. Then
			\begin{equation}\label{mu}
				\int_{\mathbb{R}^3} |v|^k \mu_{\varepsilon,\eta}(v) \,\mathrm{d}v \sim \left( \frac{\eta^{1/2}}{\varepsilon} \right)^{k}.
			\end{equation}
		\end{proposition}			
		
	Next, we address the fundamental estimates for the nonlinear collision operators $\Gamma$, $\Gamma_{\scriptscriptstyle \mathcal R}$ and $\Gamma_{\scriptscriptstyle \mathcal D}$.
		
		\begin{lemma}\label{Gamma.Q.R.D}
			The following estimates hold:
		\begin{align}
			\left\langle  \Gamma(g^*,g), g_* \right\rangle _{L^2_w} &\lesssim \|g^*\|_{L^2_w} |g|_{\nu} |g_*|_{\nu} + \|g\|_{L^2_w} |g^*|_{\nu} |g_*|_{\nu}, \label{Gamma.Q.} \\
			\left\langle  \Gamma_{\scriptscriptstyle \mathcal R}(g^*,h), g_* \right\rangle _{L^2_w} &\lesssim \|h\|_{L^2_v} |g^*|_{\nu} |g_*|_{\nu} + \|g^*\|_{L^2_w} |h|_{\nu_1} |g_*|_{\nu}, \label{Gamma.R.} \\
			\left\langle  \Gamma_{\scriptscriptstyle \mathcal D}(h^*,g), h_* \right\rangle _{L^2_v} &\lesssim \|g\|_{L^2_w} |h^*|_{\nu_1} |h_*|_{\nu_1} + \|h^*\|_{L^2_v} |g|_{\nu} |h_*|_{\nu_1}. \label{Gamma.D.}
		\end{align}
	
		\end{lemma}	
		
		\begin{proof}
			Estimate \eqref{Gamma.Q.} was originally proven in \cite{MR1908664}. We omit the detailed proof of \eqref{Gamma.D.} as it follows from an analogous argument, and instead focus on establishing \eqref{Gamma.R.}.
			
	According to relations \eqref{velocitya} and \eqref{velocityb}, the operator $\Gamma_{\scriptscriptstyle \mathcal R}(g^*,h)$ can be decomposed as follows:
		\begin{align*}
			\Gamma_{\scriptscriptstyle \mathcal R}(g^*,h) &= \frac{1}{\sqrt{\mu(w)}} \int_{\mathbb{R}^3_v \times \mathbb{S}^2} \sqrt{\mu(w'')} \, g^*(w'') \sqrt{\mu_{\varepsilon,\eta}(v'')} \, h(v'') |\varepsilon v - w| \,\mathrm{d}\omega \,\mathrm{d}v \\
			&\quad - \frac{1}{\sqrt{\mu(w)}} \int_{\mathbb{R}^3_v \times \mathbb{S}^2} \sqrt{\mu(w)} \, g^*(w) \sqrt{\mu_{\varepsilon,\eta}(v)} \, h(v) |\varepsilon v - w| \,\mathrm{d}\omega \,\mathrm{d}v \\
			&= \int_{\mathbb{R}^3_v \times \mathbb{S}^2} \sqrt{\mu_{\varepsilon,\eta}(v)} \, g^*(w'') h(v'') |\varepsilon v - w| \,\mathrm{d}\omega \,\mathrm{d}v \\
			&\quad - g^*(w) \int_{\mathbb{R}^3_v \times \mathbb{S}^2} \sqrt{\mu_{\varepsilon,\eta}(v)} \, h(v) |\varepsilon v - w| \,\mathrm{d}\omega \,\mathrm{d}v.
		\end{align*}					
		We first estimate the loss term associated with $\left\langle  \Gamma_{\scriptscriptstyle \mathcal R}(g^*,h), g_* \right\rangle _{L^2_w}$. Applying the Cauchy--Schwarz inequality yields
		\begin{align*}
			&\left| \int_{\mathbb{R}^3_w} \left[ \int_{\mathbb{R}^3_v \times \mathbb{S}^2} \sqrt{\mu_{\varepsilon,\eta}(v)} \, h(v) |\varepsilon v - w| \,\mathrm{d}\omega \,\mathrm{d}v \right] g^*(w) g_*(w) \,\mathrm{d}w \right| \\
			\lesssim& \int_{\mathbb{R}^3_w} \left[ \int_{\mathbb{R}^3_v} \mu_{\varepsilon,\eta}(v) |\varepsilon v - w|^2 \,\mathrm{d}v \right]^{1/2} \left[ \int_{\mathbb{R}^3_v} |h(v)|^2 \,\mathrm{d}v \right]^{1/2} |g^*(w) g_*(w)| \,\mathrm{d}w.
		\end{align*}	
		Under the hard-sphere interaction assumption, using \eqref{mu}, a direct calculation gives
		\begin{equation}\label{Est.nu.}
			\begin{aligned}
				\left[ \int_{\mathbb{R}^3_v} \mu_{\varepsilon,\eta}(v) |\varepsilon v - w|^2 \,\mathrm{d}v \right]^{1/2}\lesssim (\eta^{1/2}+|w|) \lesssim \nu(w)
			\end{aligned}			
		\end{equation}
		Consequently, the loss term is bounded by
		\begin{equation*}
			C \left[ \int_{\mathbb{R}^3_v} |h(v)|^2 \,\mathrm{d}v \right]^{1/2} \int_{\mathbb{R}^3_w} \nu(w) |g^*(w) g_*(w)| \,\mathrm{d}w \lesssim \|h\|_{L^2_v} |g^*|_{\nu} |g_*|_{\nu}.
		\end{equation*}

	Next, for the gain term of $\left\langle  \Gamma_{\scriptscriptstyle \mathcal R}(g^*,h), g_* \right\rangle _{L^2_w}$, applying the Cauchy--Schwarz inequality successively gives
		\begin{align*}
			&\left| \int_{\mathbb{R}^3_w} \left[ \int_{\mathbb{R}^3_v \times \mathbb{S}^2} \sqrt{\mu_{\varepsilon,\eta}(v)} \, g^*(w'') h(v'') |\varepsilon v - w| \,\mathrm{d}\omega \,\mathrm{d}v \right] g_*(w) \,\mathrm{d}w \right| \\
			\lesssim& \int_{\mathbb{R}^3_w} \left[ \int_{\mathbb{R}^3_v} \mu_{\varepsilon,\eta}(v) |\varepsilon v - w|^2 \,\mathrm{d}v \right]^{1/2} \left[ \int_{\mathbb{R}^3_v} |g^*(w'') h(v'')|^2 \,\mathrm{d}v \right]^{1/2} |g_*(w)| \,\mathrm{d}w \\
			\lesssim& \int_{\mathbb{R}^3_w} \nu(w) \left[ \int_{\mathbb{R}^3_v} |g^*(w'') h(v'')|^2 \,\mathrm{d}v \right]^{1/2} |g_*(w)| \,\mathrm{d}w \\
			\lesssim& \left[ \int_{\mathbb{R}^3_w \times \mathbb{R}^3_v} \nu(w) |g^*(w'') h(v'')|^2 \,\mathrm{d}v \,\mathrm{d}w \right]^{1/2} \left[ \int_{\mathbb{R}^3_w} \nu(w) |g_*(w)|^2 \,\mathrm{d}w \right]^{1/2} \\
			\lesssim& \left[ \int_{\mathbb{R}^3_w \times \mathbb{R}^3_v} \big\{ \nu(w'') + \nu_1(v'') \big\} |g^*(w'') h(v'')|^2 \,\mathrm{d}v \,\mathrm{d}w \right]^{1/2} \left[ \int_{\mathbb{R}^3_w} \nu(w) |g_*(w)|^2 \,\mathrm{d}w \right]^{1/2} \\
			\lesssim& \|h\|_{L^2_v} |g^*|_{\nu} |g_*|_{\nu} + \|g^*\|_{L^2_w} |h|_{\nu_1} |g_*|_{\nu}.
		\end{align*}
	
			In the estimates above, we used the change of variables $(w,v) \to (w'',v'')$, the relation $|\varepsilon v| + |w| \lesssim |\varepsilon v''| + |w''|$, and the measure preservation $\mathrm{d}w'' \mathrm{d}v'' = \mathrm{d}w \mathrm{d}v$ (refer to Lemma \ref{Jcb.} in the Appendix for details). This concludes the proof of \eqref{Gamma.R.}.
		\end{proof}	
		
		\begin{remark}\label{rmk.3.1} 
			It is worth emphasizing that Lemma \ref{Gamma.Q.R.D} alone is insufficient to handle two critical types of terms, primarily due to the weaker dissipation associated with macroscopic terms under low regularity:
			\begin{itemize}
				\item[(i)] For the first class of problematic terms, such as $\frac{\varepsilon}{\eta} \left\langle  \Gamma_{\scriptscriptstyle \mathcal D}({\bf P_{\scriptscriptstyle 1}}h, {\bf P}g), {\bf (I-P_{\scriptscriptstyle 1})}h \right\rangle _{L^2_{x,v}}$, inequality \eqref{Gamma.D.} from Lemma \ref{Gamma.Q.R.D} yields:
				\begin{equation*}
					\frac{\varepsilon}{\eta} \left\langle  \Gamma_{\scriptscriptstyle \mathcal D}({\bf P_{\scriptscriptstyle 1}}h, {\bf P}g), {\bf (I-P_{\scriptscriptstyle 1})}h \right\rangle _{L^2_{x,v}} \lesssim \frac{\varepsilon}{\eta^{1/2}} \mathcal{E}_h \mathcal{C}_g \mathcal{D}_{h}.
				\end{equation*}
				The singular factor $\frac{\varepsilon}{\eta^{1/2}}$ would ruin the uniform energy estimates. Thus, a refined analysis of $\left\langle  \Gamma_{\scriptscriptstyle \mathcal D}(h,g), h \right\rangle _{L^2_v}$ is required.
				
				\item[(ii)] For the second class of delicate terms, such as $\left\langle  \Gamma_{\scriptscriptstyle \mathcal R}({\bf P}g, {\bf P_{\scriptscriptstyle 1}}h), {\bf P}g \right\rangle _{L^2_{x,w}}$, notice that $\|{\bf P}g\|_{\nu}$, $\|{\bf P_{\scriptscriptstyle 1}}h\|_{\nu_1}$, and $\|{\bf P}g\|_{\nu}$ are no longer controlled by the dissipation rate functionals $\mathcal{C}_{g}$, $\mathcal{C}_h$, and $\mathcal{C}_{g}$, respectively. Consequently, \eqref{Gamma.R.} fails to bound $\left\langle  \Gamma_{\scriptscriptstyle \mathcal R}({\bf P}g, {\bf P_{\scriptscriptstyle 1}}h), {\bf P}g \right\rangle _{L^2_{x,w}}$ by $\mathcal{E}(g,h) \left( \mathcal{C}_{g} + \mathcal{C}_h\right)  \mathcal{C}_{g}$.
			\end{itemize}
		\end{remark}
		
		\begin{lemma}[Estimate for $\frac{\varepsilon}{\eta}\Gamma_{\scriptscriptstyle \mathcal D}(h,g)$]
			Let $\alpha, \alpha_1, \alpha_2 \in \mathbb{N}^3$ with $\alpha_1 + \alpha_2 = \alpha$. Then the following estimate holds:
			\begin{equation}\label{Gamma.D.1}
				\begin{aligned}
					&\left| \big\langle \Gamma_{\scriptscriptstyle \mathcal D} (\partial_x^{\alpha_1}{\bf P_{\scriptscriptstyle 1}}h, \partial_x^{\alpha_2}{\bf P}g), \partial_x^\alpha {\bf (I-P_{\scriptscriptstyle 1})}h \big\rangle_{L^2_v} \right| \\[2pt]
					&\quad \lesssim \eta^{1/2} |\partial_x^{\alpha_1} a_{\scriptscriptstyle 1}| \left[ |\partial_x^{\alpha_2} b| + |\partial_x^{\alpha_2} c| \right] |\partial_x^\alpha {\bf (I-P_{\scriptscriptstyle 1})}h|_{\nu_1}.
				\end{aligned}
			\end{equation}
		\end{lemma}		
		
		\begin{proof}
			Although estimate \eqref{Gamma.D.} can be derived via arguments similar to those in \cite[Lemma~2.3, pp.~1111--1113]{MR1908664}, we focus here on establishing \eqref{Gamma.D.1}. 
			
			Substituting the expansions \eqref{Def.abc.} and \eqref{Def.bar.a.} into $\big\langle \Gamma_{\scriptscriptstyle \mathcal D} (\partial_x^{\alpha_1}{\bf P_{\scriptscriptstyle 1}}h, \partial_x^{\alpha_2}{\bf P}g), \partial_x^\alpha {\bf (I-P_{\scriptscriptstyle 1})}h \big\rangle_{L^2_v}$, we decompose the inner product into three parts:
			\begin{equation}\label{Deqn1}
				\begin{aligned}
					&\big\langle \Gamma_{\scriptscriptstyle \mathcal D}(\partial_x^{\alpha_1}a_{\scriptscriptstyle 1}\sqrt{\mu_{\varepsilon,\eta}}, \partial_x^{\alpha_2}a\sqrt{\mu}), \partial_x^\alpha {\bf (I-P_{\scriptscriptstyle 1})}h \big\rangle_{L^2_v} \\
					&\quad + \big\langle \Gamma_{\scriptscriptstyle \mathcal D}(\partial_x^{\alpha_1}a_{\scriptscriptstyle 1}\sqrt{\mu_{\varepsilon,\eta}}, \partial_x^{\alpha_2}b \cdot w\sqrt{\mu}), \partial_x^\alpha {\bf (I-P_{\scriptscriptstyle 1})}h \big\rangle_{L^2_v} \\
					&\quad + \big\langle \Gamma_{\scriptscriptstyle \mathcal D}(\partial_x^{\alpha_1}a_{\scriptscriptstyle 1}\sqrt{\mu_{\varepsilon,\eta}}, \partial_x^{\alpha_2}c |w|^2\sqrt{\mu}), \partial_x^\alpha {\bf (I-P_{\scriptscriptstyle 1})}h \big\rangle_{L^2_v}.
				\end{aligned}
			\end{equation}	
			In view of relation \eqref{velb}, we recall the fundamental identity
			\begin{equation}\label{ENr}
				\mu_{\varepsilon,\eta}(v'') \mu(w'') = \mu_{\varepsilon,\eta}(v) \mu(w),
			\end{equation}
			which directly implies that the first term vanishes:
			\begin{equation}\label{Deqn2}
				\big\langle \Gamma_{\scriptscriptstyle \mathcal D}(\partial_x^{\alpha_1}a_{\scriptscriptstyle 1}\sqrt{\mu_{\varepsilon,\eta}}, \partial_x^{\alpha_2}a\sqrt{\mu}), \partial_x^\alpha {\bf (I-P_{\scriptscriptstyle 1})}h \big\rangle_{L^2_v} = 0.
			\end{equation}
			Next, using identity \eqref{ENr}, the second term in \eqref{Deqn1} can be rewritten as
			\begin{equation}\label{I-12}
				\begin{aligned}
					&\big\langle \Gamma_{\scriptscriptstyle \mathcal D}(\partial_x^{\alpha_1}a_{\scriptscriptstyle 1}\sqrt{\mu_{\varepsilon,\eta}}, \partial_x^{\alpha_2}b \cdot w\sqrt{\mu}), \partial_x^\alpha {\bf (I-P_{\scriptscriptstyle 1})}h \big\rangle_{L^2_v} \\
					=& \partial_x^{\alpha_1}a_{\scriptscriptstyle 1} \int_{\mathbb{R}^3} \frac{\partial_x^{\alpha}{\bf (I-P_{\scriptscriptstyle 1})}h(v)}{\sqrt{\mu_{\varepsilon,\eta}(v)}} \int_{\mathbb{R}^3 \times \mathbb{S}^2} \partial_x^{\alpha_2}b \cdot w'' \mu_{\varepsilon,\eta}(v'') \mu(w'') |(\varepsilon v - w) \cdot \omega| \,\mathrm{d}\omega \,\mathrm{d}w \,\mathrm{d}v \\
					&- \partial_x^{\alpha_1}a_{\scriptscriptstyle 1} \int_{\mathbb{R}^3} \frac{\partial_x^{\alpha}{\bf (I-P_{\scriptscriptstyle 1})}h(v)}{\sqrt{\mu_{\varepsilon,\eta}(v)}} \int_{\mathbb{R}^3 \times \mathbb{S}^2} \partial_x^{\alpha_2}b \cdot w \, \mu_{\varepsilon,\eta}(v) \mu(w) |(\varepsilon v - w) \cdot \omega| \,\mathrm{d}\omega \,\mathrm{d}w \,\mathrm{d}v \\
					=& \partial_x^{\alpha_1}a_{\scriptscriptstyle 1} \int_{\mathbb{R}^3} \partial_x^{\alpha}{\bf (I-P_{\scriptscriptstyle 1})}h(v) \cdot \sqrt{\mu_{\varepsilon,\eta}(v)} \\
					&\times \left\{ \int_{\mathbb{R}^3} \mu(w) \int_{\mathbb{S}^2} \partial_x^{\alpha_2}b \cdot (w'' - w) |(\varepsilon v - w) \cdot \omega| \,\mathrm{d}\omega \,\mathrm{d}w \right\} \mathrm{d}v.
				\end{aligned}
			\end{equation}
			Taking advantage of the identity
			\[
			w'' - w = \frac{2}{1+\eta}\omega(\varepsilon v \cdot \omega) - \frac{2}{1+\eta}\omega (w \cdot \omega),
			\]
			the inner integral over $(w, \omega)$ in \eqref{I-12} becomes
			\begin{equation}\label{I-12-1}
				\begin{aligned}
					&\int_{\mathbb{R}^3} \mu(w) \int_{\mathbb{S}^2} \partial_x^{\alpha_2}b \cdot (w'' - w) |(\varepsilon v - w) \cdot \omega| \,\mathrm{d}\omega \,\mathrm{d}w \\
					=& \frac{2}{1+\eta} \int_{\mathbb{R}^3} \mu(w) \int_{\mathbb{S}^2} \partial_x^{\alpha_2}b \cdot \omega (\varepsilon v \cdot \omega) |(\varepsilon v - w) \cdot \omega| \,\mathrm{d}\omega \,\mathrm{d}w \\
					&- \frac{2}{1+\eta} \int_{\mathbb{R}^3} \mu(w) \int_{\mathbb{S}^2} \partial_x^{\alpha_2}b \cdot \omega (w \cdot \omega) |(\varepsilon v - w) \cdot \omega| \,\mathrm{d}\omega \,\mathrm{d}w.
				\end{aligned}
			\end{equation}
			For the first term on the right-hand side of \eqref{I-12-1}, applying estimate \eqref{nu1-eqiv} yields
			\begin{equation}\label{I-12-11}
				\begin{aligned}
					&\left| \frac{2}{1+\eta} \int_{\mathbb{R}^3} \mu(w) \int_{\mathbb{S}^2} \partial_x^{\alpha_2}b \cdot \omega (\varepsilon v \cdot \omega) |(\varepsilon v - w) \cdot \omega| \,\mathrm{d}\omega \,\mathrm{d}w \right| \\
					\lesssim& |\partial_x^{\alpha_2}b| \cdot |\varepsilon v| \cdot \int_{\mathbb{R}^3} \mu(w) |\varepsilon v - w| \,\mathrm{d}w \lesssim |\partial_x^{\alpha_2}b| \cdot |\varepsilon v| \cdot \nu_1(v). 
				\end{aligned}
			\end{equation}
			For the second term on the right-hand side of \eqref{I-12-1}, we employ the triangle inequality
			\[
			\big| |(\varepsilon v - w) \cdot \omega| - |w \cdot \omega| \big| \leqslant |\varepsilon v \cdot \omega|
			\]
			together with the symmetry identity
			\[
			\frac{2}{1+\eta} \int_{\mathbb{R}^3} \mu(w) \int_{\mathbb{S}^2} \partial_x^{\alpha_2}b \cdot \omega (w \cdot \omega) |w \cdot \omega| \,\mathrm{d}\omega \,\mathrm{d}w = 0.
			\]
			This allows us to estimate the second term as
			\begin{equation}\label{I-12-12}
				\begin{aligned}
					&\left| -\frac{2}{1+\eta} \int_{\mathbb{R}^3} \mu(w) \int_{\mathbb{S}^2} \partial_x^{\alpha_2}b \cdot \omega (w \cdot \omega) |(\varepsilon v - w) \cdot \omega| \,\mathrm{d}\omega \,\mathrm{d}w \right| \\
					=& \left| -\frac{2}{1+\eta} \int_{\mathbb{R}^3} \mu(w) \int_{\mathbb{S}^2} \partial_x^{\alpha_2}b \cdot \omega (w \cdot \omega) \big[ |(\varepsilon v - w) \cdot \omega| - |w \cdot \omega| \big] \mathrm{d}\omega \,\mathrm{d}w \right| \\
				\lesssim & \int_{\mathbb{R}^3} \mu(w) \int_{\mathbb{S}^2} |\partial_x^{\alpha_2}b \cdot \omega| \cdot |w \cdot \omega| \cdot \big| |(\varepsilon v - w) \cdot \omega| - |w \cdot \omega| \big| \,\mathrm{d}\omega \,\mathrm{d}w \\
				\lesssim & |\partial_x^{\alpha_2}b| \cdot |\varepsilon v| \cdot \int_{\mathbb{R}^3} \mu(w) |w| \,\mathrm{d}w \lesssim |\partial_x^{\alpha_2}b| \cdot |\varepsilon v|.
				\end{aligned}
			\end{equation}
			Combining \eqref{I-12-1} with estimates \eqref{I-12-11} and \eqref{I-12-12}, we obtain
			\begin{equation}\label{I-12-13}
				\left| \int_{\mathbb{R}^3} \mu(w) \int_{\mathbb{S}^2} \partial_x^{\alpha_2}b \cdot (w'' - w) |(\varepsilon v - w) \cdot \omega| \,\mathrm{d}\omega \,\mathrm{d}w \right| \lesssim |\partial_x^{\alpha_2}b| \cdot |\varepsilon v| \cdot \nu_1(v).
			\end{equation}
			Inserting \eqref{I-12-13} into \eqref{I-12} gives
			\begin{equation*}
				\begin{aligned}
					&\big\langle \Gamma_{\scriptscriptstyle \mathcal D}(\partial_x^{\alpha_1}a_{\scriptscriptstyle 1}\sqrt{\mu_{\varepsilon,\eta}}, \partial_x^{\alpha_2}b \cdot w\sqrt{\mu}), \partial_x^\alpha {\bf (I-P_{\scriptscriptstyle 1})}h \big\rangle_{L^2_v} \\
					&\quad \lesssim |\partial_x^{\alpha_1}a_{\scriptscriptstyle 1}| \cdot |\partial_x^{\alpha_2}b| \int_{\mathbb{R}^3} |\partial_x^{\alpha}{\bf (I-P_{\scriptscriptstyle 1})}h(v)| \cdot \sqrt{\mu_{\varepsilon,\eta}(v)} \cdot |\varepsilon v| \cdot \nu_1(v) \,\mathrm{d}v.
				\end{aligned}
			\end{equation*}
			Applying Hölder's inequality along with \eqref{nu1-eqiv} and the moment equivalence \eqref{mu}, we arrive at
			\begin{equation}\label{Deqn3}
				\begin{aligned}
					&\big\langle \Gamma_{\scriptscriptstyle \mathcal D}(\partial_x^{\alpha_1}a_{\scriptscriptstyle 1}\sqrt{\mu_{\varepsilon,\eta}}, \partial_x^{\alpha_2}b \cdot w\sqrt{\mu}), \partial_x^\alpha {\bf (I-P_{\scriptscriptstyle 1})}h \big\rangle_{L^2_v} \\
					&\quad \lesssim |\partial_x^{\alpha_1}a_{\scriptscriptstyle 1}| \cdot |\partial_x^{\alpha_2}b| \cdot |\partial_x^{\alpha}{\bf (I-P_{\scriptscriptstyle 1})}h|_{\nu_1} \cdot \left\| \sqrt{\mu_{\varepsilon,\eta}(v)} \cdot |\varepsilon v| \cdot (1+|\varepsilon v|)^{1/2} \right\|_{L^2_v} \\
					&\quad \lesssim \eta^{1/2} |\partial_x^{\alpha_1}a_{\scriptscriptstyle 1}| \cdot |\partial_x^{\alpha_2}b| \cdot |\partial_x^{\alpha}{\bf (I-P_{\scriptscriptstyle 1})}h|_{\nu_1}.
				\end{aligned}
			\end{equation}		
			By an argument analogous to \eqref{I-12}, the third term in \eqref{Deqn1} can be expressed as
			\begin{align*}
				&\big\langle \Gamma_{\scriptscriptstyle \mathcal D}(\partial_x^{\alpha_1}a_{\scriptscriptstyle 1}\sqrt{\mu_{\varepsilon,\eta}}, \partial_x^{\alpha_2}c |w|^2\sqrt{\mu}), \partial_x^\alpha {\bf (I-P_{\scriptscriptstyle 1})}h \big\rangle_{L^2_v} \\
				=&\partial_x^{\alpha_1}a_{\scriptscriptstyle 1} \int_{\mathbb{R}^3} \partial_x^{\alpha}{\bf (I-P_{\scriptscriptstyle 1})}h(v) \cdot \sqrt{\mu_{\varepsilon,\eta}(v)} \\
				&\times \int_{\mathbb{R}^3} \mu(w) \int_{\mathbb{S}^2} \partial_x^{\alpha_2}c \left( |w''|^2 - |w|^2 \right) |(\varepsilon v - w) \cdot \omega| \,\mathrm{d}\omega \,\mathrm{d}w \,\mathrm{d}v.
			\end{align*}
			Observing the relation
			\begin{align*}
				|w''|^2 - |w|^2 &= \frac{\varepsilon^2}{\eta} \left( |v|^2 - |v''|^2 \right) \\
				&= \frac{4}{1+\eta} \left[ (\varepsilon v - w) \cdot \omega \right] (\varepsilon v \cdot \omega) - \frac{4\eta}{(1+\eta)^2} |(\varepsilon v - w) \cdot \omega|^2,
			\end{align*}
			we further write
		\begin{equation}\label{I-13}
			\begin{aligned}
				&\big\langle \Gamma_{\scriptscriptstyle \mathcal D}(\partial_x^{\alpha_1}a_{\scriptscriptstyle 1}\sqrt{\mu_{\varepsilon,\eta}}, \partial_x^{\alpha_2}c |w|^2\sqrt{\mu}), \partial_x^\alpha {\bf (I-P_{\scriptscriptstyle 1})}h \big\rangle_{L^2_v} \\
				=& \frac{4}{1+\eta} \partial_x^{\alpha_1}a_{\scriptscriptstyle 1} \int_{\mathbb{R}^3} \partial_x^{\alpha}{\bf (I-P_{\scriptscriptstyle 1})}h(v) \cdot \sqrt{\mu_{\varepsilon,\eta}(v)} \\
				&\times \int_{\mathbb{R}^3} \mu(w) \int_{\mathbb{S}^2} \partial_x^{\alpha_2}c \left[ (\varepsilon v - w) \cdot \omega \right] (\varepsilon v \cdot \omega) |(\varepsilon v - w) \cdot \omega| \,\mathrm{d}\omega \,\mathrm{d}w \,\mathrm{d}v \\
				&-\frac{4\eta}{(1+\eta)^2} \partial_x^{\alpha_1}a_{\scriptscriptstyle 1} \int_{\mathbb{R}^3} \partial_x^{\alpha}{\bf (I-P_{\scriptscriptstyle 1})}h(v) \cdot \sqrt{\mu_{\varepsilon,\eta}(v)} \\
				&\times \int_{\mathbb{R}^3} \mu(w) \int_{\mathbb{S}^2} \partial_x^{\alpha_2}c \, |(\varepsilon v - w) \cdot \omega|^3 \,\mathrm{d}\omega \,\mathrm{d}w \,\mathrm{d}v.
			\end{aligned}
		\end{equation}		
			For the first integral on the right-hand side of \eqref{I-13}, applying Hölder's inequality together with \eqref{nu1-eqiv} and \eqref{mu} yields
			\begin{equation}\label{I-13-1}
			\begin{aligned}
				&\left| \frac{4}{1+\eta} \partial_x^{\alpha_1}a_{\scriptscriptstyle 1} \int_{\mathbb{R}^3} \partial_x^{\alpha}{\bf (I-P_{\scriptscriptstyle 1})}h(v) \cdot \sqrt{\mu_{\varepsilon,\eta}(v)} \right. \\
				&\quad \left. \times \int_{\mathbb{R}^3} \mu(w) \int_{\mathbb{S}^2} \partial_x^{\alpha_2}c \left[ (\varepsilon v - w) \cdot \omega \right] (\varepsilon v \cdot \omega) |(\varepsilon v - w) \cdot \omega| \,\mathrm{d}\omega \,\mathrm{d}w \,\mathrm{d}v \right| \\
				\lesssim &|\partial_x^{\alpha_1}a_{\scriptscriptstyle 1}| \cdot |\partial_x^{\alpha_2}c| \int_{\mathbb{R}^3} |\partial_x^{\alpha}{\bf (I-P_{\scriptscriptstyle 1})}h(v)| \cdot \sqrt{\mu_{\varepsilon,\eta}(v)} \cdot |\varepsilon v| \cdot (1+|\varepsilon v|)^2 \,\mathrm{d}v \\
				\lesssim & |\partial_x^{\alpha_1}a_{\scriptscriptstyle 1}| \cdot |\partial_x^{\alpha_2}c| \cdot |\partial_x^{\alpha}{\bf (I-P_{\scriptscriptstyle 1})}h|_{\nu_1} \cdot \left\| \sqrt{\mu_{\varepsilon,\eta}(v)} \cdot |\varepsilon v| \cdot (1+|\varepsilon v|)^{3/2} \right\|_{L^2_v} \\
		 	\lesssim & \eta^{1/2} |\partial_x^{\alpha_1}a_{\scriptscriptstyle 1}| \cdot |\partial_x^{\alpha_2}c| \cdot |\partial_x^{\alpha}{\bf (I-P_{\scriptscriptstyle 1})}h|_{\nu_1}.
			\end{aligned}
		\end{equation}
			For the second integral on the right-hand side of \eqref{I-13}, a similar application of Hölder's inequality, \eqref{nu1-eqiv}, and \eqref{mu} leads to
			\begin{equation}\label{I-13-2}
				\begin{aligned}
					&\left| \frac{4\eta}{(1+\eta)^2} \partial_x^{\alpha_1}a_{\scriptscriptstyle 1} \int_{\mathbb{R}^3} \partial_x^{\alpha}{\bf (I-P_{\scriptscriptstyle 1})}h(v) \cdot \sqrt{\mu_{\varepsilon,\eta}(v)} \int_{\mathbb{R}^3} \mu(w) \right. \\
					&\quad \left. \times \int_{\mathbb{S}^2} \partial_x^{\alpha_2}c \, |(\varepsilon v - w) \cdot \omega|^3 \,\mathrm{d}\omega \,\mathrm{d}w \,\mathrm{d}v \right| \\
					&\lesssim \eta |\partial_x^{\alpha_1}a_{\scriptscriptstyle 1}| \cdot |\partial_x^{\alpha_2}c| \cdot |\partial_x^{\alpha}{\bf (I-P_{\scriptscriptstyle 1})}h|_{\nu_1}.
				\end{aligned}
			\end{equation}
			Combining \eqref{I-13} with \eqref{I-13-1} and \eqref{I-13-2}, we obtain
			\begin{equation}\label{Deqn4}
				\big\langle \Gamma_{\scriptscriptstyle \mathcal D}(\partial_x^{\alpha_1}a_{\scriptscriptstyle 1}\sqrt{\mu_{\varepsilon,\eta}}, \partial_x^{\alpha_2}c |w|^2\sqrt{\mu}), \partial_x^\alpha {\bf (I-P_{\scriptscriptstyle 1})}h \big\rangle_{L^2_v} \lesssim \eta^{1/2} |\partial_x^{\alpha_1}a_{\scriptscriptstyle 1}| \cdot |\partial_x^{\alpha_2}c| \cdot |\partial_x^{\alpha}{\bf (I-P_{\scriptscriptstyle 1})}h|_{\nu_1}.
			\end{equation}
			Finally, substituting \eqref{Deqn2}, \eqref{Deqn3}, and \eqref{Deqn4} back into \eqref{Deqn1} completes the proof of \eqref{Gamma.D.1}.
		\end{proof}
		
		Now, we establish an estimate for $\frac{\varepsilon}{\eta}\Gamma_{\scriptscriptstyle \mathcal D}(h,g)$ in terms of the instantaneous energy functionals and dissipation rates.
		
		\begin{lemma}\label{lemmaD}
			Let $h$ and $g$ be smooth functions. Suppose that $\varepsilon, \eta \in (0,1]$ and $\alpha \in \mathbb{N}^3$ with $|\alpha| \leqslant N$ and $N \geqslant 2$. Then the following estimate holds:
		\begin{equation}\label{D}
				\frac{\varepsilon}{\eta} \left\langle \partial_x^\alpha \Gamma_{\scriptscriptstyle \mathcal D}(h,g), \partial_x^\alpha h \right\rangle_{L^2_{x,v}} \lesssim \frac{\varepsilon^2}{\eta^{1/2}} {\mathcal E}_h {\mathcal D}_g {\mathcal D}_h +\varepsilon {\mathcal E}_g {\mathcal D}_h^2+ {\mathcal E}_g{\mathcal C}_h{\mathcal D}_h.
			\end{equation}
		\end{lemma}	
		
		\begin{proof}
			Fixing $\alpha$ with $|\alpha| \leqslant N$ and $N \geqslant 2$, the null property induced by the symmetry of $\partial_x^\alpha \Gamma_{\scriptscriptstyle \mathcal D}(h,g)$ yields
			\[
			\left\langle \partial_x^\alpha \Gamma_{\scriptscriptstyle \mathcal D}(h,g), \partial_x^\alpha \mathbf{P}_{\scriptscriptstyle 1}h \right\rangle_{L^2_{x,v}} = 0.
			\]
			Consequently, by applying the Leibniz rule, we obtain
			\[
			\frac{\varepsilon}{\eta} \left\langle \partial_x^\alpha \Gamma_{\scriptscriptstyle \mathcal D}(h,g), \partial_x^\alpha h \right\rangle_{L^2_{x,v}} = \frac{\varepsilon}{\eta} \sum_{\alpha_1 + \alpha_2 = \alpha} \binom{\alpha}{\alpha_1} \left\langle \Gamma_{\scriptscriptstyle \mathcal D}(\partial_x^{\alpha_1} h, \partial_x^{\alpha_2}g), \partial_x^\alpha (\mathbf{I - P}_{\scriptscriptstyle 1})h \right\rangle_{L^2_{x,v}}.
			\]
			In view of the macro-micro decompositions for $h$ and $g$ in \eqref{Decom-h} and \eqref{Decom-g}, each inner product term $\frac{\varepsilon}{\eta} \left\langle \Gamma_{\scriptscriptstyle \mathcal D}(\partial_x^{\alpha_1} h, \partial_x^{\alpha_2}g), \partial_x^\alpha (\mathbf{I - P}_{\scriptscriptstyle 1})h \right\rangle_{L^2_{x,v}}$ can be further expanded into four components:
			\begin{equation}\label{Deqn10}
				\begin{aligned}
					&\frac{\varepsilon}{\eta} \left\langle \Gamma_{\scriptscriptstyle \mathcal D}(\partial_x^{\alpha_1} \mathbf{P}_{\scriptscriptstyle 1}h, \partial_x^{\alpha_2} \mathbf{P}g), \partial_x^\alpha (\mathbf{I - P}_{\scriptscriptstyle 1})h \right\rangle_{L^2_{x,v}} \\
					&\quad + \frac{\varepsilon}{\eta} \left\langle \Gamma_{\scriptscriptstyle \mathcal D}(\partial_x^{\alpha_1} \mathbf{P}_{\scriptscriptstyle 1}h, \partial_x^{\alpha_2} (\mathbf{I - P})g), \partial_x^\alpha (\mathbf{I - P}_{\scriptscriptstyle 1})h \right\rangle_{L^2_{x,v}} \\
					&\quad + \frac{\varepsilon}{\eta} \left\langle \Gamma_{\scriptscriptstyle \mathcal D}(\partial_x^{\alpha_1} (\mathbf{I - P}_{\scriptscriptstyle 1})h, \partial_x^{\alpha_2} \mathbf{P}g), \partial_x^\alpha (\mathbf{I - P}_{\scriptscriptstyle 1})h \right\rangle_{L^2_{x,v}} \\
					&\quad + \frac{\varepsilon}{\eta} \left\langle \Gamma_{\scriptscriptstyle \mathcal D}(\partial_x^{\alpha_1} (\mathbf{I - P}_{\scriptscriptstyle 1})h, \partial_x^{\alpha_2} (\mathbf{I - P})g), \partial_x^\alpha (\mathbf{I - P}_{\scriptscriptstyle 1})h \right\rangle_{L^2_{x,v}}.
				\end{aligned}
			\end{equation}
			
			We estimate each of these four terms individually. For the first term on the right-hand side of \eqref{Deqn10}. Applying estimate \eqref{Gamma.D.1} in combination with H\"older's inequality and Sobolev's inequality over $x$---specifically, using the $L^\infty_x$-$L^2_x$-$L^2_x$ H\"older inequality together with the Sobolev inequality
				$
				\|u\|_{L^\infty_x}\lesssim \|\nabla_x u\|_{L^2_x}^{1/2}\|\nabla_x^2 u\|_{L^2_x}^{1/2}\lesssim\|\nabla_x u\|_{H^1_x}
				$
				for $|\alpha_1|=0$; using the $L^3_x$-$L^6_x$-$L^2_x$ H\"older inequality together with the Sobolev inequalities
				$
				\|u\|_{L^6_x}\lesssim\|u\|_{H^1_x},\;
				\|u\|_{L^3_x}\lesssim\|u\|_{H^1_x}^{1/2}\|u\|_{L^2_x}^{1/2}
				$
				for $1\leqslant|\alpha_1|\leqslant|\alpha|-1$; and using the $L^2_x$-$L^\infty_x$-$L^2_x$ H\"older inequality together with the Sobolev inequality
				$
				\|u\|_{L^\infty_x}\lesssim\|u\|_{H^2_x}
				$
				for $|\alpha_1|=|\alpha|$---we arrive at
				\begin{equation}\label{Deqn11}
					\frac{\varepsilon}{\eta} \left| \left\langle \Gamma_{\scriptscriptstyle \mathcal D}(\partial_x^{\alpha_1} \mathbf{P}_{\scriptscriptstyle 1}h, \partial_x^{\alpha_2}\mathbf{P}g), \partial_x^\alpha (\mathbf{I - P}_{\scriptscriptstyle 1})h \right\rangle_{L^2_{x,v}} \right| \lesssim {\mathcal E}_g{\mathcal C}_h{\mathcal D}_h.
			\end{equation}
			
			For the second term in \eqref{Deqn10}, invoking estimate \eqref{Gamma.D.} with $(h^*, g, h_*)$ replaced by $(\partial_x^{\alpha_1} \mathbf{P}_{\scriptscriptstyle 1}h, \partial_x^{\alpha_2}(\mathbf{I - P})g, \partial_x^\alpha (\mathbf{I - P}_{\scriptscriptstyle 1})h)$ gives
			\begin{equation}\label{Deqn5}
				\begin{aligned}
					&\frac{\varepsilon}{\eta} \left| \left\langle \Gamma_{\scriptscriptstyle \mathcal D}(\partial_x^{\alpha_1} \mathbf{P}_{\scriptscriptstyle 1}h, \partial_x^{\alpha_2}(\mathbf{I - P})g), \partial_x^\alpha (\mathbf{I - P}_{\scriptscriptstyle 1})h \right\rangle_{L^2_{x,v}} \right| \\
					&\quad \lesssim \frac{\varepsilon}{\eta} \int_{\mathbb{R}^3} \left( \|\partial_x^{\alpha_1} \mathbf{P}_{\scriptscriptstyle 1}h\|_{L^2_v} |\partial_x^{\alpha_2}(\mathbf{I - P})g|_{\nu} + |\partial_x^{\alpha_1} \mathbf{P}_{\scriptscriptstyle 1}h|_{\nu_1} \|\partial_x^{\alpha_2}(\mathbf{I - P})g\|_{L^2_w} \right) \\
					&\quad\quad \times |\partial_x^\alpha (\mathbf{I - P}_{\scriptscriptstyle 1})h|_{\nu_1} \,\mathrm{d}x \\
					&\quad \lesssim \frac{\varepsilon}{\eta} \int_{\mathbb{R}^3} |\partial_x^{\alpha_1} a_{\scriptscriptstyle 1}| \cdot |\partial_x^{\alpha_2}(\mathbf{I - P})g|_{\nu} |\partial_x^\alpha (\mathbf{I - P}_{\scriptscriptstyle 1})h|_{\nu_1} \,\mathrm{d}x \lesssim \frac{\varepsilon^2}{\eta^{1/2}} {\mathcal E}_h {\mathcal D}_g {\mathcal D}_h.
				\end{aligned}
			\end{equation}
			
			Similarly, replacing $(h^*, g, h_*)$ in \eqref{Gamma.D.} by $(\partial_x^{\alpha_1} (\mathbf{I - P}_{\scriptscriptstyle 1})h, \partial_x^{\alpha_2}\mathbf{P}g, \partial_x^\alpha (\mathbf{I - P}_{\scriptscriptstyle 1})h)$ and repeating the arguments in \eqref{Deqn5}, the third term in \eqref{Deqn10} is bounded by
			\begin{equation}\label{Deqn8}
				\frac{\varepsilon}{\eta} \left| \left\langle \Gamma_{\scriptscriptstyle \mathcal D}(\partial_x^{\alpha_1} (\mathbf{I - P}_{\scriptscriptstyle 1})h, \partial_x^{\alpha_2}\mathbf{P}g), \partial_x^\alpha (\mathbf{I - P}_{\scriptscriptstyle 1})h \right\rangle_{L^2_{x,v}} \right| \lesssim \varepsilon {\mathcal E}_g {\mathcal D}_h^2.
			\end{equation}
			
			Along the same lines, by replacing $(h^*, g, h_*)$ in \eqref{Gamma.D.} by $\partial_x^{\alpha_1} (\mathbf{I - P}_{\scriptscriptstyle 1})h, \partial_x^{\alpha_2}(\mathbf{I - P})g$, and $\partial_x^\alpha (\mathbf{I - P}_{\scriptscriptstyle 1})h$, the fourth term in \eqref{Deqn10} satisfies
			\begin{equation}\label{Deqn9}
				\frac{\varepsilon}{\eta} \left| \left\langle \Gamma_{\scriptscriptstyle \mathcal D}(\partial_x^{\alpha_1}(\mathbf{I - P}_{\scriptscriptstyle 1})h, \partial_x^{\alpha_2}(\mathbf{I - P})g), \partial_x^\alpha (\mathbf{I - P}_{\scriptscriptstyle 1})h \right\rangle_{L^2_{x,v}} \right| \lesssim \frac{\varepsilon^2}{\eta^{1/2}} {\mathcal E}_h {\mathcal D}_g {\mathcal D}_h + \varepsilon {\mathcal E}_g {\mathcal D}_h^2.
			\end{equation}
			
			
			Substituting \eqref{Deqn5}, \eqref{Deqn8}, \eqref{Deqn9}, and \eqref{Deqn11} back into \eqref{Deqn10}, we conclude that
		\[
		\frac{\varepsilon}{\eta} \left\langle \Gamma_{\scriptscriptstyle \mathcal D}(\partial_x^{\alpha_1} h, \partial_x^{\alpha_2}g), \partial_x^\alpha (\mathbf{I - P}_{\scriptscriptstyle 1})h \right\rangle_{L^2_{x,v}} \lesssim \frac{\varepsilon^2}{\eta^{1/2}} {\mathcal E}_h {\mathcal D}_g {\mathcal D}_h +\varepsilon {\mathcal E}_g {\mathcal D}_h^2+ {\mathcal E}_g{\mathcal C}_h{\mathcal D}_h.
		\]
			This completes the proof of estimate \eqref{D}.
		\end{proof}
		
		We  first present a lemma to the time-derivative estimates for the macroscopic components $(a_{\scriptscriptstyle 1}, a, b, c)$, which play an important role in the following estimate for the nonlinear term \( \frac{1}{\varepsilon}\mathcal L_{\scriptscriptstyle \mathcal R} h+\Gamma_{\scriptscriptstyle \mathcal R} (g,h)\).		
		
		\begin{lemma}
			Let $a_{\scriptscriptstyle 1}$ be defined by \eqref{Def.bar.a.} and let $(a,b,c)$ be defined by \eqref{Def.abc.}. Then it holds
			\begin{align}
				\|\partial_ta_{\scriptscriptstyle 1} \|_{H^{N-1}_x}&\lesssim \frac{\eta}{\varepsilon} {\mathcal D}_h,\label{Est.t.bar.a}\\
				\|\partial_t a \|_{H^{N-1}_x}&\lesssim {\mathcal D}_g
				+\frac{\eta^{1/2}}{\varepsilon} {\mathcal D}_h+{\mathcal E}_h\left (\varepsilon{\mathcal D}_g+{\mathcal C}_g\right )+\eta^{1/2}{\mathcal E}_g{\mathcal D}_h, \label{Parti.a.}\\
				\|\partial_t (b,c) \|_{H^{N-1}_x}&\lesssim {\mathcal D}_g+\frac{1}{\varepsilon}{\mathcal C}_g
				+\frac{\eta^{1/2}}{\varepsilon} {\mathcal D}_h+{\mathcal E}_h\left (\varepsilon{\mathcal D}_g+{\mathcal C}_g\right )+\eta^{1/2}{\mathcal E}_g{\mathcal D}_h.\label{Est.t.a.b.c}
			\end{align}
		\end{lemma} 		
		
		\begin{proof}
			We first prove \eqref{Est.t.bar.a}. The identity \eqref{CONSE.bar.a} gives us
			\[
			\begin{aligned}
			\|\partial_ta_{\scriptscriptstyle 1}\|_{H^{N-1}_x}
			&\leqslant \|\nabla_x \cdot\langle v\sqrt{\mu_{\varepsilon,\eta}}, (\mathbf{I - P}_{\scriptscriptstyle 1})h \rangle_{L^2_v}\|_{H^{N-1}_x}\\
			&\lesssim \frac{\eta^{1/2}}{\varepsilon}\|\nabla_x  (\mathbf{I - P}_{\scriptscriptstyle 1})h\|_{L^2_v(H^{N-1}_x)}
			\lesssim \frac{\eta}{\varepsilon}{\mathcal D}_h.
		\end{aligned}
			\]
			Next, we prove \eqref{Parti.a.}. The identity \eqref{CONSE.a.b.c1} gives us 	
			\begin{equation}\label{R-3}
				\begin{aligned}
					\|\partial_t a\|_{H^{N-1}_x}
					&\lesssim\frac{1}{\varepsilon} \|\nabla_x \cdot \langle |w|^2w\sqrt\mu,(\mathbf{I - P})g\rangle_{L^2_w}\|_{H^{N-1}_x}+\frac{1}{\varepsilon} \| \langle\mathcal L_{\scriptscriptstyle \mathcal R}h,|w|^2\sqrt\mu \rangle_{L^2_w} \|_{H^{N-1}_x}\\
					&\quad+ \|\left \langle\Gamma_{\scriptscriptstyle \mathcal R} (g,h),|w|^2\sqrt\mu\right \rangle_{L^2_w} \|_{H^{N-1}_x}.\\
				\end{aligned}
			\end{equation}
			For the first term on the RHS of \eqref{R-3}, we have from Cauchy-Schwarz's inequality that
			\begin{equation}\label{R-31}
				\begin{aligned}
					\frac{1}{\varepsilon} \|\nabla_x \cdot \langle |w|^2w\sqrt\mu,(\mathbf{I - P})g\rangle_{L^2_w}\|_{H^{N-1}_x}
					&\lesssim \frac{1}{\varepsilon}\|\nabla_x(\mathbf{I - P})g\|_{L^2_w(H^{N-1}_x)}
					\lesssim{\mathcal D}_g.
				\end{aligned}
			\end{equation}
			For the second term on the RHS of \eqref{R-3}, we apply the identity (from Proposition \ref{LR})
			$\mathcal L_{\scriptscriptstyle \mathcal R}{\bf P_{\scriptscriptstyle 1}}h=0$
			and we recall the definition of $\mathcal L_{\scriptscriptstyle \mathcal R}h$ in \eqref{Ope.LR.} to obtain
			\begin{equation}\label{R-32}
				\begin{aligned}
					&\frac{1}{\varepsilon} \| \langle\mathcal L_{\scriptscriptstyle \mathcal R}h,|w|^2\sqrt\mu \rangle_{L^2_w} \|_{H^{N-1}_x}
					=	\frac{1}{\varepsilon} \| \langle\mathcal L_{\scriptscriptstyle \mathcal R}(\mathbf{I - P}_{\scriptscriptstyle 1})h,|w|^2\sqrt\mu \rangle_{L^2_w} \|_{H^{N-1}_x}\\
					=&	\frac{1}{\varepsilon} \| \langle \Gamma_{\scriptscriptstyle \mathcal R} (\sqrt{\mu},(\mathbf{I - P}_{\scriptscriptstyle 1})h) ,|w|^2\sqrt\mu \rangle_{L^2_w} \|_{H^{N-1}_x}\lesssim\frac{1}{\varepsilon}\| |(\mathbf{I - P}_{\scriptscriptstyle 1})h|_{\nu_1}\|_{H^{N-1}_x}	\lesssim\frac{\eta^{1/2}}{\varepsilon}{\mathcal D}_h.
				\end{aligned}
			\end{equation}
			For the third term on the RHS of \eqref{R-3}, we use the macro-micro decompositions as in \eqref{Decom-g} for $g$ and \eqref{Decom-h} for $h$ to further decompose it into
			\begin{equation}\label{R-34}
				\begin{aligned}
					&\|\left \langle\Gamma_{\scriptscriptstyle \mathcal R} ({\bf P}g,{\bf P_{\scriptscriptstyle 1}}h),|w|^2\sqrt\mu\right \rangle_{L^2_w} \|_{H^{N-1}_x}
					+\|\left \langle\Gamma_{\scriptscriptstyle \mathcal R} ((\mathbf{I - P})g,{\bf P_{\scriptscriptstyle 1}}h),|w|^2\sqrt\mu\right \rangle_{L^2_w} \|_{H^{N-1}_x}\\
					&+\|\left \langle\Gamma_{\scriptscriptstyle \mathcal R} (g,(\mathbf{I - P}_{\scriptscriptstyle 1})h),|w|^2\sqrt\mu\right \rangle_{L^2_w} \|_{H^{N-1}_x}.
				\end{aligned}
			\end{equation}
			Recall the estimate \eqref{Gamma.R.} and replace $g^*$, $h$ and $g_*$ in \eqref{Gamma.R.} by $(\mathbf{I - P})g$, ${\bf P_{\scriptscriptstyle 1}}h$ and $|w|^2\sqrt\mu$ respectively to obtain
			\begin{equation}\label{R-35}
				\begin{aligned}
					&\|\left \langle\Gamma_{\scriptscriptstyle \mathcal R} ((\mathbf{I - P})g,{\bf P_{\scriptscriptstyle 1}}h),|w|^2\sqrt\mu\right \rangle_{L^2_w} \|_{H^{N-1}_x}\\
					\lesssim&\sum_{ |\alpha|\leqslant N-1} \sum_{ \alpha_1 + \alpha_2 = \alpha}\||\partial^{\alpha_1}_xa_{\scriptscriptstyle 1}|\cdot|\partial^{\alpha_2}_x(\mathbf{I - P})g|_{\nu}\|_{L^2_x}
					\lesssim \varepsilon{\mathcal E}_h{\mathcal D}_g,
				\end{aligned}
			\end{equation}
			Similar to \eqref{R-35}, we have
			\begin{equation}\label{R-39}
				\begin{aligned}
					\|\left \langle\Gamma_{\scriptscriptstyle \mathcal R} (g(\mathbf{I - P}_{\scriptscriptstyle 1})h),|w|^2\sqrt\mu\right \rangle_{L^2_w} \|_{H^{N-1}_x}	\lesssim \eta^{1/2}{\mathcal E}_g{\mathcal D}_h.
				\end{aligned}
			\end{equation}
			On the other hand, for the first term on the RHS of \eqref{R-34}, we replace $g^*$, $h$ and $g_*$ in \eqref{Gamma.R.} by ${\bf P}g$, ${\bf P_{\scriptscriptstyle 1}}h$ and $|w|^2\sqrt\mu$ respectively to obtain
			\begin{equation}\label{R-40}
				\begin{aligned}
					\|\left \langle\Gamma_{\scriptscriptstyle \mathcal R} ({\bf P}g,{\bf P_{\scriptscriptstyle 1}}h),|w|^2\sqrt\mu\right \rangle_{L^2_w} \|_{H^{N-1}_x}	&\lesssim\sum_{|\alpha|\leqslant N-1} \sum_{ \alpha_1 + \alpha_2 = \alpha}\||\partial^{\alpha_1}_x(a,b,c)|\cdot|\partial^{\alpha_2}_xa_{\scriptscriptstyle 1}|\|_{L^2_x}\\
					&\lesssim {\mathcal E}_h{\mathcal C}_g.
				\end{aligned}
			\end{equation}
			%
			We plug \eqref{R-35}, \eqref{R-39} and \eqref{R-40} into \eqref{R-34} to get
			\begin{equation}\label{R-44}
				\begin{aligned}
					\|\left \langle\Gamma_{\scriptscriptstyle \mathcal R} (g,h),|w|^2\sqrt\mu\right \rangle_{L^2_w}\lesssim {\mathcal E}_h\left (\varepsilon{\mathcal D}_g+{\mathcal C}_g\right )+\eta^{1/2}{\mathcal E}_g{\mathcal D}_h.
				\end{aligned}
			\end{equation}
			Thus putting the estimates \eqref{R-31}, \eqref{R-32} and \eqref{R-44} into \eqref{R-3}, we can establish
			\begin{equation}\label{R-45}
				\|\partial_t a \|_{H^{N-1}_x}\lesssim {\mathcal D}_g
				+\frac{\eta^{1/2}}{\varepsilon} {\mathcal D}_h+{\mathcal E}_h\left (\varepsilon{\mathcal D}_g+{\mathcal C}_g\right )+\eta^{1/2}{\mathcal E}_g{\mathcal D}_h.
			\end{equation}
			Repeating the process from \eqref{R-3} to \eqref{R-45} on \eqref{CONSE.a.b.c2} and \eqref{CONSE.a.b.c3}, we can also establish \eqref{Est.t.a.b.c}.
		\end{proof}
		
		We now establish an estimate for $\frac{1}{\varepsilon}\mathcal L_{\scriptscriptstyle \mathcal R} h + \Gamma_{\scriptscriptstyle \mathcal R} (g,h)$ in terms of the instantaneous energy functionals and dissipation rates.
		
		\begin{lemma}[Estimate for $\frac{1}{\varepsilon}\mathcal L_{\scriptscriptstyle \mathcal R} h + \Gamma_{\scriptscriptstyle \mathcal R} (g,h)$]\label{LemmaR}
			Let $h$ and $g$ be smooth functions. Let $\varepsilon, \eta \in (0,1]$ and $\alpha \in \mathbb{N}^3$ with $|\alpha| \leqslant N$ and $N \geqslant 2$. Then the following estimate holds:
			\begin{equation}\label{R}
				\begin{aligned}
					&\frac{\mathrm{d}}{\mathrm{d}t} \mathcal I^{g}_{h} \chi_{\{|\alpha|=0\}} + \left\langle \frac{1}{\varepsilon}\partial_x^\alpha\mathcal L_{\scriptscriptstyle \mathcal R} h + \partial_x^\alpha\Gamma_{\scriptscriptstyle \mathcal R} (g,h), \partial_x^\alpha g \right\rangle_{L^2_{x,w}} \\
					&\quad \lesssim \left( 1 + \frac{\eta^{1/2}}{\varepsilon} + \frac{\eta}{\varepsilon^2} \right) \left( {\mathcal C}_g + {\mathcal D}_g + {\mathcal D}_h \right) {\mathcal D}_h + \left( 1 + \frac{\eta}{\varepsilon} \right) {\mathcal E} \left\{ {\mathcal D}^2 + {\mathcal C}^2 \right\},
				\end{aligned}
			\end{equation}
			where the interactive energy functional $\mathcal I^{g}_{h}$ is defined by
			\begin{equation}\label{def.E-1}
				\mathcal I^{g}_{h} = \frac{\varepsilon}{3} \left\langle a_{\scriptscriptstyle 1}\mu_{\varepsilon,\eta}, (a+3c)|v|^2 \right\rangle_{L^2_{x,v}} + \left\langle (\mathbf{I - P}_{\scriptscriptstyle 1})h \sqrt{\mu_{\varepsilon,\eta}}, b \cdot v + \varepsilon c|v|^2 \right\rangle_{L^2_{x,v}}.
			\end{equation}
		\end{lemma}
		
		\begin{proof}
			Fixing $\alpha$, Proposition \ref{LR} together with the identity $\mathcal L_{\scriptscriptstyle \mathcal R}\mathbf{P}_{\scriptscriptstyle 1}h = 0$ implies
			\[
			\left\langle \partial_x^\alpha \mathcal L_{\scriptscriptstyle \mathcal R} \mathbf{P}_{\scriptscriptstyle 1}h, \partial_x^\alpha g \right\rangle_{L^2_{x,w}} = 0.
			\]
			Consequently, applying the Leibniz rule yields
			\begin{equation}\label{Reqn1}
				\begin{aligned}
					&\left\langle \frac{1}{\varepsilon}\partial_x^\alpha\mathcal L_{\scriptscriptstyle \mathcal R} h + \partial_x^\alpha\Gamma_{\scriptscriptstyle \mathcal R} (g,h), \partial_x^\alpha g \right\rangle_{L^2_{x,w}} \\
					&\quad = \left\langle \frac{1}{\varepsilon}\mathcal L_{\scriptscriptstyle \mathcal R}\partial_x^\alpha (\mathbf{I - P}_{\scriptscriptstyle 1})h, \partial_x^\alpha g \right\rangle_{L^2_{x,w}} + \sum_{\alpha_1 + \alpha_2 = \alpha} \binom{\alpha}{\alpha_1} \left\langle \Gamma_{\scriptscriptstyle \mathcal R} (\partial_x^{\alpha_1} g, \partial_x^{\alpha_2} h), \partial_x^\alpha g \right\rangle_{L^2_{x,w}}.
				\end{aligned}
			\end{equation}
			
			For $|\alpha| > 0$, invoking definition \eqref{Ope.LR.} for $\mathcal L_{\scriptscriptstyle \mathcal R}h$ and replacing $(g^*, h, g_*)$ in \eqref{Gamma.R.} by $\big(\sqrt{\mu}, \partial_x^\alpha (\mathbf{I - P}_{\scriptscriptstyle 1})h, \partial_x^\alpha g\big)$, we obtain
			\begin{equation}\label{Reqn2}
				\begin{aligned}
					\frac{1}{\varepsilon} \left| \left\langle \mathcal L_{\scriptscriptstyle \mathcal R}\partial_x^\alpha (\mathbf{I - P}_{\scriptscriptstyle 1})h, \partial_x^\alpha g \right\rangle_{L^2_{x,w}} \right| &= \frac{1}{\varepsilon} \left| \left\langle \Gamma_{\scriptscriptstyle \mathcal R} \big( \sqrt{\mu}, \partial_x^\alpha (\mathbf{I - P}_{\scriptscriptstyle 1})h \big), \partial_x^\alpha g \right\rangle_{L^2_{x,w}} \right| \\
					&\lesssim \frac{1}{\varepsilon} \left\langle \big| \partial_x^\alpha (\mathbf{I - P}_{\scriptscriptstyle 1})h \big|_{\nu_1}, |\partial_x^\alpha g|_{\nu} \right\rangle_{L^2_x} \\
					&\lesssim \frac{1}{\varepsilon} \left\| \partial_x^\alpha (\mathbf{I - P}_{\scriptscriptstyle 1})h \right\|_{\nu_1} \|\partial_x^\alpha g\|_{\nu} \lesssim \frac{\eta^{1/2}}{\varepsilon} {\mathcal D}_h (\varepsilon {\mathcal D}_g + {\mathcal C}_g).
				\end{aligned}
			\end{equation}
			
			For the second term on the right-hand side of \eqref{Reqn1}, a procedure similar to \eqref{R-34}--\eqref{R-44} yields
			\begin{equation}\label{Reqn3}
				\begin{aligned}
					\left| \left\langle \Gamma_{\scriptscriptstyle \mathcal R} (\partial_x^{\alpha_1} g, \partial_x^{\alpha_2} h), \partial_x^\alpha g \right\rangle_{L^2_{x,w}} \right| &\lesssim \left\{ {\mathcal E}_h (\varepsilon{\mathcal D}_g + {\mathcal C}_g) + \eta^{1/2} {\mathcal E}_g {\mathcal D}_h \right\} (\varepsilon{\mathcal D}_g + {\mathcal C}_g) \\
					&\lesssim {\mathcal E} \left\{ {\mathcal D}^2 + {\mathcal C}^2 \right\}.
				\end{aligned}
			\end{equation}
			Substituting \eqref{Reqn2} and \eqref{Reqn3} into \eqref{Reqn1}, we obtain for $|\alpha| > 0$:
			\begin{equation}\label{Reqn4}
				\left| \left\langle \frac{1}{\varepsilon}\partial_x^\alpha\mathcal L_{\scriptscriptstyle \mathcal R} h + \partial_x^\alpha\Gamma_{\scriptscriptstyle \mathcal R} (g,h), \partial_x^\alpha g \right\rangle_{L^2_{x,w}} \right| \lesssim \frac{\eta^{1/2}}{\varepsilon} {\mathcal D}_h ({\mathcal C}_g + {\mathcal D}_g) + {\mathcal E} \left\{ {\mathcal D}^2 + {\mathcal C}^2 \right\}.
			\end{equation}
			
			Next, for $|\alpha| = 0$, applying the identity $\mathcal L_{\scriptscriptstyle \mathcal R}\mathbf{P}_{\scriptscriptstyle 1}h = 0$ along with the macro-micro decomposition \eqref{Decom-g} to $g$ gives
			\begin{equation}\label{Reqn5}
				\begin{aligned}
					\left\langle \frac{1}{\varepsilon}\mathcal L_{\scriptscriptstyle \mathcal R} h + \Gamma_{\scriptscriptstyle \mathcal R} (g,h), g \right\rangle_{L^2_{x,w}} &= \left\langle \frac{1}{\varepsilon}\mathcal L_{\scriptscriptstyle \mathcal R} h + \Gamma_{\scriptscriptstyle \mathcal R} (g,h), \mathbf{P}g \right\rangle_{L^2_{x,w}} \\
					&\quad + \left\langle \frac{1}{\varepsilon}\mathcal L_{\scriptscriptstyle \mathcal R} \mathbf{P}_{\scriptscriptstyle 1}h + \Gamma_{\scriptscriptstyle \mathcal R} (g,h), (\mathbf{I - P})g \right\rangle_{L^2_{x,w}}.
				\end{aligned}
			\end{equation}
			Similar to \eqref{R-34}--\eqref{R-44}, the second inner product on the right-hand side of \eqref{Reqn5} satisfies
			\begin{equation}\label{Reqn6}
				\begin{aligned}
					\left| \left\langle \frac{1}{\varepsilon}\mathcal L_{\scriptscriptstyle \mathcal R} \mathbf{P}_{\scriptscriptstyle 1}h + \Gamma_{\scriptscriptstyle \mathcal R} (g,h), (\mathbf{I - P})g \right\rangle_{L^2_{x,w}} \right| &\lesssim \eta^{1/2} {\mathcal D}_h {\mathcal D}_g + \varepsilon \left\{ {\mathcal E}_h (\varepsilon{\mathcal D}_g + {\mathcal C}_g) + \eta^{1/2} {\mathcal E}_g {\mathcal D}_h \right\} {\mathcal D}_g \\
					&\lesssim \eta^{1/2} {\mathcal D}_h {\mathcal D}_g + {\mathcal E} \left\{ {\mathcal D}^2 + {\mathcal C}^2 \right\}.
				\end{aligned}
			\end{equation}
			
			Next, we treat the first term on the right-hand side of \eqref{Reqn5}, which requires a subtle analysis distinct from the previously estimated nonlinear terms. Recalling definitions \eqref{Ope.LR.} and \eqref{def,gamma,R,D} for $\mathcal L_{\scriptscriptstyle \mathcal R}h$ and $\Gamma_{\scriptscriptstyle \mathcal R} (g,h)$, together with conservation laws \eqref{Jonit.Cons.M0}--\eqref{Jonit.Cons.M,E.b}, we rewrite
			\begin{equation}\label{key3}
				\begin{aligned}
					I_1 &:= \left\langle \frac{1}{\varepsilon}\mathcal L_{\scriptscriptstyle \mathcal R} h + \Gamma_{\scriptscriptstyle \mathcal R} (g,h), \mathbf{P}g \right\rangle_{L^2_{x,w}} \\
					&= \left\langle \frac{1}{\varepsilon}\Gamma_{\scriptscriptstyle \mathcal R} (\sqrt{\mu} + \varepsilon g, h)(w), \{a(t,x) + b(t,x) \cdot w + c(t,x)|w|^2\} \sqrt{\mu(w)} \right\rangle_{L^2_{x,w}} \\
					&= \left\langle \frac{1}{\varepsilon}\mathcal R(f_{\varepsilon,\eta}, F_{\varepsilon,\eta})(w), a(t,x) + b(t,x) \cdot w + c(t,x)|w|^2 \right\rangle_{L^2_{x,w}} \\
					&= -\left\langle \frac{1}{\eta}\mathcal D(F_{\varepsilon,\eta}, f_{\varepsilon,\eta})(v), b(t,x) \cdot v + \varepsilon c(t,x)|v|^2 \right\rangle_{L^2_{x,v}} \\
					&= -\left\langle \{\partial_t + v \cdot \nabla_x \} F_{\varepsilon,\eta}, b(t,x) \cdot v + \varepsilon c(t,x)|v|^2 \right\rangle_{L^2_{x,v}}.
				\end{aligned}
			\end{equation}
			
			Furthermore, we expand $F_{\varepsilon,\eta}$ to obtain
			\begin{equation}\label{R-23}
				\begin{aligned}
					\{\partial_t + v \cdot \nabla_x \} F_{\varepsilon,\eta} &= \partial_t a_{\scriptscriptstyle 1}(t,x) \mu_{\varepsilon,\eta} + \big\{ \partial_t (\mathbf{I - P}_{\scriptscriptstyle 1})h \big\} \sqrt{\mu_{\varepsilon,\eta}} \\
					&\quad + \{v \cdot \nabla_x a_{\scriptscriptstyle 1}(t,x)\} \mu_{\varepsilon,\eta} + \big\{ v \cdot \nabla_x (\mathbf{I - P}_{\scriptscriptstyle 1})h \big\} \sqrt{\mu_{\varepsilon,\eta}}.
				\end{aligned}
			\end{equation}
			Substituting \eqref{R-23} into \eqref{key3}, $I_1$ is naturally decomposed into four components:
			\[
			\begin{aligned}
				I_{1,1} &= -\left\langle \{\partial_t a_{\scriptscriptstyle 1}(t,x)\} \mu_{\varepsilon,\eta}, b(t,x) \cdot v + \varepsilon c(t,x)|v|^2 \right\rangle_{L^2_{x,v}}, \\
				I_{1,2} &= -\left\langle \big\{ \partial_t (\mathbf{I - P}_{\scriptscriptstyle 1})h \big\} \sqrt{\mu_{\varepsilon,\eta}}, b(t,x) \cdot v + \varepsilon c(t,x)|v|^2 \right\rangle_{L^2_{x,v}}, \\
				I_{1,3} &= -\left\langle \{v \cdot \nabla_x a_{\scriptscriptstyle 1}(t,x)\} \mu_{\varepsilon,\eta}, b(t,x) \cdot v + \varepsilon c(t,x)|v|^2 \right\rangle_{L^2_{x,v}}, \\
				I_{1,4} &= -\left\langle \big\{ v \cdot \nabla_x (\mathbf{I - P}_{\scriptscriptstyle 1})h \big\} \sqrt{\mu_{\varepsilon,\eta}}, b(t,x) \cdot v + \varepsilon c(t,x)|v|^2 \right\rangle_{L^2_{x,v}}.
			\end{aligned}
			\]
			We estimate these four terms systematically.
			
			\smallskip
			\noindent\textit{Estimate for $I_{1,1}$.} Observing that $\big\langle \partial_t a_{\scriptscriptstyle 1}, b \cdot v \mu_{\varepsilon,\eta} \big\rangle_{L^2_{x,v}} = 0$ and applying conservation law \eqref{CONSE.bar.a}, we obtain
			\begin{equation}\label{R-24}
				\begin{aligned}
					I_{1,1} &= -\varepsilon \left\langle \partial_t a_{\scriptscriptstyle 1}, c|v|^2 \mu_{\varepsilon,\eta} \right\rangle_{L^2_{x,v}} \\
					&= \varepsilon \left\langle \nabla_x \cdot \big\langle v\sqrt{\mu_{\varepsilon,\eta}}, (\mathbf{I - P}_{\scriptscriptstyle 1})h \big\rangle_{L^2_v}, c|v|^2 \mu_{\varepsilon,\eta} \right\rangle_{L^2_{x,v}} \\
					&\lesssim \varepsilon \|v\sqrt{\mu_{\varepsilon,\eta}}\|_{L^2_v} \||v|^2\mu_{\varepsilon,\eta}\|_{L^1_v} \|(\mathbf{I - P}_{\scriptscriptstyle 1})h\|_{L^2_{x,v}} \|\nabla_x c\|_{L^2_x} \\
					&\lesssim \varepsilon \cdot \left( \frac{\eta^{1/2}}{\varepsilon} \right)^3 \|\nabla_x c\|_{L^2_x} \|(\mathbf{I - P}_{\scriptscriptstyle 1})h\|_{L^2_{x,v}} \lesssim \left( \frac{\eta}{\varepsilon} \right)^2 {\mathcal C}_g {\mathcal D}_h.
				\end{aligned}
			\end{equation}
			
			\smallskip
			\noindent\textit{Estimate for $I_{1,2}$.} Integrating by parts in time $t \in \mathbb{R}^+$ and applying Hölder's inequality together with equivalence relation \eqref{mu}, we have
			\begin{equation}\label{R-25}
				\begin{aligned}
					&\frac{\mathrm{d}}{\mathrm{d}t} \left\langle (\mathbf{I - P}_{\scriptscriptstyle 1})h \sqrt{\mu_{\varepsilon,\eta}}, b \cdot v + \varepsilon c|v|^2 \right\rangle_{L^2_{x,v}} + I_{1,2} \\
					&\quad = \left\langle (\mathbf{I - P}_{\scriptscriptstyle 1})h \sqrt{\mu_{\varepsilon,\eta}}, \partial_t b \cdot v + \varepsilon \partial_t c|v|^2 \right\rangle_{L^2_{x,v}} \\
					&\quad \lesssim \left\{ \|v\sqrt{\mu_{\varepsilon,\eta}}\|_{L^2_v} \|\partial_t b\|_{L^2_x} + \varepsilon \||v|^2\sqrt{\mu_{\varepsilon,\eta}}\|_{L^2_v} \|\partial_t c\|_{L^2_x} \right\} \|(\mathbf{I - P}_{\scriptscriptstyle 1})h\|_{L^2_{x,v}} \\
					&\quad \lesssim \left\{ \frac{\eta}{\varepsilon} \|\partial_t b\|_{L^2_x} + \frac{\eta^{3/2}}{\varepsilon} \|\partial_t c\|_{L^2_x} \right\} {\mathcal D}_h.
				\end{aligned}
			\end{equation}
			
			\smallskip
			\noindent\textit{Estimate for $I_{1,3}$.} Since $\big\langle \{v \cdot \nabla_x a_{\scriptscriptstyle 1}\} \mu_{\varepsilon,\eta}, c|v|^2 \big\rangle_{L^2_{x,v}} = 0$, integration by parts over $\mathbb{R}^3_x$ yields
			\begin{equation}\label{R-4}
				I_{1,3} = -\left\langle \{v \cdot \nabla_x a_{\scriptscriptstyle 1}\} \mu_{\varepsilon,\eta}, b \cdot v \right\rangle_{L^2_{x,v}} = \frac{1}{3} \left\langle a_{\scriptscriptstyle 1}\mu_{\varepsilon,\eta}, \nabla_x \cdot b |v|^2 \right\rangle_{L^2_{x,v}}.
			\end{equation}
			Taking linear combinations of \eqref{CONSE.a.b.c1} and \eqref{CONSE.a.b.c3}, we find $\nabla_x \cdot b = -\varepsilon \partial_t (a + 3c)$. Substituting this into \eqref{R-4} gives
			\[
			I_{1,3} = -\frac{\varepsilon}{3} \left\langle a_{\scriptscriptstyle 1}\mu_{\varepsilon,\eta}, \partial_t (a + 3c)|v|^2 \right\rangle_{L^2_{x,v}}.
			\]
			Differentiating the product with respect to $t$ then leads to
			\begin{equation}\label{R-26}
				\begin{aligned}
					\frac{\mathrm{d}}{\mathrm{d}t} \left[ \frac{\varepsilon}{3} \left\langle a_{\scriptscriptstyle 1}\mu_{\varepsilon,\eta}, (a + 3c)|v|^2 \right\rangle_{L^2_{x,v}} \right] + I_{1,3} &= \frac{\varepsilon}{3} \left\langle \partial_t a_{\scriptscriptstyle 1}\mu_{\varepsilon,\eta}, (a + 3c)|v|^2 \right\rangle_{L^2_{x,v}} \\
					&\lesssim \left( \frac{\eta}{\varepsilon} \right)^2 {\mathcal C}_g {\mathcal D}_h,
				\end{aligned}
			\end{equation}
			where the final inequality follows by analogy with \eqref{R-24}.
			
			\smallskip
			\noindent\textit{Estimate for $I_{1,4}$.} Integrating by parts over $\mathbb{R}^3_x$ and employing Hölder's inequality alongside \eqref{mu}, we obtain
			\begin{equation}\label{R-27}
				I_{1,4} = \sum_{i,j=1}^3 \left\langle v_i (\mathbf{I - P}_{\scriptscriptstyle 1})h \sqrt{\mu_{\varepsilon,\eta}}, \partial_i b_j v_j + \varepsilon \partial_i c|v|^2 \right\rangle_{L^2_{x,v}} \lesssim \left( \frac{\eta^{3/2}}{\varepsilon^2} + \frac{\eta^2}{\varepsilon^2} \right) {\mathcal C}_g {\mathcal D}_h.
			\end{equation}
			
			Combining bounds \eqref{R-24}, \eqref{R-25}, \eqref{R-26}, and \eqref{R-27}, we establish
			\begin{equation}\label{key4}
				\begin{aligned}
					&\frac{\mathrm{d}}{\mathrm{d}t} \left[ \frac{\varepsilon}{3} \left\langle a_{\scriptscriptstyle 1}\mu_{\varepsilon,\eta}, (a + 3c)|v|^2 \right\rangle_{L^2_{x,v}} + \left\langle (\mathbf{I - P}_{\scriptscriptstyle 1})h \sqrt{\mu_{\varepsilon,\eta}}, b \cdot v + \varepsilon c|v|^2 \right\rangle_{L^2_{x,v}} \right] + I_1 \\
					&\quad \lesssim \frac{\eta^{3/2}}{\varepsilon^2} {\mathcal C}_g {\mathcal D}_h + \frac{\eta}{\varepsilon} \|\partial_t (b,c)\|_{L^2_x} {\mathcal D}_h.
				\end{aligned}
			\end{equation}
			Note that the term under the time derivative in \eqref{key4} is precisely $\mathcal I^{g}_{h}$ defined in \eqref{def.E-1}. Hence,
			\[
			\frac{\mathrm{d}}{\mathrm{d}t} \mathcal I^{g}_{h} + \left\langle \frac{1}{\varepsilon}\mathcal L_{\scriptscriptstyle \mathcal R} h + \Gamma_{\scriptscriptstyle \mathcal R} (g,h), \mathbf{P}g \right\rangle_{L^2_{x,w}} \lesssim \frac{\eta^{3/2}}{\varepsilon^2} {\mathcal C}_g {\mathcal D}_h + \frac{\eta}{\varepsilon} \|\partial_t (b,c)\|_{L^2_x} {\mathcal D}_h.
			\]
			Inserting estimate \eqref{Est.t.a.b.c} into the above bound gives
			\begin{equation}\label{R-28}
				\frac{\mathrm{d}}{\mathrm{d}t} \mathcal I^{g}_{h} + \left\langle \frac{1}{\varepsilon}\mathcal L_{\scriptscriptstyle \mathcal R} h + \Gamma_{\scriptscriptstyle \mathcal R} (g,h), \mathbf{P}g \right\rangle_{L^2_{x,w}} \lesssim \frac{\eta}{\varepsilon^2} \left( {\mathcal C}_g + {\mathcal D}_g + {\mathcal D}_h \right) {\mathcal D}_h + \frac{\eta}{\varepsilon} {\mathcal E} \left\{ {\mathcal D}^2 + {\mathcal C}^2 \right\}.
			\end{equation}
			Consequently, combining \eqref{Reqn5} with \eqref{Reqn6} and \eqref{R-28} yields
			\begin{equation}\label{R-29}
				\begin{aligned}
					&\frac{\mathrm{d}}{\mathrm{d}t} \mathcal I^{g}_{h} + \left\langle \frac{1}{\varepsilon}\mathcal L_{\scriptscriptstyle \mathcal R} h + \Gamma_{\scriptscriptstyle \mathcal R} (g,h), g \right\rangle_{L^2_{x,w}} \\
					&\quad \lesssim \left( 1 + \frac{\eta}{\varepsilon^2} \right) \left( {\mathcal C}_g + {\mathcal D}_g + {\mathcal D}_h \right) {\mathcal D}_h + \left( 1 + \frac{\eta}{\varepsilon} \right) {\mathcal E} \left\{ {\mathcal D}^2 + {\mathcal C}^2 \right\}.
				\end{aligned}
			\end{equation}
			Finally, combining \eqref{Reqn4} and \eqref{R-29} completes the proof of estimate \eqref{R}.
		\end{proof}
		
		\section{Proof of Theorem \ref{Theom.1.}: Existence of Global Solutions}\label{U.E.G.S}
		
		In this section, we establish the well-posedness of the Cauchy problem for the Boltzmann system \eqref{equ:h}--\eqref{equ:hg0}. The local existence of smooth solutions is standard; we refer the reader to, e.g., \cite{MR4296180} for a similar construction.
		Throughout the analysis, we assume that the Cauchy problem \eqref{equ:h}--\eqref{equ:hg0} admits a unique smooth solution $(h_{\varepsilon,\eta}, g_{\varepsilon,\eta})$ on $0 \leqslant t \leqslant T$ for some $T > 0$, satisfying the \textit{a priori} hypothesis
		\begin{equation}\label{priori assumption}
			\sup_{0 \leqslant t \leqslant T} {\mathcal E}(h_{\varepsilon,\eta}, g_{\varepsilon,\eta}) \leqslant 2\epsilon_0,
		\end{equation}
		where $\epsilon_0 > 0$ is a sufficiently small constant to be determined.
		
		Our main goal is to close the \textit{a priori} estimates. Specifically, we establish both microscopic and macroscopic energy estimates for the system \eqref{equ:h}--\eqref{equ:hg0}, which will lead to the global energy estimate \eqref{Glob.Enery.1} and ultimately yield the global existence of solutions.
		
		\subsection{Microscopic Energy Estimates}
		
		\begin{lemma}\label{Lemma.Micro}
			Let $\varepsilon, \eta \in (0,1]$ and let $(h,g)$ be a smooth solution of the Boltzmann system \eqref{equ:h}--\eqref{equ:hg0} on $0 \leqslant t \leqslant T$ for $T > 0$. Then there exist positive constants $\lambda$ and $\lambda_1$, independent of $\varepsilon$ and $\eta$, such that for any $\delta_0 > 0$, the following estimates hold:
			\begin{align}
				\frac{\mathrm{d}}{\mathrm{d}t}\mathcal E_h^{2} + \lambda_1{\mathcal D}^2_h &\lesssim \left( 1 + \frac{\varepsilon^2}{\eta^{1/2}} \right) {\mathcal E} \left\{ {\mathcal D}^2 + {\mathcal C}^2 \right\}, \label{Mic,h,g} \\
				\frac{\mathrm{d}}{\mathrm{d}t} \left\{ \mathcal E_g^2 + 2\mathcal I^{g}_{h} \right\} + \lambda{\mathcal D}^2_g &\lesssim \delta_0{\mathcal C}^2_g + \frac{1}{\delta_0} \left( 1 + \frac{\eta^{1/2}}{\varepsilon} \right)^4 {\mathcal D}^2_h + \left( 1 + \frac{\eta}{\varepsilon} \right) {\mathcal E} \left\{ {\mathcal D}^2 + {\mathcal C}^2 \right\}, \label{Mic,h,g.}
			\end{align}
			where the interactive energy functional $\mathcal I^{g}_{h}$ is defined as in \eqref{def.E-1}.
		\end{lemma}
		
		\begin{proof}
			We first establish \eqref{Mic,h,g}. Applying $\partial^\alpha_x$ to \eqref{equ:h} and taking the $L^2(\mathbb{R}^3_x \times \mathbb{R}^3_v)$ inner product with $\partial^\alpha_x h$, we obtain
			\begin{equation}\label{eqn-h}
				\frac{1}{2} \frac{\mathrm{d}}{\mathrm{d}t} \sum_{|\alpha|\leqslant N} \left\langle \partial^\alpha_x h, \partial^\alpha_x h \right\rangle_{L^2_{x,v}} + \frac{1}{\eta} \sum_{|\alpha|\leqslant N} \left\langle \mathcal L_{\scriptscriptstyle \mathcal D} \partial^\alpha_x h, \partial^\alpha_x h \right\rangle_{L^2_{x,v}} = \frac{\varepsilon}{\eta} \sum_{|\alpha|\leqslant N} \left\langle \partial^\alpha_x \Gamma_{\scriptscriptstyle \mathcal D}(h,g), \partial^\alpha_x h \right\rangle_{L^2_{x,v}},
			\end{equation}
			where the transport term $\left\langle v \cdot \nabla_x \partial^\alpha_x h, \partial^\alpha_x h \right\rangle_{L^2_{x,v}}$ vanishes by integration by parts.
			
			By the coercivity of $\mathcal L_{\scriptscriptstyle \mathcal D}$ \eqref{Coev-2}, we have
			\begin{equation}\label{disp-h}
				\frac{1}{\eta} \sum_{|\alpha|\leqslant N} \left\langle \mathcal L_{\scriptscriptstyle \mathcal D} \partial^\alpha_x h, \partial^\alpha_x h \right\rangle_{L^2_{x,v}} \textcolor{red}{\geqslant} \frac{\lambda_1}{\eta} \sum_{|\alpha|\leqslant N} \left\| \partial^\alpha_x (\mathbf{I - P}_{\scriptscriptstyle 1})h \right\|^2_{\nu_1} = \lambda_1 {\mathcal D}^2_h.
			\end{equation}
			Furthermore, Lemma \ref{lemmaD} yields the estimate for the nonlinear term:
			\begin{equation}\label{Non-h}
				\frac{\varepsilon}{\eta} \sum_{|\alpha|\leqslant N} \left\langle \partial^\alpha_x \Gamma_{\scriptscriptstyle \mathcal D}(h,g), \partial^\alpha_x h \right\rangle_{L^2_{x,v}} \lesssim \left( 1 + \frac{\varepsilon^2}{\eta^{1/2}} \right) {\mathcal E} \left\{ {\mathcal D}^2 + {\mathcal C}^2 \right\}.
			\end{equation}
			Substituting \eqref{disp-h} and \eqref{Non-h} into \eqref{eqn-h}, we arrive at
			\[
			\frac{\mathrm{d}}{\mathrm{d}t} \mathcal E_h^2 + \lambda_1 {\mathcal D}^2_h \lesssim \left( 1 + \frac{\varepsilon^2}{\eta^{1/2}} \right) {\mathcal E} \left\{ {\mathcal D}^2 + {\mathcal C}^2 \right\},
			\]
			which proves \eqref{Mic,h,g}.
			
			Next, we prove \eqref{Mic,h,g.}. Applying $\partial^\alpha_x$ to \eqref{equ:g} and taking the $L^2(\mathbb{R}^3_x \times \mathbb{R}^3_w)$ inner product with $\partial^\alpha_x g$, we get
			\[
			\begin{aligned}
				&\frac{1}{2}\frac{\mathrm{d}}{\mathrm{d}t} \sum_{|\alpha|\leqslant N} \left\langle \partial^\alpha_x g, \partial^\alpha_x g \right\rangle_{L^2_{x,w}} + \frac{1}{\varepsilon^2} \sum_{|\alpha|\leqslant N} \left\langle \mathcal L \partial^\alpha_x g, \partial^\alpha_x g \right\rangle_{L^2_{x,w}} \\
				=& \sum_{|\alpha|\leqslant N} \left\langle \frac{1}{\varepsilon} \partial^\alpha_x \mathcal L_{\scriptscriptstyle \mathcal R} h + \partial^\alpha_x \Gamma_{\scriptscriptstyle \mathcal R}(g,h), \partial^\alpha_x g \right\rangle_{L^2_{x,w}} + \frac{1}{\varepsilon} \sum_{|\alpha|\leqslant N} \left\langle \partial^\alpha_x \Gamma(g,g), \partial^\alpha_x g \right\rangle_{L^2_{x,w}},
			\end{aligned}
			\]
			where the transport term involving $w \cdot \nabla_x g$ again vanishes. Adding $\frac{\mathrm{d}}{\mathrm{d}t} \mathcal I^{g}_{h}$ to both sides yields
			\begin{equation}\label{eqn-g}
				\begin{aligned}
					&\frac{1}{2} \frac{\mathrm{d}}{\mathrm{d}t} \left\{ \sum_{|\alpha|\leqslant N} \left\langle \partial^\alpha_x g, \partial^\alpha_x g \right\rangle_{L^2_{x,w}} + 2\mathcal I^{g}_{h} \right\} + \frac{1}{\varepsilon^2} \sum_{|\alpha|\leqslant N} \left\langle \mathcal L \partial^\alpha_x g, \partial^\alpha_x g \right\rangle_{L^2_{x,w}} \\
					=&  \frac{\mathrm{d}}{\mathrm{d}t} \mathcal I^{g}_{h} + \sum_{|\alpha|\leqslant N} \left\langle \frac{1}{\varepsilon} \partial^\alpha_x \mathcal L_{\scriptscriptstyle \mathcal R} h + \partial^\alpha_x \Gamma_{\scriptscriptstyle \mathcal R}(g,h), \partial^\alpha_x g \right\rangle_{L^2_{x,w}} + \frac{1}{\varepsilon} \sum_{|\alpha|\leqslant N} \left\langle \partial^\alpha_x \Gamma(g,g), \partial^\alpha_x g \right\rangle_{L^2_{x,w}}.
				\end{aligned}
			\end{equation}
			
			The coercivity of $\mathcal L$ \eqref{Coev-1} implies
			\begin{equation}\label{disp-g}
				\frac{1}{\varepsilon^2} \sum_{|\alpha|\leqslant N} \left\langle \mathcal L \partial^\alpha_x g, \partial^\alpha_x g \right\rangle_{L^2_{x,w}} \geq \frac{\lambda}{\varepsilon^2} \sum_{|\alpha|\leqslant N} \left\| \partial^\alpha_x (\mathbf{I - P}) g \right\|^2_{\nu} = \lambda {\mathcal D}^2_g.
			\end{equation}
			Applying Lemma \ref{LemmaR} together with Cauchy--Schwarz's inequality, for any $\delta_0, \delta_1 > 0$, we have
			\begin{equation}\label{Non-g-2}
				\begin{aligned}
					&\frac{\mathrm{d}}{\mathrm{d}t}\mathcal I^{g}_{h} + \sum_{|\alpha|\leqslant N} \left\langle \frac{1}{\varepsilon} \partial^\alpha_x \mathcal L_{\scriptscriptstyle \mathcal R} h + \partial^\alpha_x \Gamma_{\scriptscriptstyle \mathcal R}(g,h), \partial^\alpha_x g \right\rangle_{L^2_{x,w}} \\
					\lesssim& \left( 1 + \frac{\eta^{1/2}}{\varepsilon} \right)^2 \left( {\mathcal C}_g + {\mathcal D}_g + {\mathcal D}_h \right) {\mathcal D}_h + \left( 1 + \frac{\eta}{\varepsilon} \right) {\mathcal E} \left\{ {\mathcal D}^2 + {\mathcal C}^2 \right\} \\
					\lesssim & \delta_1 {\mathcal D}^2_g + \delta_0 {\mathcal C}^2_g + \left( \frac{1}{\delta_0} + \frac{1}{\delta_1} \right) \left( 1 + \frac{\eta^{1/2}}{\varepsilon} \right)^4 {\mathcal D}^2_h + \left( 1 + \frac{\eta}{\varepsilon} \right) {\mathcal E} \left\{ {\mathcal D}^2 + {\mathcal C}^2 \right\}.
				\end{aligned}
			\end{equation}
			Moreover, by Lemma \ref{Gamma.Q.R.D} (or Theorem 3.1 in \cite{MR2095473}), the nonlinear operator $\Gamma(g,g)$ satisfies
			\begin{equation}\label{Non-g}
				\frac{1}{\varepsilon} \sum_{|\alpha|\leqslant N} \left\langle \partial^\alpha_x \Gamma(g,g), \partial^\alpha_x g \right\rangle_{L^2_{x,w}} \lesssim \mathcal E_g \left( {\mathcal D}^2_g + {\mathcal C}^2_g \right) \lesssim \mathcal E \left\{ {\mathcal D}^2 + {\mathcal C}^2 \right\}.
			\end{equation}
			
			Substituting \eqref{disp-g}, \eqref{Non-g-2}, and \eqref{Non-g} into \eqref{eqn-g}, we obtain
			\[
			\begin{aligned}
				&\frac{1}{2} \frac{\mathrm{d}}{\mathrm{d}t} \left\{ \mathcal E_g^2 + 2\mathcal I^{g}_{h} \right\} + \lambda {\mathcal D}^2_g \\
				\lesssim& \delta_1 {\mathcal D}^2_g + \delta_0 {\mathcal C}^2_g + \left( \frac{1}{\delta_0} + \frac{1}{\delta_1} \right) \left( 1 + \frac{\eta^{1/2}}{\varepsilon} \right)^4 {\mathcal D}^2_h + \left( 1 + \frac{\eta}{\varepsilon} \right) {\mathcal E} \left\{ {\mathcal D}^2 + {\mathcal C}^2 \right\}.
			\end{aligned}
			\]
			Choosing $\delta_1 > 0$ sufficiently small such that $\lambda - \delta_1 \geqslant \frac{\lambda}{2}$, we conclude that
			\[
			\frac{\mathrm{d}}{\mathrm{d}t} \left\{ \mathcal E_g^2 + 2\mathcal I^{g}_{h} \right\} + \lambda {\mathcal D}^2_g \lesssim \delta_0 {\mathcal C}^2_g + \frac{1}{\delta_0} \left( 1 + \frac{\eta^{1/2}}{\varepsilon} \right)^4 {\mathcal D}^2_h + \left( 1 + \frac{\eta}{\varepsilon} \right) {\mathcal E} \left\{ {\mathcal D}^2 + {\mathcal C}^2 \right\}.
			\]
			This completes the proof of Lemma \ref{Lemma.Micro}.
		\end{proof}

		\subsection{Macroscopic Energy Estimates}
		In this subsection, we analyze the macroscopic components $\mathbf{P}_{\scriptscriptstyle 1}h$ and $\mathbf{P}g$, where $(h,g)$ is a solution to the Boltzmann system \eqref{equ:h}--\eqref{equ:hg0}. The primary goal is to show that the fluid-like variables $(a_{\scriptscriptstyle 1}, a, b, c)$ associated with the macroscopic dynamics are bounded by the microscopic dissipation terms $(\mathbf{I - P}_{\scriptscriptstyle 1})h$ and $(\mathbf{I - P})g$.
		
		\subsubsection*{Macroscopic Equations for $a_{\scriptscriptstyle 1}$}
		Inserting the macro-micro decomposition \eqref{Decom-h} into equation \eqref{equ:h}, we obtain the time evolution equation for the macroscopic part $\mathbf{P}_{\scriptscriptstyle 1}h$:
		\begin{equation}\label{Macro.Sys.1}
			\begin{aligned}
				\{\partial_t  + v\cdot \nabla_x\} \mathbf{P}_{\scriptscriptstyle 1}h &= -\left\{ \partial_t + v\cdot \nabla_x + \frac{1}{\eta}\mathcal L_{\scriptscriptstyle \mathcal D} \right\} (\mathbf{I - P}_{\scriptscriptstyle 1})h + \frac{\varepsilon}{\eta} \Gamma_{\scriptscriptstyle \mathcal D} (h,g) \\
				&= -\partial_t \bar{r} + \bar{l} + \frac{1}{\eta}\bar{n},
			\end{aligned}
		\end{equation}
		where we have used $\mathcal L_{\scriptscriptstyle \mathcal D} \mathbf{P}_{\scriptscriptstyle 1}h = 0$, and defined the microscopic quantities $\bar{r}, \bar{l}, \bar{n}$ as
		\[
		\bar{r} \equiv (\mathbf{I - P}_{\scriptscriptstyle 1})h, \quad
		\bar{l} \equiv -v\cdot\nabla_x (\mathbf{I - P}_{\scriptscriptstyle 1})h, \quad
		\bar{n} \equiv -\mathcal L_{\scriptscriptstyle \mathcal D}(\mathbf{I - P}_{\scriptscriptstyle 1})h + \varepsilon\Gamma_{\scriptscriptstyle \mathcal D}(h,g).
		\]
		
		Next, inserting the explicit expansion $\mathbf{P}_{\scriptscriptstyle 1}h = a_{\scriptscriptstyle 1}(t,x) \sqrt{\mu_{\varepsilon,\eta}(v)}$ into \eqref{Macro.Sys.1} and projecting the equation onto the orthogonal basis
		\[
		\{\bar{e}_k\}_{k=1}^4 = \left\{ \sqrt{\mu_{\varepsilon,\eta}},\, v_i\sqrt{\mu_{\varepsilon,\eta}} \right\}, \quad 1 \leqslant i \leqslant 3,
		\]
		we derive the governing macroscopic equations for $a_{\scriptscriptstyle 1}$:
		\begin{align}
			\sqrt{\mu_{\varepsilon,\eta}}:\quad& \partial_t a_{\scriptscriptstyle 1} = -\partial_{t}\bar{r}^{(0)} + \bar{l}^{(0)} + \frac{1}{\eta}\bar{n}^{(0)}, \label{Eqn.bar.a.1} \\ 
			v_{i}\sqrt{\mu_{\varepsilon,\eta}}:\quad& \partial_i a_{\scriptscriptstyle 1} = -\partial_{t}\bar{r}_i^{(1)} + \bar{l}_i^{(1)} + \frac{1}{\eta}\bar{n}_i^{(1)}, \label{Eqn.bar.a.2}
		\end{align}
		where the coefficients are given by
		\begin{align}
			\bar{r}^{(0)} &= \big\langle \bar{r}, \sqrt{\mu_{\varepsilon,\eta}} \big\rangle_{L^2_v}, & 
			\bar{r}_i^{(1)} &= \frac{\big\langle \bar{r}, v_i\sqrt{\mu_{\varepsilon,\eta}} \big\rangle_{L^2_v}}{\big\langle v_i\sqrt{\mu_{\varepsilon,\eta}}, v_i\sqrt{\mu_{\varepsilon,\eta}} \big\rangle_{L^2_v}}, \label{Def.bar.r,l,n.1} \\
			\bar{l}^{(0)} &= \big\langle \bar{l}, \sqrt{\mu_{\varepsilon,\eta}} \big\rangle_{L^2_v}, & 
			\bar{l}_i^{(1)} &= \frac{\big\langle \bar{l}, v_i\sqrt{\mu_{\varepsilon,\eta}} \big\rangle_{L^2_v}}{\big\langle v_i\sqrt{\mu_{\varepsilon,\eta}}, v_i\sqrt{\mu_{\varepsilon,\eta}} \big\rangle_{L^2_v}}, \label{Def.bar.r,l,n.2} \\
			\bar{n}^{(0)} &= \big\langle \bar{n}, \sqrt{\mu_{\varepsilon,\eta}} \big\rangle_{L^2_v}, & 
			\bar{n}_i^{(1)} &= \frac{\big\langle \bar{n}, v_i\sqrt{\mu_{\varepsilon,\eta}} \big\rangle_{L^2_v}}{\big\langle v_i\sqrt{\mu_{\varepsilon,\eta}}, v_i\sqrt{\mu_{\varepsilon,\eta}} \big\rangle_{L^2_v}}. \label{Def.bar.r,l,n.3}
		\end{align}
		
		The following lemma demonstrates that the coefficients $\bar{r}_i^{(1)}, \bar{l}_i^{(1)}$, and $\bar{n}_i^{(1)}$ appearing on the right-hand side of \eqref{Eqn.bar.a.2} can be controlled by the microscopic dissipation rate functionals.
		
		\begin{lemma}\label{Lemma.MacCoeffs}
			Let $\varepsilon, \eta \in (0,1]$ and $\alpha \in \mathbb N^3$ with $|\alpha| \leqslant N-1$. Let $\bar{r}_i^{(1)}, \bar{l}_i^{(1)}$, and $\bar{n}_i^{(1)}$ be defined as in \eqref{Def.bar.r,l,n.1}, \eqref{Def.bar.r,l,n.2}, and \eqref{Def.bar.r,l,n.3}, respectively. Then the following estimates hold:
			\begin{align}
				\|\partial_i \partial^\alpha_x \bar{r}_i^{(1)} \|_{L^2_x} &\lesssim \varepsilon {\mathcal D}_h, \label{r.i} \\[4pt]
				\|\partial^\alpha_x \bar{l}_i^{(1)}\|_{L^2_x} &\lesssim \eta^{1/2}{\mathcal D}_h, \label{l.i} \\[4pt]	
				\frac{\varepsilon}{\eta} \|\partial^\alpha_x \bar{n}_i^{(1)}\|_{L^2_x} &\lesssim \frac{\varepsilon^2}{\eta^{1/2}}{\mathcal D}_h + \frac{\varepsilon^3}{\eta} \left\{ {\mathcal E}_h ({\mathcal D}_g + {\mathcal C}_g) + {\mathcal E}_g {\mathcal D}_h \right\}. \label{n.i}
			\end{align}
		\end{lemma}			
		
		\begin{proof}
			We first establish \eqref{r.i}. Recalling the definition of $\bar{r}_i^{(1)}$ in \eqref{Def.bar.r,l,n.1}, we have
			\[
			\partial_i \partial^\alpha_x \bar{r}_i^{(1)} = \frac{\left\langle \partial_i\partial^\alpha_x (\mathbf{I - P}_{\scriptscriptstyle 1})h, v_i\sqrt{\mu_{\varepsilon,\eta}} \right\rangle_{L^2_v}}{\left\langle v_i\sqrt{\mu_{\varepsilon,\eta}}, v_i\sqrt{\mu_{\varepsilon,\eta}} \right\rangle_{L^2_v}}.
			\]
			Applying the equivalence relation \eqref{mu} together with Cauchy--Schwarz's inequality yields
			\[
			\begin{aligned}
				\|\partial_i \partial^\alpha_x \bar{r}_i^{(1)} \|_{L^2_x}
				&\lesssim \frac{1}{\|v_i\sqrt{\mu_{\varepsilon,\eta}}\|_{L^2_v}} \left\| \partial_i\partial^\alpha_x (\mathbf{I - P}_{\scriptscriptstyle 1})h \right\|_{L^2_{x,v}} \\
				&\lesssim \frac{\varepsilon}{\eta^{1/2}} \left\| \partial_i\partial^\alpha_x (\mathbf{I - P}_{\scriptscriptstyle 1})h \right\|_{L^2_{x,v}} \lesssim \varepsilon {\mathcal D}_h.
			\end{aligned}
			\]
			
			Next, we prove \eqref{l.i}. Proceeding similarly to the estimate for \eqref{r.i}, definition \eqref{Def.bar.r,l,n.2} implies
			\[
			\begin{aligned}
				\|\partial^\alpha_x \bar{l}_i^{(1)} \|_{L^2_x}
				&= \left\| \frac{\left\langle \nabla_x\partial^\alpha_x (\mathbf{I - P}_{\scriptscriptstyle 1})h, v v_i\sqrt{\mu_{\varepsilon,\eta}} \right\rangle_{L^2_v}}{\left\langle v_i\sqrt{\mu_{\varepsilon,\eta}}, v_i\sqrt{\mu_{\varepsilon,\eta}} \right\rangle_{L^2_v}} \right\|_{L^2_x} \\
				&\lesssim \left\| \nabla_x\partial^\alpha_x (\mathbf{I - P}_{\scriptscriptstyle 1})h \right\|_{L^2_{x,v}}
				\lesssim \eta^{1/2} {\mathcal D}_h.
			\end{aligned}
			\]	
			
			Finally, we establish \eqref{n.i}. From definition \eqref{Def.bar.r,l,n.3}, $\bar{n}_i^{(1)}$ can be decomposed as
			\begin{equation}\label{neqn1}
				\bar{n}_i^{(1)} = -\frac{\left\langle \mathcal L_{\scriptscriptstyle \mathcal D} (\mathbf{I - P}_{\scriptscriptstyle 1})h, v_i\sqrt{\mu_{\varepsilon,\eta}} \right\rangle_{L^2_v}}{\left\langle v_i\sqrt{\mu_{\varepsilon,\eta}}, v_i\sqrt{\mu_{\varepsilon,\eta}} \right\rangle_{L^2_v}} + \varepsilon \frac{\left\langle \Gamma_{\scriptscriptstyle \mathcal D}(h,g), v_i\sqrt{\mu_{\varepsilon,\eta}} \right\rangle_{L^2_v}}{\left\langle v_i\sqrt{\mu_{\varepsilon,\eta}}, v_i\sqrt{\mu_{\varepsilon,\eta}} \right\rangle_{L^2_v}}.
			\end{equation}	
			To obtain optimal bounds with respect to $\varepsilon$ and $\eta$, we reformulate the numerators in \eqref{neqn1}. Recalling definitions \eqref{Ope.LD.}--\eqref{def,gamma,R,D} and the conservation property \eqref{Jonit.Cons.M,E.a}, we rewrite the first numerator as
			\begin{equation}\label{neqn2}
				\begin{aligned}
					\left\langle \mathcal L_{\scriptscriptstyle \mathcal D} (\mathbf{I - P}_{\scriptscriptstyle 1})h, v_i\sqrt{\mu_{\varepsilon,\eta}} \right\rangle_{L^2_v}
					&= -\left\langle \Gamma_{\scriptscriptstyle \mathcal D} \big( (\mathbf{I - P}_{\scriptscriptstyle 1})h, \sqrt{\mu} \big), v_i\sqrt{\mu_{\varepsilon,\eta}} \right\rangle_{L^2_v} \\[2pt]
					&= -\left\langle \mathcal D \big( \sqrt{\mu_{\varepsilon,\eta}}(\mathbf{I - P}_{\scriptscriptstyle 1})h, \mu \big), v_i \right\rangle_{L^2_v} \\[2pt]
					&= \frac{\eta}{\varepsilon} \left\langle \mathcal R \big( \mu, \sqrt{\mu_{\varepsilon,\eta}}(\mathbf{I - P}_{\scriptscriptstyle 1})h \big), w_i \right\rangle_{L^2_w} \\[2pt]
					&= \frac{\eta}{\varepsilon} \left\langle \Gamma_{\scriptscriptstyle \mathcal R} \big( \sqrt{\mu}, (\mathbf{I - P}_{\scriptscriptstyle 1})h \big), w_i\sqrt{\mu} \right\rangle_{L^2_w} \\[2pt]
					&= \frac{\eta}{\varepsilon} \left\langle \mathcal L_{\scriptscriptstyle \mathcal R} (\mathbf{I - P}_{\scriptscriptstyle 1})h, w_i\sqrt{\mu} \right\rangle_{L^2_w}.
				\end{aligned}
			\end{equation}
			By an analogous computation, the second numerator satisfies
			\begin{equation}\label{neqn3}
				\left\langle \Gamma_{\scriptscriptstyle \mathcal D}(h,g), v_i\sqrt{\mu_{\varepsilon,\eta}} \right\rangle_{L^2_v} = -\frac{\eta}{\varepsilon} \left\langle \Gamma_{\scriptscriptstyle \mathcal R}(g,h), w_i\sqrt{\mu} \right\rangle_{L^2_w}.
			\end{equation}
			Substituting \eqref{neqn2} and \eqref{neqn3} into \eqref{neqn1}, we obtain
			\[
			\bar{n}_i^{(1)} = -\frac{\eta}{\varepsilon} \frac{\left\langle \mathcal L_{\scriptscriptstyle \mathcal R} (\mathbf{I - P}_{\scriptscriptstyle 1})h, w_i\sqrt{\mu} \right\rangle_{L^2_w}}{\left\langle v_i\sqrt{\mu_{\varepsilon,\eta}}, v_i\sqrt{\mu_{\varepsilon,\eta}} \right\rangle_{L^2_v}} - \eta \frac{\left\langle \Gamma_{\scriptscriptstyle \mathcal R}(g,h), w_i\sqrt{\mu} \right\rangle_{L^2_w}}{\left\langle v_i\sqrt{\mu_{\varepsilon,\eta}}, v_i\sqrt{\mu_{\varepsilon,\eta}} \right\rangle_{L^2_v}}.
			\]
			Applying the equivalence relation \eqref{mu}, we bound $\|\partial^\alpha_x \bar{n}_i^{(1)}\|_{L^2_x}$ as follows:
			\begin{equation}\label{neqn4}
				\begin{aligned}
					&\|\partial^\alpha_x \bar{n}_i^{(1)}\|_{L^2_x}\\
					\lesssim& \frac{1}{\|v\sqrt{\mu_{\varepsilon,\eta}}\|^2_{L^2_v}} \left\{ \frac{\eta}{\varepsilon} \left\| \left\langle \mathcal L_{\scriptscriptstyle \mathcal R} \partial^\alpha_x(\mathbf{I - P}_{\scriptscriptstyle 1})h, w_i\sqrt{\mu} \right\rangle_{L^2_w} \right\|_{L^2_x} \right.\\
					&\left.+ \eta \left\| \partial^\alpha_x \left\langle \Gamma_{\scriptscriptstyle \mathcal R}(g,h), w_i\sqrt{\mu} \right\rangle_{L^2_w} \right\|_{L^2_x} \right\} \\
					\lesssim& \frac{\varepsilon^2}{\eta} \left\{ \frac{\eta}{\varepsilon} \left\| \left\langle \mathcal L_{\scriptscriptstyle \mathcal R} \partial^\alpha_x(\mathbf{I - P}_{\scriptscriptstyle 1})h, w_i\sqrt{\mu} \right\rangle_{L^2_w} \right\|_{L^2_x} + \eta \left\| \partial^\alpha_x \left\langle \Gamma_{\scriptscriptstyle \mathcal R}(g,h), w_i\sqrt{\mu} \right\rangle_{L^2_w} \right\|_{L^2_x} \right\} \\
					\lesssim& \varepsilon \left\| \left\langle \mathcal L_{\scriptscriptstyle \mathcal R} \partial^\alpha_x(\mathbf{I - P}_{\scriptscriptstyle 1})h, w_i\sqrt{\mu} \right\rangle_{L^2_w} \right\|_{L^2_x} + \varepsilon^2 \left\| \partial^\alpha_x \left\langle \Gamma_{\scriptscriptstyle \mathcal R}(g,h), w_i\sqrt{\mu} \right\rangle_{L^2_w} \right\|_{L^2_x}.
				\end{aligned}
			\end{equation}
			
			Similar to estimate \eqref{R-32} for $\frac{1}{\varepsilon} \left\| \left\langle \mathcal L_{\scriptscriptstyle \mathcal R}h, |w|^2\sqrt\mu \right\rangle_{L^2_w} \right\|_{H^{N-1}_x}$, we have
			\begin{equation}\label{neqn5}
				\varepsilon \left\| \left\langle \mathcal L_{\scriptscriptstyle \mathcal R} \partial^\alpha_x(\mathbf{I - P}_{\scriptscriptstyle 1})h, w_i\sqrt{\mu} \right\rangle_{L^2_w} \right\|_{L^2_x} \lesssim \varepsilon \left\| \partial^\alpha_x (\mathbf{I - P}_{\scriptscriptstyle 1})h \right\|_{\nu} \lesssim \varepsilon\eta^{1/2}{\mathcal D}_h.
			\end{equation}
			Likewise, following estimates \eqref{R-34}--\eqref{R-44} for $\left\| \left\langle \Gamma_{\scriptscriptstyle \mathcal R} (g,h), |w|^2\sqrt\mu \right\rangle_{L^2_w} \right\|_{H^{N-1}_x}$, we get
			\begin{equation}\label{neqn6}
				\varepsilon^2 \left\| \partial^\alpha_x \left\langle \Gamma_{\scriptscriptstyle \mathcal R}(g,h), w_i\sqrt{\mu} \right\rangle_{L^2_w} \right\|_{L^2_x} \lesssim \varepsilon^2 \left\{ {\mathcal E}_h (\varepsilon{\mathcal D}_g + {\mathcal C}_g) + \eta^{1/2}{\mathcal E}_g{\mathcal D}_h \right\}.
			\end{equation}
			Plugging \eqref{neqn5} and \eqref{neqn6} into \eqref{neqn4} leads to
			\[
			\|\partial^\alpha_x \bar{n}_i^{(1)}\|_{L^2_x} \lesssim \varepsilon\eta^{1/2}{\mathcal D}_h + \varepsilon^2 \left\{ {\mathcal E}_h (\varepsilon{\mathcal D}_g + {\mathcal C}_g) + \eta^{1/2}{\mathcal E}_g{\mathcal D}_h \right\},
			\]
			which immediately implies
			\[
			\frac{\varepsilon}{\eta} \|\partial^\alpha_x \bar{n}_i^{(1)}\|_{L^2_x} \lesssim \frac{\varepsilon^2}{\eta^{1/2}}{\mathcal D}_h + \frac{\varepsilon^3}{\eta}\left\{ {\mathcal E}_h ({\mathcal D}_g + {\mathcal C}_g) + {\mathcal E}_g {\mathcal D}_h \right\}.
			\]
			This completes the proof of Lemma \ref{Lemma.MacCoeffs}.
		\end{proof}
		\subsubsection*{Macroscopic Equations for $(a,b,c)$}
		Substituting the macro-micro decomposition \eqref{Decom-g} into equation \eqref{equ:g} yields the time evolution equation for the macroscopic part $\mathbf{P}g$:
		\begin{equation}\label{Macro.Sys.2}
			\begin{aligned}
				\left\{ \partial_t + \frac{1}{\varepsilon} w \cdot \nabla_x \right\} \mathbf{P}g &= -\left\{ \partial_t + \frac{1}{\varepsilon} w \cdot \nabla_x + \frac{1}{\varepsilon^2} \mathcal L \right\} (\mathbf{I - P})g \\
				&\quad + \frac{1}{\varepsilon}\Gamma(g,g) + \frac{1}{\varepsilon}\mathcal L_{\scriptscriptstyle \mathcal R} (\mathbf{I - P}_{\scriptscriptstyle 1})h + \Gamma_{\scriptscriptstyle \mathcal R} (g,h) \\
				&= -\partial_t r + \frac{1}{\varepsilon}m + \frac{1}{\varepsilon^2}l + \frac{1}{\varepsilon}s + \frac{1}{\varepsilon}n + q,
			\end{aligned}
		\end{equation}
		where we have used $\mathcal L \mathbf{P}g = 0$ and $\mathcal L_{\scriptscriptstyle \mathcal R} \mathbf{P}_{\scriptscriptstyle 1}h = 0$, with the microscopic terms defined as
		\[
		r \equiv (\mathbf{I - P})g, \quad m \equiv -w \cdot \nabla_x (\mathbf{I - P})g, \quad l \equiv -\mathcal L (\mathbf{I - P})g,
		\]
		\[
		s \equiv \Gamma(g,g), \quad n \equiv \mathcal L_{\scriptscriptstyle \mathcal R} (\mathbf{I - P}_{\scriptscriptstyle 1})h, \quad q \equiv \Gamma_{\scriptscriptstyle \mathcal R}(g,h).
		\]
		
		Thus, system \eqref{Macro.Sys.2} can be compactly rewritten as
		\begin{equation}\label{Macro.Sys.2.}
			\{\partial_t + v\cdot \nabla_x\} \mathbf{P}g = -\partial_t r + \frac{1}{\varepsilon}m + \frac{1}{\varepsilon^2}l + \frac{1}{\varepsilon}s + \frac{1}{\varepsilon}n + q.
		\end{equation}
		
		Inserting the macro expansion
		\[
		\mathbf{P}g = \left\{ a(t,x) + b(t,x) \cdot w + c(t,x)|w|^2 \right\} \sqrt{\mu(w)}
		\]
		into equation \eqref{Macro.Sys.2.} and projecting onto the orthogonal basis
		\begin{equation}\label{Basis.g.}
			\{e_k\}_{k=1}^{13} = \left\{ \sqrt{\mu},\, w_i\sqrt{\mu},\, w_iw_j\sqrt{\mu},\, w_i^2\sqrt{\mu},\, w_i|w|^2\sqrt{\mu} \right\}, \quad 1 \leqslant i \leqslant 3,
		\end{equation}
		we derive the system of macroscopic equations for $a$, $b = (b_1, b_2, b_3)$, and $c$ (cf. \cite{MR2095473}):
		\begin{align}
			\sqrt{\mu}:\quad & \partial_t a = -\partial_t r^{(0)} + \frac{1}{\varepsilon}m^{(0)} + \frac{1}{\varepsilon^2}l^{(0)} + \frac{1}{\varepsilon}s^{(0)} + \frac{1}{\varepsilon}n^{(0)} + q^{(0)}, \label{Eqn.bar.abc.1} \\
			w_{i}\sqrt{\mu}:\quad & \partial_t b_i + \frac{1}{\varepsilon}\partial_i a = -\partial_t r_i^{(1)} + \frac{1}{\varepsilon}m_i^{(1)} + \frac{1}{\varepsilon^2}l_i^{(1)} + \frac{1}{\varepsilon}s_i^{(1)} + \frac{1}{\varepsilon}n_i^{(1)} + q_i^{(1)}, \label{Eqn.bar.abc.2} \\
			w_i^2\sqrt{\mu}:\quad & \partial_t c + \frac{1}{\varepsilon}\partial_i b_i = -\partial_t r_i^{(2)} + \frac{1}{\varepsilon}m_i^{(2)} + \frac{1}{\varepsilon^2}l_i^{(2)} + \frac{1}{\varepsilon}s_i^{(2)} + \frac{1}{\varepsilon}n_i^{(2)} + q_i^{(2)}, \label{Eqn.bar.abc.3} \\
			w_iw_j\sqrt{\mu}:\quad & \frac{1}{\varepsilon}\partial_i b_j + \frac{1}{\varepsilon}\partial_j b_i = -\partial_t r_{ij}^{(2)} + \frac{1}{\varepsilon}m_{ij}^{(2)} + \frac{1}{\varepsilon^2}l_{ij}^{(2)} + \frac{1}{\varepsilon}s_{ij}^{(2)} + \frac{1}{\varepsilon}n_{ij}^{(2)} + q_{ij}^{(2)}, \quad i \neq j, \label{Eqn.bar.abc.4} \\
			w_i|w|^2\sqrt{\mu}:\quad & \frac{1}{\varepsilon}\partial_i c = -\partial_t r_i^{(3)} + \frac{1}{\varepsilon}m_i^{(3)} + \frac{1}{\varepsilon^2}l_i^{(3)} + \frac{1}{\varepsilon}s_i^{(3)} + \frac{1}{\varepsilon}n_i^{(3)} + q_i^{(3)}, \label{Eqn.bar.abc.5}
		\end{align}
		where $r^{(0)}, r_i^{(1)}, \dots, q_{ij}^{(2)}, q_i^{(3)}$ denote the coefficients of $r, m, l, s, n, q$ with respect to the thirteen moments in \eqref{Basis.g.}.
		
		The following lemma demonstrates that these moment coefficients on the right-hand side of equations \eqref{Eqn.bar.abc.1}--\eqref{Eqn.bar.abc.5} can be bounded by the microscopic dissipation rate functionals.
		
		\begin{lemma}\label{Lemma.MacMoments}
			Let $\varepsilon, \eta \in (0,1]$ and $\alpha \in \mathbb{N}^3$ with $|\alpha| \leqslant N-1$. Then the following estimates hold:
			\begin{align}
				\left\| \partial_i\partial^\alpha_x \left[ r^{(0)}, r_i^{(1)}, r_i^{(2)}, r_{ij}^{(2)}, r_i^{(3)} \right] \right\|_{L^2_x} &\lesssim {\mathcal D}_g, \label{r} \\
				\left\| \partial^{\alpha}_x \left[ l^{(0)}, l_i^{(1)}, l_i^{(2)}, l_{ij}^{(2)}, l_i^{(3)} \right] \right\|_{L^2_x} &\lesssim {\mathcal D}_g, \label{l} \\
				\left\| \partial^{\alpha}_x \left[ n^{(0)}, n_i^{(1)}, n_i^{(2)}, n_{ij}^{(2)}, n_i^{(3)} \right] \right\|_{L^2_x} &\lesssim {\mathcal D}_h, \label{n} \\
				\left\| \partial^\alpha_x \left[ m^{(0)}, m_i^{(1)}, m_i^{(2)}, m_{ij}^{(2)}, m_i^{(3)} \right] \right\|_{L^2_x} &\lesssim {\mathcal D}_g, \label{m} \\
				\left\| \partial^{\alpha}_x \left[ s^{(0)}, s_i^{(1)}, s_i^{(2)}, s_{ij}^{(2)}, s_i^{(3)} \right] \right\|_{L^2_x} &\lesssim {\mathcal E}_g({\mathcal D}_g + {\mathcal C}_g), \label{s} \\
				\left\| \partial^\alpha_x \left[ q^{(0)}, q_i^{(1)}, q_i^{(2)}, q_{ij}^{(2)}, q_i^{(3)} \right] \right\|_{L^2_x} &\lesssim {\mathcal E}_h ({\mathcal D}_g + {\mathcal C}_g) + {\mathcal E}_g {\mathcal D}_h. \label{q}
			\end{align}
		\end{lemma}
		
		\begin{proof}
			Since $r$, $l$, and $n$ are directly generated by $(\mathbf{I - P})g$, $-\mathcal L (\mathbf{I - P})g$, and $\mathcal L_{\scriptscriptstyle \mathcal R} (\mathbf{I - P}_{\scriptscriptstyle 1})h$ respectively, their definitions involve no spatial derivatives; thus estimates \eqref{r}, \eqref{l}, and \eqref{n} follow directly.
			Estimate \eqref{m} holds because $m$ contains first-order spatial derivatives $\nabla_x (\mathbf{I - P})g$.
			Estimate \eqref{s} concerning the quadratic collision term $\Gamma(g,g)$ was established in \cite{MR2095473}.
			Finally, estimate \eqref{q} for the interaction operator $\Gamma_{\scriptscriptstyle \mathcal R} (g,h)$ follows by arguments similar to estimates \eqref{R-34}--\eqref{R-44} for $\left\| \left\langle \Gamma_{\scriptscriptstyle \mathcal R} (g,h), |w|^2\sqrt\mu \right\rangle_{L^2_w} \right\|_{H^{N-1}_x}$.
		\end{proof}
		
		Now, we establish the bounds for the macroscopic dissipation rate functionals of $h$ and $g$ as stated in the following lemma.
		
		\begin{lemma}\label{Lemma.MacDissipation}
			Let $\varepsilon, \eta \in (0,1]$ and $\alpha \in \mathbb{N}^3$ with $|\alpha| \leqslant N-1$. Let $(h,g)$ be a solution to the Boltzmann system \eqref{equ:h}--\eqref{equ:hg0} over $[0, T]$ for some $T > 0$. Then the following estimates hold:
			\begin{equation}\label{Mac.bar.a}
				\begin{aligned}
					&\frac{\mathrm{d}}{\mathrm{d}t} \sum_{|\alpha|\leqslant N-1} \sum_{i=1}^3 \mathcal I^{a_{\scriptscriptstyle 1}}_{\alpha,i} + \varepsilon^2 \sum_{|\alpha|\leqslant N-1} \|\nabla_x\partial^\alpha_x a_{\scriptscriptstyle 1}\|_{L^2_x}^2 \\
					&\quad \lesssim \left( 1 + \frac{\varepsilon^4}{\eta} \right) {\mathcal D}^2_h + \left( \frac{\varepsilon^3}{\eta} \right)^2 {\mathcal E}^2 \left( {\mathcal D}^2_h + {\mathcal D}^2_g + {\mathcal C}^2_g \right),
				\end{aligned}
			\end{equation}
			and
			\begin{equation}\label{Mac.a.b.c}
				\begin{aligned}
					&\frac{\mathrm{d}}{\mathrm{d}t} \sum_{|\alpha|\leqslant N-1} \sum_{i=1}^3 \left[ \mathcal I^{ab}_{\alpha,i} + \mathcal I^a_{\alpha,i} + \mathcal I^b_{\alpha,i} + \mathcal I^c_{\alpha,i} \right] + \sum_{|\alpha|\leqslant N-1} \|\nabla_x\partial^\alpha_x (a,b,c)\|_{L^2_x}^2 \\
					&\quad \lesssim {\mathcal D}^2_g + {\mathcal D}^2_h + {\mathcal E}^2 \left( {\mathcal D}^2_h + {\mathcal D}^2_g + {\mathcal C}^2_g \right),
				\end{aligned}
			\end{equation}
			where the interactive energy functionals are defined by
			\begin{align}
				\mathcal I^{a_{\scriptscriptstyle 1}}_{\alpha,i} &= \varepsilon^2 \big\langle \partial^\alpha_x\bar{r}_i^{(1)}, \partial_i\partial^\alpha_x a_{\scriptscriptstyle 1} \big\rangle_{L^2_x}, \label{I,a_1,h_2} \\
				\mathcal I^a_{\alpha,i} &= \varepsilon \big\langle \partial^\alpha_x r_i^{(1)}, \partial_i\partial^\alpha_x a \big\rangle_{L^2_x}, \label{I,a} \\
				\mathcal I^b_{\alpha,i} &= \varepsilon \left\langle \sum_{i\neq j} \partial_j\partial^\alpha_x r_i^{(2)} - \sum_{i\neq j} \partial_i\partial^\alpha_x r_{ij}^{(2)} - 2\partial_j\partial^\alpha_x r_j^{(2)},\, \partial^\alpha_x b_j \right\rangle_{L^2_x}, \label{I,b} \\
				\mathcal I^{ab}_{\alpha,i} &= \varepsilon \big\langle \partial^\alpha_x b_i, \partial_i\partial^\alpha_x a \big\rangle_{L^2_x}, \label{I,ab} \\
				\mathcal I^c_{\alpha,i} &= \varepsilon \big\langle \partial^\alpha_x r_i^{(3)}, \partial_i\partial^\alpha_x c \big\rangle_{L^2_x}. \label{I,c}
			\end{align}
		\end{lemma}	
		
		\begin{proof}
			\emph{Estimates for $a_{\scriptscriptstyle 1}$.} For any $|\alpha| \leqslant N-1$ and $i \in \{1, 2, 3\}$, using equation \eqref{Eqn.bar.a.2}, we expand
			\begin{equation}\label{Mac-a-1}
				\begin{aligned}
					\varepsilon^2 \|\partial_i\partial^\alpha_x a_{\scriptscriptstyle 1}\|_{L^2_x}^2 
					&= \varepsilon^2 \big\langle \partial_i\partial^\alpha_x a_{\scriptscriptstyle 1}, \partial_i\partial^\alpha_x a_{\scriptscriptstyle 1} \big\rangle_{L^2_x} \\
					&= \varepsilon^2 \big\langle -\partial_t\partial^\alpha_x\bar{r}_i^{(1)}, \partial_i\partial^\alpha_x a_{\scriptscriptstyle 1} \big\rangle_{L^2_x} + \varepsilon^2 \left\langle \partial^\alpha_x \left[ \bar{l}_i^{(1)} + \frac{1}{\eta}\bar{n}_i^{(1)} \right], \partial_i\partial^\alpha_x a_{\scriptscriptstyle 1} \right\rangle_{L^2_x} \\
					&= -\varepsilon^2 \frac{\mathrm{d}}{\mathrm{d}t} \big\langle \partial^\alpha_x\bar{r}_i^{(1)}, \partial_i\partial^\alpha_x a_{\scriptscriptstyle 1} \big\rangle_{L^2_x} + \varepsilon^2 \big\langle \partial^\alpha_x\bar{r}_i^{(1)}, \partial_i\partial^\alpha_x\partial_t a_{\scriptscriptstyle 1} \big\rangle_{L^2_x} \\
					&\quad + \varepsilon^2 \left\langle \partial^\alpha_x \left[ \bar{l}_i^{(1)} + \frac{1}{\eta}\bar{n}_i^{(1)} \right], \partial_i\partial^\alpha_x a_{\scriptscriptstyle 1} \right\rangle_{L^2_x}.
				\end{aligned}
			\end{equation}
			The first term on the right-hand side of \eqref{Mac-a-1} corresponds to $-\frac{\mathrm{d}}{\mathrm{d}t}\mathcal I^{a_{\scriptscriptstyle 1}}_{\alpha,i}$, where $\mathcal I^{a_{\scriptscriptstyle 1}}_{\alpha,i}$ defined in \eqref{I,a_1,h_2} represents the interactive energy functional between the microscopic component $(\mathbf{I - P}_{\scriptscriptstyle 1})h$ and the macroscopic component $a_{\scriptscriptstyle 1}$. Applying estimates \eqref{Est.t.bar.a} and \eqref{r.i}, the second term is bounded by
			\[
			-\varepsilon^2 \big\langle \partial_i\partial^\alpha_x\bar{r}_i^{(1)}, \partial^\alpha_x\partial_t a_{\scriptscriptstyle 1} \big\rangle_{L^2_x}
			\leqslant \varepsilon^2 \|\partial_i\partial^\alpha_x\bar{r}\|_{L^2_x} \|\partial_t\partial^\alpha_x a_{\scriptscriptstyle 1}\|_{L^2_x}
			\leqslant \varepsilon^2 \cdot (\varepsilon {\mathcal D}_h) \cdot \left( \frac{\eta}{\varepsilon}{\mathcal D}_h \right)
			\lesssim {\mathcal D}^2_h.
			\]
			Furthermore, using estimates \eqref{l.i}--\eqref{n.i} and Cauchy--Schwarz's inequality with $\delta > 0$, the last term is controlled by
			\[
			\begin{aligned}
				&\delta\varepsilon^2 \|\partial_i\partial^\alpha_x a_{\scriptscriptstyle 1}\|_{L^2_x}^2 + \frac{1}{\delta} \left\{ \varepsilon^2 \|\partial^\alpha_x\bar{l}\|_{L^2_x}^2 + \frac{\varepsilon^2}{\eta^2} \|\partial^\alpha_x\bar{n}_i^{(1)}\|_{L^2_x}^2 \right\} \\
				\lesssim &\delta\varepsilon^2 \|\partial_i\partial^\alpha_x a_{\scriptscriptstyle 1}\|_{L^2_x}^2 + \frac{1}{\delta} \left\{ \left( 1 + \frac{\varepsilon^4}{\eta} \right) {\mathcal D}^2_h + \left( \frac{\varepsilon^3}{\eta} \right)^2 {\mathcal E}^2 \left( {\mathcal D}^2_h + {\mathcal D}^2_g + {\mathcal C}^2_g \right) \right\}.
			\end{aligned}
			\]
			Substituting these bounds into \eqref{Mac-a-1} and summing over $|\alpha| \leqslant N-1$ and $i \in \{1, 2, 3\}$, we obtain
			\begin{equation}\label{Mac-a-1.}
				\begin{aligned}
					&\frac{\mathrm{d}}{\mathrm{d}t} \sum_{|\alpha|\leqslant N-1} \sum_{i=1}^3 \mathcal I^{a_{\scriptscriptstyle 1}}_{\alpha,i} + \varepsilon^2 \sum_{|\alpha|\leqslant N-1} \|\nabla_x\partial^\alpha_x a_{\scriptscriptstyle 1}\|_{L^2_x}^2 \\
					\lesssim &\delta\varepsilon^2 \sum_{|\alpha|\leqslant N-1} \|\nabla_x\partial^\alpha_x a_{\scriptscriptstyle 1}\|_{L^2_x}^2 + \frac{1}{\delta} \left\{ \left( 1 + \frac{\varepsilon^4}{\eta} \right) {\mathcal D}^2_h + \left( \frac{\varepsilon^3}{\eta} \right)^2 {\mathcal E}^2 \left( {\mathcal D}^2_h + {\mathcal D}^2_g + {\mathcal C}^2_g \right) \right\}.
				\end{aligned}
			\end{equation}
			Choosing $\delta > 0$ sufficiently small in \eqref{Mac-a-1.} yields estimate \eqref{Mac.bar.a}.
			
			\emph{Estimates for $a$.} For any $|\alpha| \leqslant N-1$ and $i \in \{1, 2, 3\}$, equation \eqref{Eqn.bar.abc.2} gives
			\begin{equation}\label{Mac-a}
				\begin{aligned}
					&\|\partial_i\partial^\alpha_x a\|_{L^2_x}^2 
					= \big\langle \partial_i\partial^\alpha_x a, \partial_i\partial^\alpha_x a \big\rangle_{L^2_x} \\
					=& \left\langle -\varepsilon\partial_t\partial^\alpha_x \big[ b_i + r_i^{(1)} \big] + \partial^\alpha_x \left[ m_i^{(1)} + \frac{1}{\varepsilon}l_i^{(1)} + s_i^{(1)} + n_i^{(1)} + \varepsilon q_i^{(1)} \right], \partial_i\partial^\alpha_x a \right\rangle_{L^2_x} \\
					=& -\frac{\mathrm{d}}{\mathrm{d}t} \big\langle \varepsilon\partial^\alpha_x [b_i + r_i^{(1)}], \partial_i\partial^\alpha_x a \big\rangle_{L^2_x} + \big\langle \varepsilon\partial^\alpha_x [b_i + r_i^{(1)}], \partial_i\partial^\alpha_x \partial_t a \big\rangle_{L^2_x} \\
					&+ \left\langle \partial^\alpha_x \left[ m_i^{(1)} + \frac{1}{\varepsilon}l_i^{(1)} + s_i^{(1)} + n_i^{(1)} + \varepsilon q_i^{(1)} \right], \partial_i\partial^\alpha_x a \right\rangle_{L^2_x}.		
				\end{aligned}
			\end{equation}
			The boundary term in time equals $-\frac{\mathrm{d}}{\mathrm{d}t} [\mathcal I^{ab}_{\alpha,i} + \mathcal I^a_{\alpha,i}]$, where $\mathcal I^a_{\alpha,i}$ defined by \eqref{I,a} represents the interactive energy functional between $(\mathbf{I - P})g$ and $a$, while $\mathcal I^{ab}_{\alpha,i}$ defined by \eqref{I,ab} represents the cross-coupling between $a$ and $b$. Applying inequalities \eqref{Parti.a.}, \eqref{r}, and Cauchy--Schwarz's inequality with $\delta > 0$, the second term is bounded by
			\[
			\begin{aligned}
				&-\big\langle \varepsilon\partial^\alpha_x\partial_i b_i, \partial^\alpha_x \partial_t a \big\rangle_{L^2_x} - \big\langle \varepsilon\partial^\alpha_x\partial_i r_i^{(1)}, \partial^\alpha_x \partial_t a \big\rangle_{L^2_x} \\
				&\quad \lesssim \delta \|\partial_i\partial^\alpha_x b_i\|_{L^2_x}^2 + \frac{\varepsilon^2}{\delta} \|\partial^\alpha_x \partial_t a\|_{L^2_x}^2 + \delta \|\partial^\alpha_x\partial_i r_i^{(1)}\|_{L^2_x}^2 \\
				&\quad \lesssim \delta \|\partial_i\partial^\alpha_x b_i\|_{L^2_x}^2 + \frac{\varepsilon^2}{\delta} \left\{ {\mathcal D}_g^2 + \frac{\eta}{\varepsilon^2} {\mathcal D}_h^2 + {\mathcal E}_h^2 (\varepsilon^2{\mathcal D}_g^2 + {\mathcal C}_g^2) + \eta{\mathcal E}_g^2{\mathcal D}_h^2 \right\} + \delta {\mathcal D}^2_g \\
				&\quad \lesssim \delta \|\partial_i\partial^\alpha_x b_i\|_{L^2_x}^2 + \frac{1}{\delta} \left\{ {\mathcal D}^2_g + {\mathcal D}^2_h + {\mathcal E}^2 \left( {\mathcal D}^2_h + {\mathcal D}^2_g + {\mathcal C}^2_g \right) \right\}.
			\end{aligned}
			\]
			Furthermore, utilizing estimates \eqref{l}--\eqref{q} and Cauchy--Schwarz's inequality with $\delta > 0$, the last term is controlled by
			\[
			\begin{aligned}
				&\delta \|\partial_i\partial^\alpha_x a\|_{L^2_x}^2 + \frac{1}{\delta} \left\| \partial^\alpha_x \left[ m_i^{(1)} + \frac{1}{\varepsilon}l_i^{(1)} + s_i^{(1)} + n_i^{(1)} + \varepsilon q_i^{(1)} \right] \right\|_{L^2_x}^2 \\
				&\quad \lesssim \delta \|\partial_i\partial^\alpha_x a\|_{L^2_x}^2 + \frac{1}{\delta} \left\{ {\mathcal D}^2_g + {\mathcal D}^2_h + {\mathcal E}^2 \left( {\mathcal D}^2_h + {\mathcal D}^2_g + {\mathcal C}^2_g \right) \right\}.
			\end{aligned}
			\]
			Substituting these estimates into \eqref{Mac-a} and summing over $|\alpha| \leqslant N-1$ and $i \in \{1, 2, 3\}$, we obtain
			\begin{equation}\label{Mac-a.}
				\begin{aligned}
					&\frac{\mathrm{d}}{\mathrm{d}t} \sum_{|\alpha|\leqslant N-1} \sum_{i=1}^3 \left[ \mathcal I^{ab}_{\alpha,i} + \mathcal I^a_{\alpha,i} \right] + \sum_{|\alpha|\leqslant N-1} \|\nabla_x\partial^\alpha_x a\|_{L^2_x}^2 \\
					&\quad \lesssim \delta \sum_{|\alpha|\leqslant N-1} \|\nabla_x\partial^\alpha_x (a,b)\|_{L^2_x}^2 + \frac{1}{\delta} \left\{ {\mathcal D}^2_g + {\mathcal D}^2_h + {\mathcal E}^2 \left( {\mathcal D}^2_h + {\mathcal D}^2_g + {\mathcal C}^2_g \right) \right\}.
				\end{aligned}
			\end{equation}
			
			\emph{Estimates for $b$.} As observed in \cite{MR2095473}, equations \eqref{Eqn.bar.abc.3} and \eqref{Eqn.bar.abc.4} imply that $b = (b_1, b_2, b_3)$ satisfies the elliptic-type system
			\begin{equation}\label{b-ellipt}
				\begin{aligned}
					-\Delta_x b_j - \partial_j\partial_j b_j 
					&= -\sum_{i\neq j} \partial_i (\partial_i b_j + \partial_j b_i) - \partial_j \left( 2\partial_j b_j - \sum_{i\neq j} \partial_i b_i \right) \\
					&= -\sum_{i\neq j} \partial_i \left( -\varepsilon\partial_t r_{ij}^{(2)} + m_{ij}^{(2)} + \tfrac{1}{\varepsilon}l_{ij}^{(2)} + s_{ij}^{(2)} + n_{ij}^{(2)} + \varepsilon q_{ij}^{(2)} \right) \\
					&\quad - 2\partial_j \left( -\varepsilon\partial_t r_j^{(2)} + m_j^{(2)} + \tfrac{1}{\varepsilon}l_j^{(2)} + s_j^{(2)} + n_j^{(2)} + \varepsilon q_j^{(2)} \right) \\
					&\quad + \sum_{i\neq j} \partial_j \left( -\varepsilon\partial_t r_i^{(2)} + m_i^{(2)} + \tfrac{1}{\varepsilon}l_i^{(2)} + s_i^{(2)} + n_i^{(2)} + \varepsilon q_i^{(2)} \right).
				\end{aligned}
			\end{equation}
			Applying $\partial^\alpha_x$ with $|\alpha| \leqslant N-1$ to \eqref{b-ellipt}, multiplying by $\partial^\alpha_x b_j$, and integrating over $\mathbb{R}^3_x$, we get
			\begin{equation}\label{Mac-b}
				\begin{aligned}
					&\|\nabla_x\partial^\alpha_x b_j\|_{L^2_x}^2 + \|\partial_j\partial^\alpha_x b_j\|_{L^2_x}^2 \\
					&\quad = -\varepsilon \frac{\mathrm{d}}{\mathrm{d}t} \left\langle \sum_{i\neq j}\partial_j\partial^\alpha_x r_i^{(2)} - \sum_{i\neq j}\partial_i\partial^\alpha_x r_{ij}^{(2)} - 2\partial_j\partial^\alpha_x r_j^{(2)},\, \partial^\alpha_x b_j \right\rangle_{L^2_x} \\
					&\quad\quad + \varepsilon \left\langle \sum_{i\neq j}\partial_j\partial^\alpha_x r_i^{(2)} - \sum_{i\neq j}\partial_i\partial^\alpha_x r_{ij}^{(2)} - 2\partial_j\partial^\alpha_x r_j^{(2)},\, \partial^\alpha_x\partial_t b_j \right\rangle_{L^2_x} \\
					&\quad\quad + \sum_{i\neq j} \left\langle \partial_j\partial^\alpha_x \left[ m_i^{(2)} + \frac{1}{\varepsilon}l_i^{(2)} + s_i^{(2)} + n_i^{(2)} + \varepsilon q_i^{(2)} \right], \partial^\alpha_x b_j \right\rangle_{L^2_x} \\
					&\quad\quad - \sum_{i\neq j} \left\langle \partial_j\partial^\alpha_x \left[ m_{ij}^{(2)} + \frac{1}{\varepsilon}l_{ij}^{(2)} + s_{ij}^{(2)} + n_{ij}^{(2)} + \varepsilon q_{ij}^{(2)} \right], \partial^\alpha_x b_j \right\rangle_{L^2_x} \\
					&\quad\quad - 2 \left\langle \partial_j\partial^\alpha_x \left[ m_j^{(2)} + \frac{1}{\varepsilon}l_j^{(2)} + s_j^{(2)} + n_j^{(2)} + \varepsilon q_j^{(2)} \right], \partial^\alpha_x b_j \right\rangle_{L^2_x}.
				\end{aligned}
			\end{equation}
			The first term on the right-hand side of \eqref{Mac-b} equals $-\frac{\mathrm{d}}{\mathrm{d}t}\mathcal I^b_{\alpha,j}$. Using inequalities \eqref{Est.t.a.b.c}, \eqref{r}, and Cauchy--Schwarz's inequality with $\delta > 0$, the second term is bounded by
			\[
			\begin{aligned}
				&\delta\varepsilon^2 \|\partial^\alpha_x\partial_t b_j\|_{L^2_x}^2 + \frac{1}{\delta} \left\| \sum_{i\neq j}\partial_j\partial^\alpha_x r_i^{(2)} - \sum_{i\neq j}\partial_i\partial^\alpha_x r_{ij}^{(2)} - 2\partial_j\partial^\alpha_x r_j^{(2)} \right\|_{L^2_x}^2 \\
				&\quad \lesssim \delta\varepsilon^2 \left\{ \frac{1}{\varepsilon^2}{\mathcal C}_g^2 + {\mathcal D}_g^2 + \frac{\eta}{\varepsilon^2} {\mathcal D}_h^2 + {\mathcal E}_h^2 (\varepsilon^2{\mathcal D}_g^2 + {\mathcal C}_g^2) + \eta{\mathcal E}_g^2{\mathcal D}_h^2 \right\} + \frac{1}{\delta}{\mathcal D}^2_g \\
				&\quad \lesssim \delta \left[ \sum_{|\alpha|\leqslant 2} \|\nabla_x\partial^\alpha_x (a,b,c)\|_{L^2_x}^2 + {\mathcal D}^2_g + {\mathcal D}^2_h + {\mathcal E}^2 \left( {\mathcal D}^2_h + {\mathcal D}^2_g + {\mathcal C}^2_g \right) \right] + \frac{1}{\delta}{\mathcal D}^2_g.
			\end{aligned}
			\]
			Furthermore, using estimates \eqref{l}--\eqref{q} and integration by parts together with Cauchy--Schwarz's inequality, the remaining terms on the right-hand side of \eqref{Mac-b} are bounded by
			\[
			\delta \|\nabla_x\partial^\alpha_x b_j\|_{L^2_x}^2 + \frac{1}{\delta} \left\{ {\mathcal D}^2_g + {\mathcal D}^2_h + {\mathcal E}^2 \left( {\mathcal D}^2_h + {\mathcal D}^2_g + {\mathcal C}^2_g \right) \right\}.
			\]
			Summing \eqref{Mac-b} over $|\alpha| \leqslant N-1$ and $j \in \{1, 2, 3\}$, we obtain
			\begin{equation}\label{Mac-b.}
				\begin{aligned}
					&\frac{\mathrm{d}}{\mathrm{d}t} \sum_{|\alpha|\leqslant N-1} \sum_{j=1}^3 \mathcal I^b_{\alpha,j} + \sum_{|\alpha|\leqslant N-1} \|\nabla_x\partial^\alpha_x b\|_{L^2_x}^2 \\
					&\quad \lesssim \delta \sum_{|\alpha|\leqslant N-1} \|\nabla_x\partial^\alpha_x (a,b,c)\|_{L^2_x}^2 + \frac{1}{\delta} \left\{ {\mathcal D}^2_g + {\mathcal D}^2_h + {\mathcal E}^2 \left( {\mathcal D}^2_h + {\mathcal D}^2_g + {\mathcal C}^2_g \right) \right\}.
				\end{aligned}
			\end{equation}
			
			\emph{Estimates for $c$.} For any $|\alpha| \leqslant N-1$ and $i \in \{1, 2, 3\}$, equation \eqref{Eqn.bar.abc.5} yields
			\begin{equation}\label{Mac-c}
				\begin{aligned}
					\|\partial_i\partial^\alpha_x c\|_{L^2_x}^2 
					&= \big\langle \partial_i\partial^\alpha_x c, \partial_i\partial^\alpha_x c \big\rangle_{L^2_x} \\
					&= \left\langle -\varepsilon\partial_t\partial^\alpha_x r_i^{(3)} + \partial^\alpha_x \left[ m_i^{(3)} + \frac{1}{\varepsilon}l_i^{(3)} + s_i^{(3)} + n_i^{(3)} + \varepsilon q_i^{(3)} \right], \partial_i\partial^\alpha_x c \right\rangle_{L^2_x} \\
					&= -\frac{\mathrm{d}}{\mathrm{d}t} \big\langle \varepsilon\partial^\alpha_x r_i^{(3)}, \partial_i\partial^\alpha_x c \big\rangle_{L^2_x} + \big\langle \varepsilon\partial^\alpha_x r_i^{(3)}, \partial_i\partial^\alpha_x \partial_t c \big\rangle_{L^2_x} \\
					&\quad + \left\langle \partial^\alpha_x \left[ m_i^{(3)} + \frac{1}{\varepsilon}l_i^{(3)} + s_i^{(3)} + n_i^{(3)} + \varepsilon q_i^{(3)} \right], \partial_i\partial^\alpha_x c \right\rangle_{L^2_x}.		
				\end{aligned}
			\end{equation}
			The boundary term in time equals $-\frac{\mathrm{d}}{\mathrm{d}t}\mathcal I^c_{\alpha,i}$. Using estimates \eqref{Est.t.a.b.c}, \eqref{r}, and Cauchy--Schwarz's inequality with $\delta > 0$, the second term is bounded by
			\[
			\begin{aligned}
				&-\big\langle \varepsilon\partial_i\partial^\alpha_x r_i^{(3)}, \partial^\alpha_x \partial_t c \big\rangle_{L^2_x} 
				\lesssim \delta\varepsilon^2 \|\partial^\alpha_x\partial_t c\|_{L^2_x}^2 + \frac{1}{\delta} \|\partial_i\partial^\alpha_x r_i^{(3)}\|_{L^2_x}^2 \\
				&\quad \lesssim \delta\varepsilon^2 \left\{ \frac{1}{\varepsilon^2}{\mathcal C}_g^2 + {\mathcal D}_g^2 + \frac{\eta}{\varepsilon^2} {\mathcal D}_h^2 + {\mathcal E}_h^2 (\varepsilon^2{\mathcal D}_g^2 + {\mathcal C}_g^2) + \eta{\mathcal E}_g^2{\mathcal D}_h^2 \right\} + \frac{1}{\delta}{\mathcal D}^2_g \\
				&\quad \lesssim \delta \left[ \sum_{|\alpha|\leqslant 2} \|\nabla_x\partial^\alpha_x (a,b,c)\|_{L^2_x}^2 + {\mathcal D}^2_g + {\mathcal D}^2_h + {\mathcal E}^2 \left( {\mathcal D}^2_h + {\mathcal D}^2_g + {\mathcal C}^2_g \right) \right] + \frac{1}{\delta}{\mathcal D}^2_g.
			\end{aligned}
			\]
			By estimates \eqref{l}--\eqref{q} and Cauchy--Schwarz's inequality with $\delta > 0$, the third term is controlled by
			\[
			\delta \|\partial_i\partial^\alpha_x c\|_{L^2_x}^2 + \frac{1}{\delta} \left\{ {\mathcal D}^2_g + {\mathcal D}^2_h + {\mathcal E}^2 \left( {\mathcal D}^2_h + {\mathcal D}^2_g + {\mathcal C}^2_g \right) \right\}.
			\]
			Substituting these into \eqref{Mac-c} and summing over $|\alpha| \leqslant N-1$ and $i \in \{1, 2, 3\}$, we obtain
			\begin{equation}\label{Mac-c.}
				\begin{aligned}
					&\frac{\mathrm{d}}{\mathrm{d}t} \sum_{|\alpha|\leqslant N-1} \sum_{i=1}^3 \mathcal I^c_{\alpha,i} + \sum_{|\alpha|\leqslant N-1} \|\nabla_x\partial^\alpha_x c\|_{L^2_x}^2 \\
					&\quad \lesssim \delta \sum_{|\alpha|\leqslant N-1} \|\nabla_x\partial^\alpha_x (a,b,c)\|_{L^2_x}^2 + \frac{1}{\delta} \left\{ {\mathcal D}^2_g + {\mathcal D}^2_h + {\mathcal E}^2 \left( {\mathcal D}^2_h + {\mathcal D}^2_g + {\mathcal C}^2_g \right) \right\}.
				\end{aligned}
			\end{equation}
			
			Finally, adding estimates \eqref{Mac-a.}, \eqref{Mac-b.}, and \eqref{Mac-c.}, and choosing $\delta > 0$ sufficiently small to absorb the $\delta \sum_{|\alpha|\leqslant N-1} \|\nabla_x\partial^\alpha_x (a,b,c)\|_{L^2_x}^2$ terms into the left-hand side, we arrive at \eqref{Mac.a.b.c}. The proof is complete.
		\end{proof}
		
		\subsection{Uniform Estimates and Global Existence}
		
		We derive uniform energy estimates by combining the microscopic estimates \eqref{Mic,h,g} and \eqref{Mic,h,g.} with the macroscopic estimates \eqref{Mac.bar.a} and \eqref{Mac.a.b.c}. For a sufficiently large constant $M > 0$ and a sufficiently small constant $\tilde{\delta} > 0$ to be specified later, we define the modified energy functional:
		\begin{equation*}\label{Unifor-coe}
			\mathfrak{E}^2_{M,\tilde{\delta}}(h,g) = M{\mathcal E}_h^{2} + \left( {\mathcal E}_g^2 + 2\mathcal I^{g}_{h} \right) + \tilde{\delta} \sum_{|\alpha|\leqslant N-1} \sum_{i=1}^3 \left[ \mathcal I^{a_{\scriptscriptstyle 1}}_{\alpha,i} + \mathcal I^{ab}_{\alpha,i} + \mathcal I^a_{\alpha,i} + \mathcal I^b_{\alpha,i} + \mathcal I^c_{\alpha,i} \right].
		\end{equation*}
		
		Multiplying \eqref{Mic,h,g} by $M$, \eqref{Mac.bar.a} and \eqref{Mac.a.b.c} by $\tilde\delta$, and adding the results to \eqref{Mic,h,g.}, we obtain
		\[
		\begin{aligned}
			&\frac{\mathrm{d}}{\mathrm{d}t}\mathfrak{E}_{M,\tilde{\delta}}^2 + \left( M\lambda_1{\mathcal D}^2_h + \lambda{\mathcal D}^2_g + \tilde{\delta}{\mathcal C}^2_h + \tilde{\delta}{\mathcal C}^2_g \right) \\
			&\quad \leqslant C \left( \tilde{\delta} + \frac{1}{\delta_0} \right) \left[ \left( 1 + \frac{\eta^{1/2}}{\varepsilon} \right)^4 + \left( 1 + \frac{\varepsilon^4}{\eta} \right) \right] {\mathcal D}^2_h + C \tilde{\delta} {\mathcal D}^2_g + C \delta_0 {\mathcal C}^2_g \\
			&\quad\quad + C \left[ M \left( 1 + \frac{\varepsilon^2}{\eta^{1/2}} \right) + \left( 1 + \frac{\eta}{\varepsilon} \right) + \tilde{\delta} \left( 1 + \frac{\varepsilon^3}{\eta} \right)^2 {\mathcal E} \right] {\mathcal E} \left( {\mathcal D}^2 + {\mathcal C}^2 \right).
		\end{aligned}
		\]
		Setting $C\delta_0 = \frac{\tilde{\delta}}{2}$, we first choose $\tilde{\delta} > 0$ sufficiently small and then choose $M > 0$ sufficiently large such that
		\begin{align*}
			M\lambda_1 - C \left( \tilde{\delta} + \frac{1}{\delta_0} \right) \left[ \left( 1 + \frac{\eta^{1/2}}{\varepsilon} \right)^4 + \left( 1 + \frac{\varepsilon^4}{\eta} \right) \right] &> 0, \\
			\lambda - C\tilde{\delta} &> 0.
		\end{align*}
		Consequently, we arrive at
		\begin{equation}\label{Ceqn2}
			\begin{aligned}
				&\frac{\mathrm{d}}{\mathrm{d}t}\mathfrak{E}_{M,\tilde{\delta}}^2 + \left\{ M\lambda_1 - C \left( \tilde{\delta} + \frac{1}{\delta_0} \right) \left[ \left( 1 + \frac{\eta^{1/2}}{\varepsilon} \right)^4 + \left( 1 + \frac{\varepsilon^4}{\eta} \right) \right] \right\} {\mathcal D}^2_h \\
				&\quad + (\lambda - C\tilde{\delta}) {\mathcal D}^2_g + \tilde{\delta}{\mathcal C}^2_h + \frac{\tilde{\delta}}{2}{\mathcal C}^2_g \\
				&\leqslant C \left[ M \left( 1 + \frac{\varepsilon^2}{\eta^{1/2}} \right) + \left( 1 + \frac{\eta}{\varepsilon} \right) + \tilde{\delta} \left( 1 + \frac{\varepsilon^3}{\eta} \right)^2 {\mathcal E} \right] {\mathcal E} \left( {\mathcal D}^2 + {\mathcal C}^2 \right).
			\end{aligned}
		\end{equation}
		
		Next, we claim that there exist constants $c_1, c_2 > 0$ such that
		\begin{equation}\label{Un-Es-2}
			c_1{\mathcal E} \leqslant \mathfrak{E}_{M,\tilde{\delta}} \leqslant c_2{\mathcal E}.
		\end{equation}	
		Indeed, recalling the definition of $\mathcal I^{g}_{h}$ from \eqref{def.E-1}, we have
		\[
		\begin{aligned}
			|2\mathcal I^{g}_{h}| 
			&\leqslant C \left\{ \tilde{\delta} \|(a,b,c)\|_{L^2_x}^2 + \frac{1}{\tilde{\delta}} \left[ \|a_{\scriptscriptstyle 1}\|_{L^2_x}^2 + \|(\mathbf{I - P}_{\scriptscriptstyle 1})h\|_{L^2_{x,v}}^2 \right] \right\} \\
			&\leqslant C\tilde{\delta}{\mathcal E}_g^2 + \frac{C}{\tilde{\delta}}{\mathcal E}_h^{2}.
		\end{aligned}
		\]
		Furthermore, from the definitions of $\mathcal I^{a_{\scriptscriptstyle 1}}_{\alpha,i}, \mathcal I^{ab}_{\alpha,i}, \mathcal I^a_{\alpha,i}, \mathcal I^b_{\alpha,i}$, and $\mathcal I^c_{\alpha,i}$ in \eqref{I,a_1,h_2}--\eqref{I,c}, it follows that
		\[
		\begin{aligned}
			&\tilde{\delta} \sum_{|\alpha|\leqslant N-1} \sum_{i=1}^3 \left( |\mathcal I^{a_{\scriptscriptstyle 1}}_{\alpha,i}| + |\mathcal I^{ab}_{\alpha,i}| + |\mathcal I^a_{\alpha,i}| + |\mathcal I^b_{\alpha,i}| + |\mathcal I^c_{\alpha,i}| \right) \\
			&\quad \leqslant C\tilde{\delta} \sum_{|\alpha|\leqslant N} \left\{ \|\partial_x^{\alpha}(\mathbf{I - P})g\|_{L^2_{x,w}}^2 + \|\partial_x^{\alpha}\mathbf{P}g\|_{L^2_{x,w}}^2 + \|\partial_x^{\alpha}(\mathbf{I - P}_{\scriptscriptstyle 1})h\|_{L^2_{x,v}}^2 + \|\partial_x^{\alpha}\mathbf{P}_{\scriptscriptstyle 1}h\|_{L^2_{x,v}}^2 \right\} \\
			&\quad \leqslant C\tilde{\delta} \left( {\mathcal E}_g^2 + {\mathcal E}_h^{2} \right).
		\end{aligned}
		\]
		Combining these bounds yields
		\[
		|2\mathcal I^{g}_{h}| + \tilde{\delta} \sum_{|\alpha|\leqslant N-1} \sum_{i=1}^3 \left( |\mathcal I^{a_{\scriptscriptstyle 1}}_{\alpha,i}| + |\mathcal I^{ab}_{\alpha,i}| + |\mathcal I^a_{\alpha,i}| + |\mathcal I^b_{\alpha,i}| + |\mathcal I^c_{\alpha,i}| \right) \leqslant C\tilde{\delta}{\mathcal E}_g^2 + \frac{C}{\tilde{\delta}}{\mathcal E}_h^{2}.
		\]
		Thus, choosing $\tilde{\delta}$ sufficiently small and $M$ sufficiently large so that
		\[
		1 - C\tilde{\delta} > 0 \quad \text{and} \quad M - \frac{C}{\tilde{\delta}} > 0,
		\]
		we deduce that
		\[
		\min\left\{ M - \frac{C}{\tilde{\delta}}, \, 1 - C\tilde{\delta} \right\} {\mathcal E}^2 \leqslant \mathfrak{E}^2_{M,\tilde{\delta}} \leqslant \max\left\{ M + \frac{C}{\tilde{\delta}}, \, 1 + C\tilde{\delta} \right\} {\mathcal E}^2,
		\]
		which proves the equivalence relation \eqref{Un-Es-2}.
		
		Applying \eqref{Un-Es-2} to \eqref{Ceqn2}, we arrive at the following uniform estimate.
		
		\begin{theorem}
			Let $(h,g)$ be a solution to the scaled Boltzmann system \eqref{Eqn.bar.abc.1}--\eqref{Eqn.bar.abc.5}. Then there exists a constant $c_0 > 0$ depending on $\varepsilon$ and $\eta$ such that whenever $\mathfrak{E}_{M,\tilde{\delta}} \leqslant 1$, it holds that
			\begin{equation}\label{Glob.Estim}
				\frac{\mathrm{d}}{\mathrm{d}t}\mathfrak{E}_{M,\tilde{\delta}}^2 + \left( {\mathcal D}^2 + {\mathcal C}^2 \right) \leqslant c_0 \mathfrak{E}_{M,\tilde{\delta}} \left( {\mathcal D}^2 + {\mathcal C}^2 \right).
			\end{equation}
		\end{theorem}
		
		We are now in a position to establish Theorem \ref{Theom.1.} using a continuation argument.
		
		\begin{proof}[Proof of Theorem \ref{Theom.1.}]
			Under the regime $\varepsilon^3 \lesssim \eta \lesssim \varepsilon^2$, the following uniform parameter bound holds:
			\begin{equation}\label{Ceqn4}
				1 + \frac{\eta^{1/2}}{\varepsilon} + \frac{\eta}{\varepsilon} + \frac{\varepsilon^4}{\eta} + \frac{\varepsilon^2}{\eta^{1/2}} + \frac{\varepsilon^3}{\eta} \leqslant C,
			\end{equation}
			where $C > 0$ is a constant independent of $\varepsilon$ and $\eta$.
			
			In view of estimates \eqref{Ceqn2}--\eqref{Ceqn4}, we can select constants $c_0^*, c_1^*, c_2^* > 0$, independent of $\varepsilon$ and $\eta$, to replace $c_0, c_1, c_2$ in \eqref{Glob.Estim} and \eqref{Un-Es-2}, yielding
			\[
			c_1^*{\mathcal E} \leqslant \mathfrak{E}_{M,\tilde{\delta}} \leqslant c_2^*{\mathcal E},
			\]
			and
			\begin{equation}\label{Glob.Estim2}
				\frac{\mathrm{d}}{\mathrm{d}t}\mathfrak{E}_{M,\tilde{\delta}}^2 + \left( {\mathcal D}^2 + {\mathcal C}^2 \right) \leqslant c_0^* \mathfrak{E}_{M,\tilde{\delta}} \left( {\mathcal D}^2 + {\mathcal C}^2 \right).
			\end{equation}
			
			Now, we define the threshold
			\[
			\epsilon_0 = \min\left\{ \frac{1}{2c_2^*}, \, \frac{1}{4c_0^*c_2^*} \right\}.
			\]
			We select the initial data $(h_{\varepsilon,\eta,0}, g_{\varepsilon,\eta,0})$ satisfying
			\[
			{\mathcal E}(h_{\varepsilon,\eta,0}, g_{\varepsilon,\eta,0}) \leqslant \epsilon_0.
			\]
			Recall the a priori assumption \eqref{priori assumption} that the solution $(h_{\varepsilon,\eta}, g_{\varepsilon,\eta})$ satisfies ${\mathcal E}(h_{\varepsilon,\eta}, g_{\varepsilon,\eta}) \leqslant 2\epsilon_0$ for all $0 \leqslant t \leqslant T$. We further define
			\begin{equation*}
				T^* = \sup\left\{ t \in \mathbb{R}^+ \; \Big| \; {\mathcal E}(h_{\varepsilon,\eta}, g_{\varepsilon,\eta})(t) \leqslant 2\epsilon_0 \frac{c_2^*}{c_1^*} \right\} > 0.
			\end{equation*}
			Notice that for all $0 \leqslant t \leqslant T$,
			\[
			\mathfrak{E}(h_{\varepsilon,\eta}, g_{\varepsilon,\eta})(t) \leqslant c_2^*{\mathcal E}(h_{\varepsilon,\eta}, g_{\varepsilon,\eta})(t) \leqslant 2c_2^*\epsilon_0 < 1.
			\]
			Consequently, the global energy differential inequality \eqref{Glob.Estim2} implies
			\[
			\frac{\mathrm{d}}{\mathrm{d}t}\mathfrak{E}_{M,\tilde{\delta}}^2(t) + (1 - 2c_0^*c_2^*\epsilon_0) \left( {\mathcal D}^2 + {\mathcal C}^2 \right) \leqslant 0.
			\]
			Since $1 - 2c_0^*c_2^*\epsilon_0 \geqslant \frac{1}{2}$ by our choice of $\epsilon_0$, integrating in time yields
			\[
			\mathfrak{E}_{M,\tilde{\delta}}^2(T) + \frac{1}{2}\int_{0}^{T} \left( {\mathcal D}^2 + {\mathcal C}^2 \right) \mathrm{d}t \leqslant \mathfrak{E}_{M,\tilde{\delta}}^2(0),
			\]
			which immediately leads to ${\mathcal E}(h_{\varepsilon,\eta}, g_{\varepsilon,\eta})(T) \leqslant \epsilon_0 \frac{c_2^*}{c_1^*}$. This shows that $T^* = \infty$, thereby concluding the proof of Theorem \ref{Theom.1.}.
		\end{proof}

		\section{Proof of Theorem \ref{Theom.2}: Limits of the IVNS System}
		
		Throughout this section, we assume that the scaling parameters $\varepsilon$ and $\eta$ satisfy the Vlasov--Navier--Stokes regime \eqref{scaling assumption}, namely,
		\[
		O(1)\varepsilon^3 \leqslant \eta \leqslant o(1)\varepsilon^2 \quad \text{as } \varepsilon \to 0.
		\]
		
		By Theorem \ref{Theom.1.}, there exists a constant $\epsilon_0 > 0$ independent of $\varepsilon$ and $\eta$ such that for any initial data $(h_{\varepsilon,\eta,0}, g_{\varepsilon,\eta,0})$ satisfying
		\[
		{\mathcal E}(h_{\varepsilon,\eta,0}, g_{\varepsilon,\eta,0}) \leqslant \epsilon_0,
		\]
		the Cauchy problem \eqref{equ:h}--\eqref{equ:hg0} admits a unique global solution satisfying the uniform energy estimate \eqref{Glob.Enery.1}. Consequently, we establish the uniform bounds
		\begin{equation}\label{global,bdh}
			\sup_{t\geqslant0} \|h_{\varepsilon,\eta}\|_{L^2_v(H^N_x)}^2 \lesssim 1,
		\end{equation}
		\begin{equation}\label{global,bdg}
			\sup_{t\geqslant0} \|g_{\varepsilon,\eta}\|_{L^2_w(H^N_x)}^2 \lesssim 1,
		\end{equation}
		along with the global dissipation bounds
		\begin{equation}\label{global,dish}
			\int_{0}^\infty \sum_{|\alpha|\leqslant N} \|\partial^\alpha_x (\mathbf{I - P}_{\scriptscriptstyle 1}) h_{\varepsilon,\eta}\|_{\nu_1}^2 \, \mathrm{d}t \lesssim \eta,
		\end{equation}
		and
		\begin{equation}\label{global,disg}
			\int_{0}^\infty \sum_{|\alpha|\leqslant N} \|\partial^\alpha_x (\mathbf{I - P}) g_{\varepsilon,\eta}\|_\nu^2 \, \mathrm{d}t \lesssim \varepsilon^2.
		\end{equation}
		
		\subsection{The Limiting Behavior of $F_{\varepsilon,\eta}$}
		
		
		We define the macroscopic fluid density associated with $F_{\varepsilon,\eta}$ as
		\begin{equation*}
			\rho_{\scriptscriptstyle F_{\varepsilon,\eta}} := \langle F_{\varepsilon,\eta}, 1 \rangle_{L^2_v} = \left\langle h_{\varepsilon,\eta}, \sqrt{\mu_{\varepsilon,\eta}} \right\rangle_{L^2_v}.
		\end{equation*}
		Thus, the distribution function $F_{\varepsilon,\eta}$ admits the orthogonal decomposition
		\begin{equation}\label{expres-F}
			F_{\varepsilon,\eta} = \rho_{\scriptscriptstyle F_{\varepsilon,\eta}} \mu_{\varepsilon,\eta} + \sqrt{\mu_{\varepsilon,\eta}} \, (\mathbf{I - P}_{\scriptscriptstyle 1}) h_{\varepsilon,\eta}.
		\end{equation}
		
		Taking the $L^2_v$ inner product of \eqref{equ:h} with $\sqrt{\mu_{\varepsilon,\eta}}$ yields the continuity equation
		\begin{equation}\label{a1-t}
			\partial_t \rho_{\scriptscriptstyle F_{\varepsilon,\eta}} = -\operatorname{div}_x \left\langle v \sqrt{\mu_{\varepsilon,\eta}}, (\mathbf{I - P}_{\scriptscriptstyle 1}) h_{\varepsilon,\eta} \right\rangle_{L^2_v}.
		\end{equation}
		
		The uniform energy bound \eqref{global,bdh} directly implies
		\begin{equation}\label{mac-bound2}
			\sup_{t\geqslant0} \|\rho_{\scriptscriptstyle F_{\varepsilon,\eta}}\|_{H^N_x}^2 \lesssim \sup_{t\geqslant0} \|h_{\varepsilon,\eta}\|_{L^2_v(H^N_x)}^2 \lesssim 1.
		\end{equation}
		For any fixed time $T > 0$, taking the $L^2([0,T]; H^{N-1}_x)$-norm in \eqref{a1-t} and applying the Cauchy--Schwarz inequality, we deduce
		\[
		\|\partial_t \rho_{\scriptscriptstyle F_{\varepsilon,\eta}}\|_{L^2([0,T]; H^{N-1}_x)} \lesssim \|v \sqrt{\mu_{\varepsilon,\eta}}\|_{L^2_v} \left( \int_0^T \|\nabla_x (\mathbf{I - P}_{\scriptscriptstyle 1}) h_{\varepsilon,\eta}\|_{L^2_v(H^{N-1}_x)}^2 \, \mathrm{d}t \right)^{1/2}.
		\]
		Using the equivalence relation \eqref{mu} and the dissipation bound \eqref{global,dish}, it follows that
		\begin{equation}\label{mac-bound2.}
			\|\partial_t \rho_{\scriptscriptstyle F_{\varepsilon,\eta}}\|_{L^2([0,T]; H^{N-1}_x)} \lesssim 1.
		\end{equation}
		
		Recall the Sobolev embedding hierarchy
		\begin{equation}\label{theta-rho-3}
			H^N_x \hookrightarrow H^{N-\sigma}_x \hookrightarrow H^{N-1}_x, \quad \sigma \in (0,1],
		\end{equation}
		where $H^N_x \hookrightarrow H^{N-\sigma}_x$ is compact and $H^{N-\sigma}_x \hookrightarrow H^{N-1}_x$ is continuous. Invoking the Aubin--Lions--Simon compactness theorem (see Appendix \ref{A-L-S}) alongside estimates \eqref{mac-bound2} and \eqref{mac-bound2.}, we conclude that there exists a limit function $\rho_{\scriptscriptstyle F} \in C(\mathbb{R}^+; H^{N-\sigma}_x) \cap L^\infty(\mathbb{R}^+; H^N_x)$ such that
		\begin{equation}\label{Str.rho.F}
			\rho_{\scriptscriptstyle F_{\varepsilon,\eta}} \longrightarrow \rho_{\scriptscriptstyle F} \quad \text{strongly in } C([0,T]; H^{N-\sigma}_x) \text{ as } \varepsilon \to 0,
		\end{equation}
		for any $T > 0$ and $\sigma \in (0,1]$.
		
		To identify the distribution limit of $F_{\varepsilon,\eta}$ as $\varepsilon \to 0$, we introduce for any $T > 0$ the test function class $Y(t,x,v)$ of the form
		\begin{equation}\label{Y1}
			Y(t,x,v) = \chi_1(t,x) + \psi(t,x) \cdot v + \chi_2(t,x) |v|^2, \quad \forall \, 0 \leqslant t \leqslant T,
		\end{equation}
		where the coefficient functions satisfy
		\begin{align}
			\chi_i(t,x) &\in C_c^1([0,T); C_c^\infty(\mathbb{R}^3_x))~\text{with } \chi_i(0,x) = \chi_i^0(x) \in C_c^\infty(\mathbb{R}^3_x) \text{ for } i=1,2, \label{Y2} \\
			\psi(t,x) &\in C_c^1([0,T); C_c^\infty(\mathbb{R}^3_x)) ~\text{with } \operatorname{div}_x \psi(t,x) = 0, \, \psi(0,x) = \psi^0(x) \in C_c^\infty(\mathbb{R}^3_x). \label{Y4}
		\end{align}

		The asymptotic limit of $F_{\varepsilon,\eta}$ is characterized in the following lemma.
		\begin{lemma}\label{F,con}
			Under the assumptions of Theorem \ref{Theom.2}, as $\varepsilon \to 0$, the fluid distribution function $F_{\varepsilon,\eta}$ converges in the sense of distributions to
			\begin{equation}\label{F.def}
				F_{\varepsilon,\eta} \longrightarrow F \equiv \rho_{\scriptscriptstyle F}(t,x) \delta_{v=0}.
			\end{equation}
			More precisely, for any $T > 0$ and any test function $Y(t,x,v)$ satisfying \eqref{Y1}--\eqref{Y4},
			\begin{equation}\label{F.conv}
				\int_{0}^T \int_{\mathbb{R}^3 \times \mathbb{R}^3} F_{\varepsilon,\eta} Y \, \mathrm{d}v \mathrm{d}x \mathrm{d}t \xrightarrow{\varepsilon \to 0} \int_{0}^T \int_{\mathbb{R}^3 \times \mathbb{R}^3} F Y \, \mathrm{d}v \mathrm{d}x \mathrm{d}t.
			\end{equation}               
		\end{lemma}
		
		\begin{proof}
			Applying decomposition \eqref{expres-F}, we split the integral into two components:
			\begin{equation}\label{F-1}
				\begin{aligned}
					\int_{0}^T \int_{\mathbb{R}^3 \times \mathbb{R}^3} (F_{\varepsilon,\eta} - F) Y \, \mathrm{d}v \mathrm{d}x \mathrm{d}t 
					&= \int_{0}^T \int_{\mathbb{R}^3 \times \mathbb{R}^3} (\rho_{\scriptscriptstyle F_{\varepsilon,\eta}} \mu_{\varepsilon,\eta} - \rho_{\scriptscriptstyle F} \delta_{v=0}) Y \, \mathrm{d}v \mathrm{d}x \mathrm{d}t \\
					&\quad + \int_{0}^T \int_{\mathbb{R}^3 \times \mathbb{R}^3} \sqrt{\mu_{\varepsilon,\eta}} \, (\mathbf{I - P}_{\scriptscriptstyle 1}) h_{\varepsilon,\eta} Y \, \mathrm{d}v \mathrm{d}x \mathrm{d}t.
				\end{aligned}
			\end{equation}
			
			We further decompose the macroscopic term on the right-hand side of \eqref{F-1} as
			\begin{equation}\label{F-2}
				\begin{aligned}
					\int_{0}^T \int_{\mathbb{R}^3 \times \mathbb{R}^3} (\rho_{\scriptscriptstyle F_{\varepsilon,\eta}} \mu_{\varepsilon,\eta} - \rho_{\scriptscriptstyle F} \delta_{v=0}) Y \, \mathrm{d}v \mathrm{d}x \mathrm{d}t 
					&= \int_{0}^T \int_{\mathbb{R}^3 \times \mathbb{R}^3} (\rho_{\scriptscriptstyle F_{\varepsilon,\eta}} - \rho_{\scriptscriptstyle F}) \mu_{\varepsilon,\eta} Y \, \mathrm{d}v \mathrm{d}x \mathrm{d}t \\
					&\quad + \int_{0}^T \int_{\mathbb{R}^3 \times \mathbb{R}^3} \rho_{\scriptscriptstyle F} (\mu_{\varepsilon,\eta} - \delta_{v=0}) Y \, \mathrm{d}v \mathrm{d}x \mathrm{d}t.
				\end{aligned}
			\end{equation}
			
			We estimate each term on the right-hand side of \eqref{F-2} individually. For the first term, using the equivalence relation \eqref{mu}, the first integral is bounded by
			\begin{equation*}
				\begin{aligned}
					&\left| \int_{0}^T \int_{\mathbb{R}^3 \times \mathbb{R}^3} (\rho_{\scriptscriptstyle F_{\varepsilon,\eta}} - \rho_{\scriptscriptstyle F}) \mu_{\varepsilon,\eta} (\chi_1 + \psi \cdot v + \chi_2 |v|^2) \, \mathrm{d}v \mathrm{d}x \mathrm{d}t \right| \\
					&\quad \lesssim \|\rho_{\scriptscriptstyle F_{\varepsilon,\eta}} - \rho_{\scriptscriptstyle F}\|_{C([0,T]; L^2_x)} \|(\chi_1, \chi_2)\|_{L^1([0,T]; L^2_x)} \|\mu_{\varepsilon,\eta} (1 + |v|^2)\|_{L^1_v} \\
					&\quad \lesssim \|\rho_{\scriptscriptstyle F_{\varepsilon,\eta}} - \rho_{\scriptscriptstyle F}\|_{C([0,T]; L^2_x)}.
				\end{aligned}
			\end{equation*}
			Combining this with the strong convergence \eqref{Str.rho.F} gives
			\[
			\int_{0}^T \int_{\mathbb{R}^3 \times \mathbb{R}^3} (\rho_{\scriptscriptstyle F_{\varepsilon,\eta}} - \rho_{\scriptscriptstyle F}) \mu_{\varepsilon,\eta} Y \, \mathrm{d}v \mathrm{d}x \mathrm{d}t \xrightarrow{\varepsilon \to 0} 0.
			\]
			
			For the second term in \eqref{F-2}, evaluating $Y(t,x,0) = \chi_1(t,x)$ against the Dirac mass $\delta_{v=0}$ yields
			\begin{equation*}
				\begin{aligned}
					&\left| \int_{0}^T \int_{\mathbb{R}^3 \times \mathbb{R}^3} \rho_{\scriptscriptstyle F} (\mu_{\varepsilon,\eta} - \delta_{v=0}) (\chi_1 + \psi \cdot v + \chi_2 |v|^2) \, \mathrm{d}v \mathrm{d}x \mathrm{d}t \right| \\
					&\quad = \left| \int_{0}^T \int_{\mathbb{R}^3 \times \mathbb{R}^3} \rho_{\scriptscriptstyle F} \chi_2 |v|^2 \mu_{\varepsilon,\eta} \, \mathrm{d}v \mathrm{d}x \mathrm{d}t \right| \\
					&\quad \lesssim \|\rho_{\scriptscriptstyle F}\|_{L^\infty([0,T]; L^2_x)} \|\chi_2\|_{L^1([0,T]; L^2_x)} \|\mu_{\varepsilon,\eta} |v|^2\|_{L^1_v} \\
					&\quad \lesssim \frac{\eta}{\varepsilon^2} \|\rho_{\scriptscriptstyle F}\|_{L^\infty([0,T]; L^2_x)} \|\chi_2\|_{L^1([0,T]; L^2_x)} = o(1) \xrightarrow{\varepsilon \to 0} 0.
				\end{aligned}
			\end{equation*}
			
			Consequently, combining the estimates above yields
			\begin{equation}\label{Fconv1}
				\int_{0}^T \int_{\mathbb{R}^3 \times \mathbb{R}^3} (\rho_{\scriptscriptstyle F_{\varepsilon,\eta}} \mu_{\varepsilon,\eta} - \rho_{\scriptscriptstyle F} \delta_{v=0}) Y \, \mathrm{d}v \mathrm{d}x \mathrm{d}t \xrightarrow{\varepsilon \to 0} 0.
			\end{equation}
			
			Next, for the microscopic component in \eqref{F-1}, applying the Cauchy--Schwarz inequality gives
			\begin{equation}\label{Fconv2}
				\begin{aligned}
					&\left| \int_{0}^T \int_{\mathbb{R}^3 \times \mathbb{R}^3} \sqrt{\mu_{\varepsilon,\eta}} \, (\mathbf{I - P}_{\scriptscriptstyle 1}) h_{\varepsilon,\eta} Y \, \mathrm{d}v \mathrm{d}x \mathrm{d}t \right| \\
					\lesssim &\left( \int_{0}^T \|(\mathbf{I - P}_{\scriptscriptstyle 1}) h_{\varepsilon,\eta}\|_{L^2_{x,v}}^2 \, \mathrm{d}t \right)^{1/2} \|(\chi_1, \psi, \chi_2)\|_{L^2([0,T]; L^2_x)} \left\| \sqrt{\mu_{\varepsilon,\eta}} (1 + |v| + |v|^2) \right\|_{L^2_v} \\
					\lesssim & \left( \int_{0}^T \|(\mathbf{I - P}_{\scriptscriptstyle 1}) h_{\varepsilon,\eta}\|_{\nu_1}^2 \, \mathrm{d}t \right)^{1/2} \lesssim \eta^{1/2} \xrightarrow{\varepsilon \to 0} 0,
				\end{aligned}
			\end{equation}
			where we used the energy dissipation bound \eqref{global,dish}.
			
			Finally, substituting \eqref{Fconv1} and \eqref{Fconv2} into \eqref{F-1} yields \eqref{F.conv}, completing the proof of Lemma \ref{F,con}.
		\end{proof}
		
		\subsection{Convergence of Fluid Variables}
		
		From the uniform energy bound \eqref{global,bdg}, there exists a limit function $g \in L^\infty([0,\infty); L^2_w(H^N_x))$ such that, as $\varepsilon \to 0$,
		\begin{equation}\label{g'Limit}
			g_{\varepsilon,\eta} \stackrel{\ast}{\rightharpoonup} g \quad \text{weakly-$\star$ in } L^\infty([0,\infty); L^2_w(H^N_x)).
		\end{equation}
		Furthermore, the energy dissipation bound \eqref{global,disg} yields
		\begin{equation}\label{I-P.g'Limit}
			(\mathbf{I - P}) g_{\varepsilon,\eta} \longrightarrow 0 \quad \text{strongly in } L^2([0,\infty); L^2_w(H^N_x)) \quad \text{as } \varepsilon \to 0.
		\end{equation}
		Combining \eqref{g'Limit} and \eqref{I-P.g'Limit}, we obtain $(\mathbf{I - P}) g = 0$, which implies the existence of macroscopic fluid quantities $(\rho, u, \theta) \in L^\infty([0,\infty); H^N_x)$ such that
		\begin{equation*}
			g = \left\{ \rho + u \cdot w + \frac{\theta}{2} (|w|^2 - 3) \right\} \sqrt{\mu}.
		\end{equation*}
		We define the macroscopic fluid variables associated with $g_{\varepsilon,\eta}$ as
		\begin{equation*}
			\rho_{\varepsilon,\eta} = \left\langle g_{\varepsilon,\eta}, \sqrt{\mu} \right\rangle_{L^2_w}, \quad
			u_{\varepsilon,\eta} = \left\langle g_{\varepsilon,\eta}, w \sqrt{\mu} \right\rangle_{L^2_w}, \quad
			\theta_{\varepsilon,\eta} = \left\langle g_{\varepsilon,\eta}, \left( \frac{|w|^2}{3} - 1 \right) \sqrt{\mu} \right\rangle_{L^2_w}.
		\end{equation*}
		Accordingly, the fluctuation distribution function $g_{\varepsilon,\eta}$ can be decomposed as
		\begin{equation}\label{expres-g}
			g_{\varepsilon,\eta} = \left\{ \rho_{\varepsilon,\eta} + u_{\varepsilon,\eta} \cdot w + \frac{\theta_{\varepsilon,\eta}}{2} (|w|^2 - 3) \right\} \sqrt{\mu} + (\mathbf{I - P}) g_{\varepsilon,\eta}.
		\end{equation}
		Meanwhile, it follows from \eqref{global,bdg} that as $\varepsilon \to 0$,
		\begin{equation}\label{Limit.Mac}
			(\rho_{\varepsilon,\eta}, u_{\varepsilon,\eta}, \theta_{\varepsilon,\eta}) \stackrel{\ast}{\rightharpoonup} (\rho, u, \theta) \quad \text{weakly-$\star$ in } L^\infty([0,\infty); H^N_x).
		\end{equation}
		To derive the fluid equations for $u$ and $\theta$, taking the $L^2_w$-inner products of \eqref{equ:g} with $\sqrt{\mu}$, $w \sqrt{\mu}$, and $\big( \frac{|w|^2}{3} - 1 \big) \sqrt{\mu}$ yields the local conservation laws:
		\begin{equation}\label{Eqn.Mac.}
			\left\{
			\begin{aligned}
				\partial_t \rho_{\varepsilon,\eta} +  \frac{1}{\varepsilon}\nabla_x\cdot u_{\varepsilon,\eta} =&0,\\
				\partial_t u_{\varepsilon,\eta} + \frac{1}{\varepsilon}\nabla_x\bigl(\rho_{\varepsilon,\eta}+\theta_{\varepsilon,\eta}\bigr)
				+& \frac{1}{\varepsilon}\nabla_x\cdot\left\langle \sqrt{\mu}\,\widehat{A},\,\mathcal{L}g_{\varepsilon,\eta}\right\rangle_{L^2_w}\\
				=&\left\langle \frac1\varepsilon \mathcal{L}_{\scriptscriptstyle\mathcal{R}} h_{\varepsilon,\eta}
				+\Gamma_{\scriptscriptstyle\mathcal{R}}\bigl(g_{\varepsilon,\eta},h_{\varepsilon,\eta}\bigr),\, w\sqrt{\mu} \right\rangle_{L^2_w},\\
				\partial_t \theta_{\varepsilon,\eta} + \frac{2}{3\varepsilon}\nabla_x\cdot u_{\varepsilon,\eta}
				+\frac{2}{3\varepsilon}&\nabla_x\cdot\left\langle \sqrt{\mu}\,\widehat{B},\,\mathcal{L}g_{\varepsilon,\eta}\right\rangle_{L^2_w}\\
				=\frac13&\left\langle \frac1\varepsilon \mathcal{L}_{\scriptscriptstyle\mathcal{R}} h_{\varepsilon,\eta}
				+\Gamma_{\scriptscriptstyle\mathcal{R}}\bigl(g_{\varepsilon,\eta},h_{\varepsilon,\eta}\bigr),\,
				\bigl(|w|^2-3\bigr)\sqrt{\mu} \right\rangle_{L^2_w}.
			\end{aligned}
			\right.
		\end{equation}
		where $\widehat{A}$ and $\widehat{B}$ are defined in \eqref{ABhat}.
		
		\paragraph{Incompressibility and Boussinesq Relation.}
		The continuity equation in \eqref{Eqn.Mac.} reads
		\begin{equation*}
			\nabla_x \cdot u_{\varepsilon,\eta} = -\varepsilon \partial_t \rho_{\varepsilon,\eta}.
		\end{equation*}
		By the uniform energy bound \eqref{global,bdg}, passing to the limit as $\varepsilon \to 0$ in the sense of distributions yields $\nabla_x \cdot u_{\varepsilon,\eta} \to 0$. In view of \eqref{Limit.Mac}, we obtain the incompressibility condition
		\begin{equation}\label{u.Diver.free.}
			\nabla_x \cdot u = 0.
		\end{equation}
		
		Similarly, from the momentum equation in \eqref{Eqn.Mac.}, we rewrite the pressure gradient term as
		\begin{equation*}
			\begin{aligned}
				\nabla_x (\rho_{\varepsilon,\eta} + \theta_{\varepsilon,\eta}) 
				&= -\varepsilon \partial_t u_{\varepsilon,\eta} - \nabla_x \cdot \left\langle \sqrt{\mu} \widehat{A}, \mathcal{L} (\mathbf{I - P}) g_{\varepsilon,\eta} \right\rangle_{L^2_w} \\
				&\quad + \left\langle \mathcal{L}_{\scriptscriptstyle\mathcal{R}} (\mathbf{I - P}_{\scriptscriptstyle 1}) h_{\varepsilon,\eta} + \varepsilon \Gamma_{\scriptscriptstyle\mathcal{R}}(g_{\varepsilon,\eta}, h_{\varepsilon,\eta}), \, w \sqrt{\mu} \right\rangle_{L^2_w}.
			\end{aligned}
		\end{equation*}
		Applying the uniform global energy bounds \eqref{global,bdh}--\eqref{global,bdg} and the dissipation bounds \eqref{global,dish}--\eqref{global,disg}, we deduce that $\nabla_x (\rho_{\varepsilon,\eta} + \theta_{\varepsilon,\eta}) \to 0$ in the sense of distributions as $\varepsilon \to 0$. This establishes the Boussinesq relation
		\begin{equation}\label{Bos}
			\nabla_x (\rho + \theta) = 0.
		\end{equation}
		
		\paragraph{Convergence of $\frac{3}{5}\theta_{\varepsilon,\eta} - \frac{2}{5}\rho_{\varepsilon,\eta}$ and $\mathcal{P}u_{\varepsilon,\eta}$.}
		Multiplying the third equation of \eqref{Eqn.Mac.} by $\frac{3}{5}$ and subtracting $\frac{2}{5}$ times the first equation, we eliminate the singular divergence term $\frac{1}{\varepsilon} \nabla_x \cdot u_{\varepsilon,\eta}$:
		\begin{equation}\label{theta-Equ}
			\begin{aligned}
				\partial_t \left( \frac{3}{5}\theta_{\varepsilon,\eta} - \frac{2}{5}\rho_{\varepsilon,\eta} \right) &+ \frac{2}{5\varepsilon} \nabla_x \cdot \left\langle \sqrt{\mu} \widehat{B}, \mathcal{L} g_{\varepsilon,\eta} \right\rangle_{L^2_w} \\
				&= \frac{1}{5} \left\langle \frac{1}{\varepsilon} \mathcal{L}_{\scriptscriptstyle\mathcal{R}} h_{\varepsilon,\eta} + \Gamma_{\scriptscriptstyle\mathcal{R}}(g_{\varepsilon,\eta}, h_{\varepsilon,\eta}), \, (|w|^2 - 3) \sqrt{\mu} \right\rangle_{L^2_w}.
			\end{aligned}
		\end{equation}
		For any $T > 0$, the uniform bound \eqref{global,bdg} implies
		\begin{equation*}
			\left\| \frac{3}{5}\theta_{\varepsilon,\eta} - \frac{2}{5}\rho_{\varepsilon,\eta} \right\|_{L^\infty([0,T]; H^N_x)} \lesssim 1.
		\end{equation*}
		Furthermore, by the energy and dissipation bounds \eqref{global,bdh}--\eqref{global,disg}, the time derivative is uniformly bounded in $H^{N-1}_x$:
		\begin{equation*}
			\begin{aligned}
				&	\left\| \partial_t \left( \frac{3}{5}\theta_{\varepsilon,\eta} - \frac{2}{5}\rho_{\varepsilon,\eta} \right) \right\|_{H^{N-1}_x}\\
				\lesssim &\frac{1}{\varepsilon} \left\| \nabla_x \cdot \left\langle \sqrt{\mu} \widehat{B}, \mathcal{L} (\mathbf{I - P}) g_{\varepsilon,\eta} \right\rangle_{L^2_w} \right\|_{H^{N-1}_x} \\
				&+ \left\| \left\langle \frac{1}{\varepsilon}\mathcal{L}_{\scriptscriptstyle\mathcal{R}} (\mathbf{I - P}_{\scriptscriptstyle 1}) h_{\varepsilon,\eta} + \Gamma_{\scriptscriptstyle\mathcal{R}}(g_{\varepsilon,\eta}, h_{\varepsilon,\eta}), \, (|w|^2 - 3) \sqrt{\mu} \right\rangle_{L^2_w} \right\|_{H^{N-1}_x} \\
				&\lesssim \left( 1 + \frac{\eta^{1/2}}{\varepsilon} \right) \lesssim 1.
			\end{aligned}
		\end{equation*}
		Therefore, by the Aubin--Lions--Simon Theorem (Appendix \ref{A-L-S}), there exists a function $\tilde{\theta} \in C(\mathbb{R}^+; H^{N-\sigma}_x) \cap L^{\infty}(\mathbb{R}^+; H^N_x)$ such that for any $\sigma \in (0,1]$, as $\varepsilon \to 0$,
		\begin{equation}\label{theta-rho-4}
			\left( \frac{3}{5}\theta_{\varepsilon,\eta} - \frac{2}{5}\rho_{\varepsilon,\eta} \right) \longrightarrow \tilde{\theta} \quad \text{strongly in } C(\mathbb{R}^+; H^{N-\sigma}_x).
		\end{equation}
		Recalling that $\tilde{\theta} = \frac{3}{5}\theta - \frac{2}{5}\rho$, $\theta = \left( \frac{3}{5}\theta - \frac{2}{5}\rho \right) + \frac{2}{5}(\rho + \theta)$, and applying the Boussinesq relation \eqref{Bos}, we conclude that $\tilde{\theta} = \theta$ and $\rho + \theta = 0$.
		
		Taking the Leray projection on the second
		equation of (\ref{Eqn.Mac.}) gives
		\begin{equation}\label{u-Equation}
			\partial_t\mathcal{P}u_{\varepsilon,\eta}+\frac{1}{\varepsilon}\mathcal{P}\nabla_x\cdot\left \langle\sqrt{\mu} \widehat{A},\mathcal{L}g_{\varepsilon,\eta}\right \rangle_{L^2_w}=\mathcal{P}\left \langle\frac{1}{\varepsilon}\mathcal L_{\scriptscriptstyle \mathcal R} h_{\varepsilon,\eta}+\Gamma_{\scriptscriptstyle \mathcal R} (g_{\varepsilon,\eta},h_{\varepsilon,\eta}),w\sqrt\mu\right \rangle_{L^2_w}.
		\end{equation}
		Here $\mathcal{P}$ is the Leray projection operator given by
		$\mathcal P=\mathcal I-\nabla_x\Delta^{-1}_x\nabla_x\cdot,
		$
		where $\mathcal I$ is the identical mapping.
		Similar arguments as above show that there is a divergence-free
		$\tilde{u}\in C(\mathbb{R}^+;H^{N-\sigma}_x)\cap L^{\infty}(\mathbb{R}^+;H^N_x)$ such that
		\begin{equation}\label{Conver,pu}
			\mathcal{P} u_{\varepsilon,\eta}\longrightarrow \tilde{u}\quad \text{strongly in}\ C(\mathbb{R}^+;H^{N-\sigma}_x),
		\end{equation}
		as $\varepsilon\rightarrow0$ for any $\sigma\in(0,1]$. Noting that $\tilde{u}=\mathcal{P} u$, we have $\mathcal{P} u=u$.
		
		Regarding the convergence of $\mathcal{P}^{\perp}$, we have
		\begin{equation}\label{Conver,p.perp.u}
			\mathcal{P}^{\perp}u_{\varepsilon,\eta}\longrightarrow0\quad \text{weakly-}\star\ \text{in}\ L^{\infty}([0,+\infty);H^N_x),
		\end{equation}
		as $\varepsilon\rightarrow0$.
		Indeed, we have
		\begin{equation}\label{Conver,p.perp.u1}
			\|\mathcal{P}^{\perp}u_{\varepsilon,\eta}\|_{L^{\infty}([0,+\infty);H^N_x)}\lesssim \|u_{\varepsilon,\eta}\|_{L^{\infty}([0,+\infty);H^N_x)}\lesssim 1,
		\end{equation}
		and $\mathcal{P}^{\perp}u_{\varepsilon,\eta}$ converge to 0 in the sense of distributions as $\varepsilon \rightarrow 0$, hence the relation \eqref{Conver,p.perp.u} holds. 
		
		\subsection{The Limiting Equations}		
		
	\subsubsection{The Equation of $F$}
			To derive the limiting equation for $F_{\varepsilon,\eta}$, we first present one auxiliary lemma which establishes the convergence of the collision integral $\mathcal{D}(F_{\varepsilon,\eta}, f_{\varepsilon,\eta})$ to the acceleration term in the Vlasov equation. The proof follows the strategy in \cite[Sec.~4.2.2, pp.~1728--1732]{MR3668954} and is collected in the appendix for completeness.
		
		\begin{lemma}\label{Vlasov equation.}
			Under the assumptions of Theorem \ref{Theom.2}, as $\varepsilon \to 0$,
			\begin{equation*}
				\frac{1}{\eta} \mathcal{D}(F_{\varepsilon,\eta}, f_{\varepsilon,\eta}) \longrightarrow \kappa \operatorname{div}_v \big( (v - u) F \big)
			\end{equation*}
			weakly in the sense of distributions, where $\kappa$ is defined in \eqref{kappa nu}. More precisely, for any $T > 0$ and any test function $Y$ satisfying \eqref{Y1}--\eqref{Y4},
			\begin{equation*}
				-\frac{1}{\eta} \int_0^T \int_{\mathbb{R}^3 \times \mathbb{R}^3} \mathcal{D}(F_{\varepsilon,\eta}, f_{\varepsilon,\eta}) Y \, \mathrm{d}v \mathrm{d}x \mathrm{d}t \xrightarrow{\varepsilon \to 0} \kappa \int_0^T \int_{\mathbb{R}^3 \times \mathbb{R}^3} F \nabla_v Y \cdot (v - u) \, \mathrm{d}v \mathrm{d}x \mathrm{d}t.
			\end{equation*}
		\end{lemma}
		
		We are now in a position to derive the limiting equation for $F_{\varepsilon,\eta}$. Recall that $F_{\varepsilon,\eta}$ satisfies
		\begin{equation*}
			\partial_t F_{\varepsilon,\eta} + v \cdot \nabla_x F_{\varepsilon,\eta} = \frac{1}{\eta} \mathcal{D}(F_{\varepsilon,\eta}, f_{\varepsilon,\eta}).
		\end{equation*}
		From the distributional convergence \eqref{F.conv}, it follows that as $\varepsilon \to 0$,
		\begin{equation*}
			\begin{aligned}
				\int_0^T \int_{\mathbb{R}^3 \times \mathbb{R}^3} \partial_t F_{\varepsilon,\eta} Y(t,x,v) \, \mathrm{d}v \mathrm{d}x \mathrm{d}t 
				&\longrightarrow -\int_{\mathbb{R}^3 \times \mathbb{R}^3} \rho_{\scriptscriptstyle F_0} \delta_{v=0} Y(0,x,v) \, \mathrm{d}v \mathrm{d}x \\
				&\quad -\int_0^T \int_{\mathbb{R}^3 \times \mathbb{R}^3} F \partial_t Y(t,x,v) \, \mathrm{d}v \mathrm{d}x \mathrm{d}t.
			\end{aligned}	
		\end{equation*}
		Similarly, by arguments analogous to those used for \eqref{F.conv}, we obtain
		\begin{equation*}
			\int_0^T \int_{\mathbb{R}^3 \times \mathbb{R}^3} v \cdot \nabla_x F_{\varepsilon,\eta} Y(t,x,v) \, \mathrm{d}v \mathrm{d}x \mathrm{d}t \longrightarrow -\int_0^T \int_{\mathbb{R}^3 \times \mathbb{R}^3} F v \cdot \nabla_x Y(t,x,v) \, \mathrm{d}v \mathrm{d}x \mathrm{d}t.
		\end{equation*}
		Furthermore, Lemma \ref{Vlasov equation.} provides the limit for the right-hand side:
		\begin{equation*}
			\int_0^T \int_{\mathbb{R}^3 \times \mathbb{R}^3} \frac{1}{\eta} \mathcal{D}(F_{\varepsilon,\eta}, f_{\varepsilon,\eta}) Y(t,x,v) \, \mathrm{d}v \mathrm{d}x \mathrm{d}t \longrightarrow -\kappa \int_0^T \int_{\mathbb{R}^3 \times \mathbb{R}^3} F \nabla_v Y \cdot (v - u) \, \mathrm{d}v \mathrm{d}x \mathrm{d}t.
		\end{equation*}
		
		Combining these limits, we establish that $F(t,x,v) = \rho_{\scriptscriptstyle F}(t,x) \delta_{v=0}$ satisfies the macroscopic Vlasov equation
		\begin{equation*}
			\partial_t F + v \cdot \nabla_x F = \kappa \operatorname{div}_v \big( (v - u) F \big),
		\end{equation*}
		subject to the initial condition $F(0,x,v) = \rho_{\scriptscriptstyle F_0} \delta_{v=0}$.
		
		\subsubsection{The Equations of $\theta$ and $u$}
		For convenience, we define
		\begin{equation}\label{tep0}
			\begin{aligned}
				&R_{\varepsilon,\eta,E}=-\varepsilon\left \langle\sqrt{\mu}\widehat{E},\partial_tg_{\varepsilon,\eta}\right \rangle_{L^2_w}+\left \langle\sqrt{\mu}\widehat{E},w\cdot \nabla_x(\mathbf{I - P})g_{\varepsilon,\eta}\right \rangle_{L^2_w}\\
				&+\left \langle\sqrt{\mu}\widehat{E},\Gamma((\mathbf{I - P})g_{\varepsilon,\eta},(\mathbf{I - P})g_{\varepsilon,\eta})\right \rangle_{L^2_w}+\left \langle\sqrt{\mu}\widehat{E},\Gamma((\mathbf{I - P})g_{\varepsilon,\eta},{\bf P}g_{\varepsilon,\eta})\right \rangle_{L^2_w}\\
				&+\left \langle\sqrt{\mu}\widehat{E},\Gamma({\bf P}g_{\varepsilon,\eta},(\mathbf{I - P})g_{\varepsilon,\eta})\right \rangle_{L^2_w},
			\end{aligned}
		\end{equation}
		for $E=A$ or $B$, where $A, B$ are defined in \eqref{AB}.
		Next, we note that 
		\begin{equation}\label{AB1}
			\mathcal{P}\left \langle\frac{1}{\varepsilon}\mathcal L_{\scriptscriptstyle \mathcal R} h_{\varepsilon,\eta}+\Gamma_{\scriptscriptstyle \mathcal R} (g_{\varepsilon,\eta},h_{\varepsilon,\eta}),w\sqrt\mu\right \rangle_{L^2_w}=\frac{1}{\varepsilon}\mathcal{P}\langle\mathcal R(f_{\varepsilon,\eta},F_{\varepsilon,\eta}),w\rangle_{L^2_w},
		\end{equation}
		and 
		\begin{equation}\label{AB2}
			\left \langle\frac{1}{\varepsilon}\mathcal L_{\scriptscriptstyle \mathcal R} h_{\varepsilon,\eta}+\Gamma_{\scriptscriptstyle \mathcal R} (g_{\varepsilon,\eta},h_{\varepsilon,\eta}),(|w|^2-3)\sqrt\mu\right \rangle_{L^2_w}=\frac{1}{\varepsilon}\langle\mathcal R(f_{\varepsilon,\eta},F_{\varepsilon,\eta}),|w|^2\rangle_{L^2_w}.
		\end{equation}
		We decompose $u_{\varepsilon,\eta}=\mathcal{P}u_{\varepsilon,\eta}+\mathcal{P}^{\perp}u_{\varepsilon,\eta}$ and denote $\tilde{\theta}_{\varepsilon,\eta}=\left (\frac{3}{5}\theta_{\varepsilon,\eta}-\frac{2}{5}\rho_{\varepsilon,\eta} \right )$. Then, following the standard calculation (cf. \cite{MR1115587,MR4296180}) together with \eqref{AB2}, the equation \eqref{theta-Equ} can be rewritten as
		\begin{equation*}
			\begin{aligned}
				\partial_t\tilde{\theta}_{\varepsilon,\eta}+\nabla_x\cdot\left (\mathcal{P} u_{\varepsilon,\eta} \tilde{\theta}_{\varepsilon,\eta}\right )-\widetilde{\kappa}\Delta_x\tilde{\theta}_{\varepsilon,\eta}=R_{\varepsilon,\eta,\theta}+\frac{1}{5\varepsilon}\langle\mathcal R(f_{\varepsilon,\eta},F_{\varepsilon,\eta}),|w|^2\rangle_{L^2_w},
			\end{aligned}
		\end{equation*}
		where $\widetilde{\kappa}$ is defined in \eqref{kappa nu} and
		\begin{equation}\label{tep1}
			\begin{aligned}
				R_{\varepsilon,\eta,\theta}&=\frac{2}{5}\nabla_x\cdot R_{\varepsilon,\eta,B}-\frac{2}{5}\nabla_x\cdot[\mathcal{P} u_{\varepsilon,\eta}(\rho_{\varepsilon,\eta}+\theta_{\varepsilon,\eta})]-\nabla_x\cdot\left (\mathcal{P}^\perp u_{\varepsilon,\eta} \tilde{\theta}_{\varepsilon,\eta}\right )\\
				&\quad-\frac{2}{5}\nabla_x\cdot[\mathcal{P}^\perp  u_{\varepsilon,\eta}(\rho_{\varepsilon,\eta}+\theta_{\varepsilon,\eta})]+\frac{2}{5}\widetilde{\kappa}\Delta_x(\rho_{\varepsilon,\eta}+\theta_{\varepsilon,\eta}).
			\end{aligned}
		\end{equation}
		For any $T>0$, let $\chi_2(t,x)$ be a test function satisfying \eqref{Y2}. From \eqref{tep0} and the global bounds \eqref{global,bdg} and \eqref{global,disg}, one can show that
		\begin{equation}\label{tep2}
			\begin{aligned}
				\int_0^T \int_{\mathbb{R}^3} R_{\varepsilon,\eta,E}(t,x)\chi_2(t,x)\mathrm{d}x\mathrm{d}t\longrightarrow 0,
			\end{aligned}
		\end{equation}
		as $\varepsilon \to 0$, where $E=A$ or $B$.
		For other terms in \eqref{tep1}, noting the convergence results \eqref{u.Diver.free.} and \eqref{Bos}, together with \eqref{tep2}, we have
		\begin{equation*}
			\begin{aligned}
				\int_0^T \int_{\mathbb{R}^3} R_{\varepsilon,\eta,\theta}(t,x)\chi_2(t,x)\mathrm{d}x\mathrm{d}t\longrightarrow 0,
			\end{aligned}
		\end{equation*}
		as $\varepsilon \to 0$. Setting $Y (t,x,v)=\chi_2(t,x)\cdot|v|^2$, from \eqref{Jonit.Cons.M,E.b} and Lemma \ref{Vlasov equation.}, we deduce that
		\[
		\begin{aligned}
			&\frac{1}{5\varepsilon}	\int_0^T \int_{\mathbb{R}^3} \langle \mathcal R(f_{\varepsilon,\eta},F_{\varepsilon,\eta}),|w|^2\rangle_{L^2_w} \chi_2(t,x)\mathrm{d}x \mathrm{d}t\\
			&=-\varepsilon\cdot\left \{\frac{1}{5\eta}\int_0^T \int_{\mathbb{R}^3\times\mathbb{R}^3}\mathcal D(F_{\varepsilon,\eta},f_{\varepsilon,\eta})Y (t,x,v)\mathrm{d}v\mathrm{d}x \mathrm{d}t\right \}\longrightarrow 0,
		\end{aligned}
		\]
		as $\varepsilon \to 0$. The convergences \eqref{theta-rho-4} and \eqref{Conver,pu} gives us that as \(\varepsilon \to 0\),
		\[
		\begin{aligned}
			\int_0^T \int_{\mathbb{R}^3} \partial_t\tilde{\theta}_{\varepsilon,\eta}(t,x)\chi_2(t,x)\mathrm{d}x\mathrm{d}t&\rightarrow-\int_{\mathbb{R}^3} \left (\frac{3}{5}\theta_0-\frac{2}{5}\rho_0\right ) \chi_2^0(x) \mathrm{d}x \\
			&\quad-\int_0^T \int_{\mathbb{R}^3} \theta(t,x) \partial_t \chi_2(t,x)\mathrm{d}x \mathrm{d}t,\\
			\int_0^T \int_{\mathbb{R}^3}\Delta_x\tilde{\theta}_{\varepsilon,\eta}\chi_2(t,x)\mathrm{d}x\mathrm{d}t&\rightarrow\int_0^T \int_{\mathbb{R}^3}\theta\Delta_x\chi_2(t,x)\mathrm{d}x\mathrm{d}t,
		\end{aligned}
		\]
		and
		\[
		\int_0^T \int_{\mathbb{R}^3}\nabla_x\cdot\left (\mathcal{P} u_{\varepsilon,\eta} \tilde{\theta}_{\varepsilon,\eta}\right )\chi_2(t,x)\mathrm{d}x\mathrm{d}t\rightarrow\int_0^T \int_{\mathbb{R}^3} u(t,x)\theta(t,x)\cdot \nabla_x\chi_2(t,x)\mathrm{d}x\mathrm{d}t.
		\]
		By collecting above convergence results, we obtain that $\theta\in C(\mathbb{R}^+; H^{N-\sigma}_x)\cap L^{\infty}(\mathbb{R}^+; H^N_x)$ satisfies the following equation
		$$
		\partial_t \theta+u \cdot \nabla_x\theta=\widetilde{\kappa}\Delta_x\theta,
		$$
		with initial data $\theta(0,x)=\frac{3}{5}\theta_0(x)-\frac{2}{5}\rho_0(x)$.
		Similarly, following the standard calculation (cf. \cite{MR1115587}) together with \eqref{AB1}, the equation \eqref{u-Equation} can be rewritten as
		\begin{equation*}
			\partial_t\mathcal{P}u_{\varepsilon,\eta}+\mathcal{P}\nabla_x\cdot\left (\mathcal{P}u_{\varepsilon,\eta}\otimes\mathcal{P}u_{\varepsilon,\eta}\right )-\widetilde{\nu}\Delta_x\mathcal{P}u_{\varepsilon,\eta}=R_{\varepsilon,\eta,u}+\frac{1}{\varepsilon}\mathcal{P}\langle\mathcal R(f_{\varepsilon,\eta},F_{\varepsilon,\eta}),w\rangle_{L^2_w},
		\end{equation*}
		where $\widetilde{\nu}$ is defined in \eqref{kappa nu} and
		\begin{align*}
			R_{\varepsilon,\eta,u}=\mathcal{P}\nabla_x\cdot R_{\varepsilon,\eta,A}-\mathcal{P}\nabla_x\cdot\left (\mathcal{P}u_{\varepsilon,\eta}\otimes\mathcal{P}^\perp u_{\varepsilon,\eta}+\mathcal{P}^\perp u_{\varepsilon,\eta}\otimes\mathcal{P}u_{\varepsilon,\eta}+\mathcal{P}^\perp u_{\varepsilon,\eta}\otimes\mathcal{P}^\perp u_{\varepsilon,\eta}\right ).
		\end{align*}	
		Similar to the above, we can take the vector-valued test function $\psi(t,x)$ satisfying 
		\eqref{Y4}, set $Y (t,x,v)=\psi(t,x)\cdot v$, and prove that as
		\(\varepsilon \to 0\), 
		\[
		\begin{aligned}
			&\int_0^T \int_{\mathbb{R}^3} \left (\partial_t\mathcal{P}u_{\varepsilon,\eta}+\mathcal{P}\nabla_x\cdot\left (\mathcal{P}u_{\varepsilon,\eta}\otimes\mathcal{P}u_{\varepsilon,\eta}\right )-\widetilde{\nu}\Delta_x\mathcal{P}u_{\varepsilon,\eta}
			\right )\cdot\psi(t,x)
			\mathrm{d}x\mathrm{d}t\\
			&\rightarrow-\int_{\mathbb{R}^3} \mathcal{P}u_0 \cdot \psi^0(x) \mathrm{d}x-\int_0^T \int_{\mathbb{R}^3} \left (u\cdot\partial_t\psi+u\otimes u:\nabla_x\psi-\widetilde{\nu}u\cdot\Delta_x \psi
			\right )\mathrm{d}x\mathrm{d}t,
		\end{aligned}
		\]
		and
		\[
		\int_0^T \int_{\mathbb{R}^3}R_{\varepsilon,\eta,u}(t,x)\psi(t,x)\mathrm{d}x\mathrm{d}t\rightarrow 0,
		\]
		and
		\[
		\begin{aligned}
			\frac{1}{\varepsilon}	\int_0^T \int_{\mathbb{R}^3} \mathcal{P}\langle \mathcal{R}(f_{\varepsilon,\eta},F_{\varepsilon,\eta}),w\rangle_{L^2_w}\cdot \psi\mathrm{d}x \mathrm{d}t
			&=-\frac{1}{\eta}\int_0^T \int_{\mathbb{R}^3\times\mathbb{R}^3}\mathcal D(F_{\varepsilon,\eta},f_{\varepsilon,\eta})Y \mathrm{d}v\mathrm{d}x \mathrm{d}t\\
			&\rightarrow\kappa\int_0^T\int_{\mathbb{R}^3}\psi(t,x)\cdot \left \{\int_{\mathbb{R}^3}(v-u) F\mathrm{d}v\right \}\mathrm{d}x\mathrm{d}t.
		\end{aligned}
		\]
		By collecting all above convergence results, we obtain that $u\in C(\mathbb{R}^+; H^{N-\sigma}_x)\cap L^{\infty}(\mathbb{R}^+; H^N_x)$ satisfies the following equation
		$$
		\partial_t u+\mathcal{P}\nabla_x\cdot\left (u\otimes u\right )-\widetilde{\nu}\Delta_xu
		=\kappa\int_{\mathbb{R}^3}(v-u)F\mathrm{d}v,
		$$
		with initial data $u(0,x)=\mathcal{P} u_0(x)$.

		\begin{appendices}
			\section{Jacobian of $J$}
			Let us state the Jacobian determinant $\Big|\frac{\partial(w'',v'')}{\partial(w,v)}\Big|$.
			
			\begin{lemma}\label{Jcb.}
				The Jacobian matrix $J$ satisfies
				\begin{equation*}
					\det J \equiv \left|\frac{\partial(w'',v'')}{\partial(w,v)}\right| = -1.
				\end{equation*}
			\end{lemma}
			
			\begin{proof}
				By calculation, we get
				\[
				J \equiv \frac{\partial(w'',v'')}{\partial(w,v)} =
			\begin{pmatrix}
				\mathbb{I}_3 - \frac{2}{1+\eta}\omega^T \omega & \frac{2\varepsilon}{1+\eta}\omega^T \omega \\
				\frac{2\eta}{\varepsilon(1+\eta)}\omega^T \omega & \mathbb{I}_3 - \frac{2\eta}{1+\eta}\omega^T \omega
			\end{pmatrix},
				\]
				where $\mathbb{I}_3 = \begin{pmatrix} 1 & 0 & 0 \\ 0 & 1 & 0 \\ 0 & 0 & 1 \end{pmatrix}$, $\omega=(\omega_1,\omega_2,\omega_3)$ with $|\omega|=1$, and $\omega^T = \begin{pmatrix} \omega_1 \\ \omega_2 \\ \omega_3 \end{pmatrix}$.
				Thus, we have 
				\begin{align*}
					\det J &\equiv \left|\frac{\partial(w'',v'')}{\partial(w,v)}\right| = \left|
					\begin{pmatrix}
						\mathbb{I}_3 - \frac{2}{1+\eta}\omega^T \omega & \frac{2\varepsilon}{1+\eta}\omega^T \omega \\
						\frac{2\eta}{\varepsilon(1+\eta)}\omega^T \omega & \mathbb{I}_3 - \frac{2\eta}{1+\eta}\omega^T \omega
					\end{pmatrix}
					\right| \\
					&= \left| \mathbb{I}_6 - 
					\begin{pmatrix}
						\frac{2}{1+\eta}\omega^T \\
						-\frac{2\eta}{\varepsilon(1+\eta)}\omega^T
					\end{pmatrix} 
					\begin{pmatrix}
						\omega & -\varepsilon \omega
					\end{pmatrix}
					\right| \\
					&= (1)^5 \times \left| \mathbb{I}_1 - 
					\begin{pmatrix}
						\omega & -\varepsilon \omega
					\end{pmatrix}
					\begin{pmatrix}
						\frac{2}{1+\eta}\omega^T \\
						-\frac{2\eta}{\varepsilon(1+\eta)}\omega^T
					\end{pmatrix}
					\right| = -1.
				\end{align*}
			\end{proof}	
			
			\section{Properties of the Operator $\mathcal{L}_{\scriptscriptstyle \mathcal{D}}$}\label{Appendix}
			In this section, we first study the key properties of the operator $\mathcal{L}_{\scriptscriptstyle \mathcal{D}}$.
			
			\begin{lemma}\label{Ker.K.d.}
				Let the linear operator $\mathcal{L}_{\scriptscriptstyle \mathcal{D}} = \nu_1 - \mathcal{K}_{\scriptscriptstyle \mathcal{D}}$ be defined in \eqref{Ope.LD.}--\eqref{Ope.K1.}.
				Assume that $k_d$ is the kernel of operator $\mathcal{K}_{\scriptscriptstyle \mathcal{D}}$, and write
				\[
				\mathcal{K}_{\scriptscriptstyle \mathcal{D}} h = \int_{\mathbb{R}^3} k_d(v,\xi) h(\xi) \mathrm{d}{\xi}.
				\]
				Then we have:
				\begin{enumerate}
					\item[\rm (i)] $k_d(v,\xi) = \varepsilon^2 \cdot \frac{(1+\eta)^2}{2^{3/2}\pi^{1/2}\eta^2} \frac{1}{|\xi-v|} \exp\left\{ -\left[ \frac{\varepsilon^2 |\xi-v|^2}{8\eta^2} + \frac{\varepsilon^2(|\xi|^2-|v|^2)^2}{8|\xi-v|^2} \right] \right\}$.
					\item[\rm (ii)] $\mathcal{K}_{\scriptscriptstyle \mathcal{D}}$ is a compact operator.
				\end{enumerate} 
			\end{lemma}
			
			\begin{proof}
				Using \eqref{velb}, one has
				\begin{equation}\label{K.1}
					\begin{aligned}
						\mathcal{K}_{\scriptscriptstyle \mathcal{D}}h 
						&= \frac{1}{\sqrt{\mu_{\varepsilon,\eta}(v)}} \int_{\mathbb{R}^3\times\mathbb{S}^2} \sqrt{\mu_{\varepsilon,\eta}(v'')} h(v'') \mu(w'') |(w-\varepsilon v)\cdot{\omega}| \mathrm{d}{\omega} \mathrm{d}w \\
						&= \frac{1}{(2\pi)^{3/2}} \int_{\mathbb{R}^3\times\mathbb{S}^2} e^{\frac{\varepsilon^2|v|^2}{4\eta}} e^{-\frac{\varepsilon^2|v''|^2}{4\eta}} e^{-\frac{|w''|^2}{2}} |(w-\varepsilon v)\cdot{\omega}| h(v'') \mathrm{d}{\omega} \mathrm{d}w \\
						&= \frac{1}{(2\pi)^{3/2}} \int_{\mathbb{R}^3\times\mathbb{S}^2} e^{-\frac{\varepsilon^2|v|^2}{4\eta}} e^{\frac{\varepsilon^2|v''|^2}{4\eta}} e^{-\frac{|w|^2}{2}} |(w-\varepsilon v)\cdot{\omega}| h(v'') \mathrm{d}{\omega} \mathrm{d}w.
					\end{aligned}
				\end{equation}
				Let $V := w - \varepsilon v$, then we decompose $V$ as follows:
				\[
				V = V_{\parallel} + V_{\perp},
				\]
				where
				\[
				V_{\parallel} = (V\cdot\omega)\omega, \quad V_{\perp} = V - (V\cdot\omega)\omega.
				\]
				For fixed ${\omega}$, the change of variables $V \rightarrow (V_{\parallel},V_{\perp})$ is a rotation with unit Jacobian. Hence we have (see page 43 in \cite{MR1379589})
				\begin{equation}\label{K.2}
					\mathrm{d}{\omega}\mathrm{d}w = \frac{2\mathrm{d}V_{\parallel}\mathrm{d}V_{\perp}}{|V_{\parallel}|^2}, \quad V_{\parallel}\in\mathbb{R}^3, \ V_{\perp}\in\mathbb{R}^2.
				\end{equation}
				Then, it holds that
			\begin{equation}\label{K.3}
					\begin{aligned}
						v'' &= v - \frac{2\eta}{1+\eta}\omega\left(\omega\cdot \left(v - \frac{1}{\varepsilon}w\right)\right)  = v + \left(\frac{2\eta}{1+\eta}\cdot\frac{1}{\varepsilon}\right)V_{\parallel}, \\
						w &= \varepsilon v + w - \varepsilon v = \varepsilon v + V_{\parallel} + V_{\perp}.
					\end{aligned}
				\end{equation}
				Now, we denote
				\[
				\xi := v + \left(\frac{2\eta}{1+\eta}\cdot\frac{1}{\varepsilon}\right)V_{\parallel}.
				\]
				Then, combining \eqref{K.1}, \eqref{K.2}, and \eqref{K.3} yields 
				\begin{equation*}
					\begin{aligned}
						\mathcal{K}_{\scriptscriptstyle \mathcal{D}}h 
						&= \frac{2}{(2\pi)^{3/2}} \int_{\mathbb{R}^3_{V_{\parallel}}\times \mathbb{R}^2_{V_{\perp}}} e^{-\frac{\varepsilon^2|v|^2}{4\eta}} e^{\frac{\varepsilon^2\left|v+\frac{2\eta}{1+\eta}\cdot\frac{1}{\varepsilon}V_{\parallel}\right|^2}{4\eta}} e^{-\frac{|\varepsilon v+V_{\parallel}+V_{\perp}|^2}{2}} \\
						&\quad \times h\left(v+\frac{2\eta}{1+\eta}\cdot\frac{1}{\varepsilon}V_{\parallel}\right) \frac{1}{|V_{\parallel}|} \mathrm{d}{\widetilde{w}} \mathrm{d}V_{\parallel} \\
						&= \varepsilon^2 \cdot \frac{(1+\eta)^2}{2^{5/2}\pi^{3/2}\eta^2} \int_{\mathbb{R}^3_\xi\times \mathbb{R}^2_{V_{\perp}}} e^{-\frac{\varepsilon^2|v|^2}{4\eta}} e^{\frac{\varepsilon^2|\xi|^2}{4\eta}} e^{-\frac{\left|V_{\perp}+\varepsilon\cdot\frac{1+\eta}{2\eta}\xi+\varepsilon\cdot\frac{\eta-1}{2\eta}v\right|^2}{2}} h(\xi) \frac{1}{|\xi-v|} \mathrm{d}V_{\perp} \mathrm{d}{\xi}.
					\end{aligned}
				\end{equation*}
				Moreover, because $V_{\perp}\cdot V_{\parallel} = 0$ by definition, it holds that
				\begin{equation}\label{A11}
					V_{\perp}\cdot\xi = V_{\perp}\cdot \left(v+\frac{2\eta}{1+\eta}\cdot\frac{1}{\varepsilon}V_{\parallel}\right) = V_{\perp}\cdot v.
				\end{equation}
				Therefore, by using \eqref{A11}, we obtain
				\begin{equation*}
					\mathcal{K}_{\scriptscriptstyle \mathcal{D}}h = \varepsilon^2 \cdot \frac{(1+\eta)^2}{2^{5/2}\pi^{3/2}\eta^2} \int_{\mathbb{R}^3_\xi} e^{-\frac{\varepsilon^2 |\xi-v|^2}{8\eta^2}} \frac{1}{|\xi-v|} h(\xi) \int_{\mathbb{R}^2_{V_{\perp}}} e^{-\frac{\left|V_{\perp}+\frac{\varepsilon}{2}\xi+\frac{\varepsilon}{2}v\right|^2}{2}} \mathrm{d}V_{\perp} \mathrm{d}{\xi}.
				\end{equation*}
				Finally, we denote
				\[
				\zeta := \xi + v.
				\]
				Now, we resolve
				\[
				\zeta = \zeta_1 + \zeta_2, \quad \zeta_1 \cdot \zeta_2 = 0,
				\]
				where
				\[
				\zeta_1 = \left(\zeta\cdot\frac{V_{\parallel}}{|V_{\parallel}|}\right)\frac{V_{\parallel}}{|V_{\parallel}|}, \quad \zeta_2 = \zeta - \left(\zeta\cdot\frac{V_{\parallel}}{|V_{\parallel}|}\right)\frac{V_{\parallel}}{|V_{\parallel}|}.
				\]
				A simple calculation shows that 
				\[
				|\zeta_1|^2 = \frac{(|\xi|^2-|v|^2)^2}{|\xi-v|^2}
				\]	
				and
				\[
				\left|V_{\perp}+\frac{\varepsilon}{2}\zeta\right|^2 = \frac{\varepsilon^2}{4}|\zeta_1|^2 + \left|V_{\perp}+\frac{\varepsilon}{2}\zeta_2\right|^2 + \varepsilon\zeta_1\cdot\left(V_{\perp}+\frac{\varepsilon}{2}\zeta_2\right).
				\]
				Noting that $\zeta_1\cdot\zeta_2 = 0$ by definition and that $\zeta_1 \cdot V_{\perp} = 0$ (since $\zeta_1$ has direction $V_{\parallel}$ and $V_{\parallel}\cdot V_{\perp} = 0$), we have
				\begin{align*}
					\mathcal{K}_{\scriptscriptstyle \mathcal{D}}h 
					&= \varepsilon^2 \cdot \frac{(1+\eta)^2}{2^{5/2}\pi^{3/2}\eta^2} \int_{\mathbb{R}^3_\xi} \frac{1}{|\xi-v|} e^{-\left\{\frac{\varepsilon^2 |\xi-v|^2}{8\eta^2} + \frac{\varepsilon^2(|\xi|^2-|v|^2)^2}{8|\xi-v|^2}\right\}} \int_{\mathbb{R}^2_{V_{\perp}}} e^{-\frac{\left|V_{\perp}+\frac{\varepsilon}{2}\zeta_2\right|^2}{2}} \mathrm{d}V_{\perp} \mathrm{d}{\xi} \\
					&= \varepsilon^2 \cdot \frac{(1+\eta)^2}{2^{3/2}\pi^{1/2}\eta^2} \int_{\mathbb{R}^3_\xi} \frac{1}{|\xi-v|} e^{-\left\{\frac{\varepsilon^2 |\xi-v|^2}{8\eta^2} + \frac{\varepsilon^2(|\xi|^2-|v|^2)^2}{8|\xi-v|^2}\right\}} \mathrm{d}{\xi}.
				\end{align*}
				The compactness of $\mathcal{K}_{\scriptscriptstyle \mathcal{D}}$ follows from a standard argument similar to Lemma 3.5.1 in \cite{MR1379589}.
			\end{proof}
			
			\begin{lemma}\label{L.D.sym}
				Let the linear operator $\mathcal{L}_{\scriptscriptstyle \mathcal{D}}$ be defined in \eqref{Ope.LD.}--\eqref{Ope.K1.}. Then we have:
				\begin{enumerate}
					\item[\rm (i)] $\langle\mathcal{L}_{\scriptscriptstyle \mathcal{D}} h, h\rangle_{L^2_v} \geqslant 0$.
					\item[\rm (ii)] $\langle\mathcal{L}_{\scriptscriptstyle \mathcal{D}} h, h\rangle_{L^2_v} = 0$ if and only if $h =a_{\scriptscriptstyle 1}(t,x)\sqrt{\mu_{\varepsilon,\eta}(v)}$.
					\item[\rm (iii)] $\mathcal{L}_{\scriptscriptstyle \mathcal{D}}$ is symmetric.
				\end{enumerate}
			\end{lemma}
			
			\begin{proof}
				To establish this lemma, we break the proof into four steps.
				
				\smallskip
				\noindent\textit{Step 1.} We aim to show that for all smooth functions $h(v)$, $g(w)$, and $\phi(v)$ that vanish at infinity, it holds that
				\begin{equation}\label{Iden.2}
					\int_{\mathbb{R}^3} \mathcal{D}(h,g)\phi(v) \mathrm{d}v = -\int_{\mathbb{R}^3} \mathcal{D}(h,g)\phi(v'') \mathrm{d}v.
				\end{equation}
				Indeed, by the velocity relations \eqref{velocitya}--\eqref{velocityb}, it holds that
			\begin{align*}
						v^{\prime\prime}&=v-\frac{2\eta}{1+\eta}\omega\left(\omega\cdot \left(v-\frac{1}{\varepsilon}w\right)\right),\\
					w^{\prime\prime}&=w+\frac{2}{1+\eta}\omega(\omega\cdot (\varepsilon v-w)),\\
					\varepsilon v'' - w'' &= \varepsilon v - w - 2\omega\cdot(\varepsilon v - w)\omega,
				\end{align*}
				and
				\[n v'' - w'') = \omega\cdot(\varepsilon v - w) - 2\omega\cdot(\varepsilon v - w) = -\omega\cdot(\varepsilon v - w).
				\]
				We can invert these to get
			\begin{align*}
					v &= v'' + \frac{2\eta}{1+\eta}{\omega}\left( \omega\cdot\left(v - \frac{1}{\varepsilon}w\right)\right)  = v'' - \frac{2\eta}{1+\eta}{\omega}\left( \omega\cdot\left(v'' - \frac{1}{\varepsilon}w''\right)\right)  \equiv v(v'',w''), \\
					w &= w'' - \frac{2}{1+\eta}{\omega}\left( \omega\cdot(\varepsilon v - w)\right)  = w'' + \frac{2}{1+\eta}{\omega}\left( \omega\cdot(\varepsilon v'' - w'')\right)  \equiv w(v'',w'').
				\end{align*}
				By changing $v'' \rightarrow v$ and $w'' \rightarrow w$, it follows that
				\begin{align*}
					\int_{\mathbb{R}^3} \mathcal{D}(h,g)\phi(v) \mathrm{d}v
					&= \int_{(\mathbb{R}^3)^2\times\mathbb{S}^2} \left[ h(v'')g(w'') - h(v)g(w) \right] |(w - \varepsilon v)\cdot{\omega}| \phi(v) \mathrm{d}\omega \mathrm{d}w \mathrm{d}v \\
					&= \int_{(\mathbb{R}^3)^2\times\mathbb{S}^2} \left[ h(v)g(w) - h(v'')g(w'') \right] |(w - \varepsilon v)\cdot{\omega}| \phi(v'') \mathrm{d}\omega \mathrm{d}w \mathrm{d}v \\
					&= -\int_{\mathbb{R}^3} \mathcal{D}(h,g)\phi(v'') \mathrm{d}v.
				\end{align*}
				Therefore, we get
				\[
				\int_{\mathbb{R}^3} \mathcal{D}(h,g)\phi(v) \mathrm{d}v = 0 \quad \text{if} \quad \phi(v) = \phi(v'').
				\]
				
				\noindent\textit{Step 2.} Next, we prove that if $\phi$ is positive and continuous, and $\phi(v) = \phi(v'')$, then there exists a function $a_{\scriptscriptstyle 1}(t,x) \in \mathbb{R}$ that is independent of $v$ such that $\phi(v) \equiv a_{\scriptscriptstyle 1}(t,x)$.
				
				In fact, fixing $v$, for any $v_0 \in \mathbb{R}^3$ with $v_0 \neq v$, we define
				\[
			v^{\prime\prime}:=v-\frac{2\eta}{1+\eta}\omega\left(\omega\cdot \left(v-\frac{1}{\varepsilon}w\right)\right),
				\]
				where
				\[
					\omega:=\frac{v-v_0}{|v-v_0|},\quad
					w:= \varepsilon \cdot \frac{1+\eta}{2\eta} \left(v_0 + \frac{\eta-1}{\eta+1}v\right).
					\]
					A direct calculation gives $v'' = v_0$. Hence, by the assumption $\phi(v) = \phi(v'')$,
					we have $\phi(v) = \phi(v_0)$. For $v_0 = v$, this is trivial. Since $\phi$ is continuous,
					this equality holds for all $v_0 \in \mathbb{R}^3$. Therefore, for fixed $(t,x)$, $\phi(v)$ is
					constant in $v$. Thus there exists $a_{\scriptscriptstyle 1}(t,x) \in \mathbb{R}$ such that $\phi(v) \equiv a_{\scriptscriptstyle 1}(t,x)$.
				
				\smallskip
				\noindent\textit{Step 3.} Then, we prove that $\mathcal{L}_{\scriptscriptstyle \mathcal{D}}$ is non-negative.
				This is equivalent to showing that
				\[
				\int_{\mathbb{R}^3} \frac{h(v)}{\sqrt{\mu_{\varepsilon,\eta}(v)}} \mathcal{D}(\sqrt{\mu_{\varepsilon,\eta}}h, \mu) \mathrm{d}v \leqslant 0.
				\]
				Adding the two choices in identity \eqref{Iden.2} with $\phi = \frac{h}{\sqrt{\mu_{\varepsilon,\eta}}}$, we obtain
				\begin{align*}
					&\int_{\mathbb{R}^3} \frac{2h(v)}{\sqrt{\mu_{\varepsilon,\eta}(v)}} \mathcal{D}(\sqrt{\mu_{\varepsilon,\eta}}h, \mu) \mathrm{d}v \\
					&= \int_{\mathbb{R}^3\times\mathbb{R}^3\times\mathbb{S}^2} \left[ \mu(w'')\sqrt{\mu_{\varepsilon,\eta}(v'')} h(v'') - \mu(w)\sqrt{\mu_{\varepsilon,\eta}(v)} h(v) \right] \\
					&\quad \times \left[ \frac{h(v)}{\sqrt{\mu_{\varepsilon,\eta}(v)}} - \frac{h(v'')}{\sqrt{\mu_{\varepsilon,\eta}(v'')}} \right] |(w-\varepsilon v)\cdot{\omega}| \mathrm{d}\omega \mathrm{d}w \mathrm{d}v \\
					&= -\int_{\mathbb{R}^3\times\mathbb{R}^3\times\mathbb{S}^2} \mu(w)\mu_{\varepsilon,\eta}(v) \left[ \frac{h(v)}{\sqrt{\mu_{\varepsilon,\eta}(v)}} - \frac{h(v'')}{\sqrt{\mu_{\varepsilon,\eta}(v'')}} \right]^2 |(w-\varepsilon v)\cdot{\omega}| \mathrm{d}\omega \mathrm{d}w \mathrm{d}v \leqslant 0.
				\end{align*}
				Of course, we can write
				\begin{align*}
					&-\mathcal{L}_{\scriptscriptstyle \mathcal{D}}h 
					= \frac{1}{\sqrt{\mu_{\varepsilon,\eta}}} \mathcal{D}(\sqrt{\mu_{\varepsilon,\eta}}h, \mu) \\
					=& \frac{1}{\sqrt{\mu_{\varepsilon,\eta}(v)}} \int_{\mathbb{R}^3\times\mathbb{S}^2} \left[ \mu(w'')\sqrt{\mu_{\varepsilon,\eta}(v'')} h(v'') - \mu(w)\sqrt{\mu_{\varepsilon,\eta}(v)} h(v) \right] |(w-\varepsilon v)\cdot{\omega}| \mathrm{d}\omega \mathrm{d}w \\
					=& \int_{\mathbb{R}^3_w\times\mathbb{S}^2} \frac{1}{\sqrt{\mu_{\varepsilon,\eta}(v)}} \mu(w)\mu_{\varepsilon,\eta}(v) \left[ \frac{h(v)}{\sqrt{\mu_{\varepsilon,\eta}(v)}} - \frac{h(v'')}{\sqrt{\mu_{\varepsilon,\eta}(v'')}} \right] |(w-\varepsilon v)\cdot{\omega}| \mathrm{d}\omega \mathrm{d}w.
				\end{align*}
				Suppose that $h(v) =a_{\scriptscriptstyle 1}(t,x)\sqrt{\mu_{\varepsilon,\eta}(v)}$, then $\mathcal{L}_{\scriptscriptstyle \mathcal{D}}h = 0$. In fact,
				\[
				\int_{\mathbb{R}^3_v} h \mathcal{L}_{\scriptscriptstyle \mathcal{D}}h \mathrm{d}v = 0 \quad \text{if and only if} \quad h(v) = a_{\scriptscriptstyle 1}(t,x)\sqrt{\mu_{\varepsilon,\eta}(v)} \quad \text{a.e.}
				\]
				
				\noindent\textit{Step 4.} Finally, we prove that $\mathcal{L}_{\scriptscriptstyle \mathcal{D}}$ is symmetric.
				Recall that $\mathcal{L}_{\scriptscriptstyle \mathcal{D}}h = \nu_1 h - \mathcal{K}_{\scriptscriptstyle \mathcal{D}} h$. By Lemma \ref{Ker.K.d.}, we have $k_d(v,\xi) = k_d(\xi,v)$. Hence,
				\[
				\int_{\mathbb{R}^3} \widehat{h} \cdot \mathcal{L}_{\scriptscriptstyle \mathcal{D}}h \mathrm{d}v = \int_{\mathbb{R}^3} h \cdot \mathcal{L}_{\scriptscriptstyle \mathcal{D}}\widehat{h} \mathrm{d}v.
				\]
			\end{proof}	
			
			Now, we prove the operator $\mathcal{L}_{\scriptscriptstyle \mathcal{D}}$ is locally coercive.
			
			\begin{lemma}\label{L.D.j}
				Let the linear operator $\mathcal{L}_{\scriptscriptstyle \mathcal{D}}$ be defined in \eqref{Ope.LD.}--\eqref{Ope.K1.}. Then there is a constant $\lambda_1 > 0$ such that
				\[
				\langle\mathcal{L}_{\scriptscriptstyle \mathcal{D}} h, h\rangle_{L^2_v} \geqslant \lambda_1 {|\mathbf{(I-P_1)}h|}^{2}_{\nu_1}.
				\]
			\end{lemma}
			
			\begin{proof}
				Assume the contrary: there is a sequence of normalized functions $\{h_n(v)\}$ which satisfy
				\begin{equation}\label{HN}
					\int_{\mathbb{R}^3} h_n(v)\sqrt{\mu_{\varepsilon,\eta}(v)} \mathrm{d}v = 0,
				\end{equation}
				such that $\langle \nu_1 h_n, h_n \rangle_{L^2_v} = 1$. Moreover,
				\begin{equation}\label{L-D-n}
					\langle\mathcal{L}_{\scriptscriptstyle \mathcal{D}} h_n, h_n\rangle_{L^2_v} \leqslant \frac{1}{n}.
				\end{equation}
				Let the weak limit of $\{h_n\}$ (up to a subsequence) with respect to the inner product $\langle\nu_1 \cdot, \cdot \rangle_{L^2_v}$ be $h_0$ with 
				\begin{equation}\label{nu-1-C}
					\langle\nu_1 h_0, h_0 \rangle_{L^2_v} \leqslant 1.
				\end{equation}
				Notice that
				\begin{equation}\label{L-D}
					\langle\mathcal{L}_{\scriptscriptstyle \mathcal{D}} h_n, h_n\rangle_{L^2_v} = \langle \nu_1 h_n, h_n\rangle_{L^2_v} - \langle \mathcal{K}_{\scriptscriptstyle \mathcal{D}} h_n, h_n \rangle_{L^2_v} = 1 - \langle \mathcal{K}_{\scriptscriptstyle \mathcal{D}} h_n, h_n \rangle_{L^2_v}.
				\end{equation}
				By Lemma \ref{Ker.K.d.}, we know $\mathcal{K}_{\scriptscriptstyle \mathcal{D}}$ is compact in $L^2(\mathbb{R}^3_v)$, thus
				\[
				\lim_{n\to\infty} \left\langle \nu_1 [\mathcal{K}_{\scriptscriptstyle \mathcal{D}} h_n - \mathcal{K}_{\scriptscriptstyle \mathcal{D}} h_0], [\mathcal{K}_{\scriptscriptstyle \mathcal{D}} h_n - \mathcal{K}_{\scriptscriptstyle \mathcal{D}} h_0] \right\rangle_{L^2_v} = 0.
				\]
				Therefore,
				\begin{equation}\label{Lim.K.D}
					\lim_{n\to\infty} \langle \mathcal{K}_{\scriptscriptstyle \mathcal{D}} h_n, h_n\rangle_{L^2_v} = \langle\mathcal{K}_{\scriptscriptstyle \mathcal{D}} h_0, h_0\rangle_{L^2_v}.
				\end{equation}
				By combining Lemma \ref{L.D.sym}, \eqref{L-D-n}, \eqref{nu-1-C}, \eqref{L-D}, and \eqref{Lim.K.D}, we obtain in the limit $n\to\infty$:
				\begin{align*}
					0 \leqslant \langle\mathcal{L}_{\scriptscriptstyle \mathcal{D}} h_0, h_0\rangle_{L^2_v} 
					&= \langle \nu_1 h_0, h_0\rangle_{L^2_v} - \langle\mathcal{K}_{\scriptscriptstyle \mathcal{D}} h_0, h_0\rangle_{L^2_v} \\
					&\leqslant 1 - \langle\mathcal{K}_{\scriptscriptstyle \mathcal{D}} h_0, h_0\rangle_{L^2_v} \\
					&\leqslant 1 - \lim_{n\to\infty} \langle \mathcal{K}_{\scriptscriptstyle \mathcal{D}} h_n, h_n\rangle_{L^2_v} \\
					&\leqslant \lim_{n\to\infty} \langle\mathcal{L}_{\scriptscriptstyle \mathcal{D}} h_n, h_n\rangle_{L^2_v} \leqslant \lim_{n\to\infty} \frac{1}{n} = 0.
				\end{align*}
				Therefore, $\langle\mathcal{L}_{\scriptscriptstyle \mathcal{D}} h_0, h_0\rangle_{L^2_v} = 0$, which is equivalent to
				\[
				0 = \{1 - \langle\nu_1 h_0, h_0\rangle_{L^2_v}\} + \{\langle \nu_1 h_0, h_0\rangle_{L^2_v} - \langle\mathcal{K}_{\scriptscriptstyle \mathcal{D}} h_0, h_0\rangle_{L^2_v}\}.
				\]
				Since both terms are non-negative, $1 = \langle \nu_1 h_0, h_0\rangle_{L^2_v}$ and
				\[
				\langle\mathcal{L}_{\scriptscriptstyle \mathcal{D}} h_0, h_0\rangle_{L^2_v} = \langle \nu_1 h_0, h_0\rangle_{L^2_v} - \langle\mathcal{K}_{\scriptscriptstyle \mathcal{D}} h_0, h_0\rangle_{L^2_v} = 0.
				\]
				Therefore, $h_0 =a_{\scriptscriptstyle 1}\sqrt{\mu_{\varepsilon,\eta}}$. On the other hand, letting $n\to\infty$ in \eqref{HN} implies that $h_0$ also satisfies \eqref{HN}. Hence $a_{\scriptscriptstyle 1} = 0$, which contradicts $\langle \nu_1 h_0, h_0\rangle_{L^2_v} = 1$.
			\end{proof}	
			
			\section{Aubin-Lions-Simon Theorem}\label{A-L-S}
			Now we recall the following Aubin-Lions-Simon Theorem, which is presented in Theorem II.5.16 of \cite{MR2986590} (see also Lemma 4.2 in \cite{MR4296180}) and is a fundamental compactness tool in the study of nonlinear evolution problems.
			
			\begin{lemma}[Aubin-Lions-Simon Theorem]
				Let $B_0 \subset B_1 \subset B_2$ be three Banach spaces with $B_0$ compactly embedded in $B_1$ and $B_1$ continuously embedded in $B_2$. Let $p, r \in [1, \infty]$ and $T > 0$. Define
				\[
				E_{p,r} = \left\{ u \in L^p([0,T]; B_0) \ \big|\ \partial_t u \in L^r([0,T]; B_2) \right\}.
				\]
				\begin{enumerate}
					\item[\rm 1)] If $p < \infty$, then $E_{p,r}$ is compactly embedded in $L^p([0,T]; B_1)$.
					\item[\rm 2)] If $p = \infty$ and $r > 1$, then $E_{p,r}$ is compactly embedded in $C([0,T]; B_1)$.
				\end{enumerate}
			\end{lemma}
			
			\section{Proof of Lemma \ref{Vlasov equation.}}
			Before we give the proof of Lemma \ref{Vlasov equation.}, we first state the following preparatory lemma.
			\begin{lemma}\label{lemma6.2}
				For any $T > 0$, let $Y$ be a test function satisfying \eqref{Y1}--\eqref{Y4}. Then, as $\varepsilon \to 0$,
				\begin{equation*}
					\int_0^T \int_{\mathbb{R}^3 \times \mathbb{R}^3} F_{\varepsilon,\eta} \nabla_v Y \cdot u_{\varepsilon,\eta} \, \mathrm{d}v \mathrm{d}x \mathrm{d}t \longrightarrow \int_0^T \int_{\mathbb{R}^3 \times \mathbb{R}^3} F \nabla_v Y \cdot u \, \mathrm{d}v \mathrm{d}x \mathrm{d}t.
				\end{equation*}
			\end{lemma}
			
			\begin{proof}
				We decompose the difference into three terms:
				\begin{equation}\label{D.}
					\begin{aligned}
						&\int_0^T \int_{\mathbb{R}^3 \times \mathbb{R}^3} \Big( F_{\varepsilon,\eta} \nabla_v Y \cdot u_{\varepsilon,\eta} - F \nabla_v Y \cdot u \Big) \, \mathrm{d}v \mathrm{d}x \mathrm{d}t \\
						&\quad = \int_0^T \int_{\mathbb{R}^3 \times \mathbb{R}^3} F_{\varepsilon,\eta} \nabla_v Y \cdot (\mathcal{P} u_{\varepsilon,\eta} - u) \, \mathrm{d}v \mathrm{d}x \mathrm{d}t \\
						&\quad\quad + \int_0^T \int_{\mathbb{R}^3 \times \mathbb{R}^3} F_{\varepsilon,\eta} \nabla_v Y \cdot \mathcal{P}^{\perp} u_{\varepsilon,\eta} \, \mathrm{d}v \mathrm{d}x \mathrm{d}t \\
						&\quad\quad + \int_0^T \int_{\mathbb{R}^3 \times \mathbb{R}^3} (F_{\varepsilon,\eta} - F) \nabla_v Y \cdot u \, \mathrm{d}v \mathrm{d}x \mathrm{d}t \\
						&\quad =: \mathcal{N}_1 + \mathcal{N}_2 + \mathcal{N}_3.
					\end{aligned}
				\end{equation}
				\textbf{Estimation of $\mathcal{N}_1$.} From the uniform energy bound \eqref{global,bdh} and the convergence \eqref{Conver,pu}, we obtain
				\begin{equation}\label{N1}
					\begin{aligned}
						|\mathcal{N}_1|
						&\lesssim \left| \int_0^T \int_{\mathbb{R}^3 \times \mathbb{R}^3} \sqrt{\mu_{\varepsilon,\eta}} \, h_{\varepsilon,\eta} (\psi + 2\chi_2 v) \cdot (\mathcal{P} u_{\varepsilon,\eta} - u) \, \mathrm{d}v \mathrm{d}x \mathrm{d}t \right| \\
						&\lesssim \|h_{\varepsilon,\eta}\|_{L^\infty([0,T]; L^2_{x,v})} \|(\psi, \chi_2)\|_{L^1([0,T]; L^\infty_x)} \left\| \sqrt{\mu_{\varepsilon,\eta}} (1 + |v|) \right\|_{L^2_v} \|\mathcal{P} u_{\varepsilon,\eta} - u\|_{C([0,T]; L^2_x)} \\
						&\lesssim \|\mathcal{P} u_{\varepsilon,\eta} - u\|_{C([0,T]; L^2_x)} \xrightarrow{\varepsilon \to 0} 0.
					\end{aligned}
				\end{equation}
				\textbf{Estimation of $\mathcal{N}_2$.} Using the orthogonal decomposition of $F_{\varepsilon,\eta}$ in \eqref{expres-F}, we split $\mathcal{N}_2$ into
				\begin{equation}\label{N-2}
					\begin{aligned}
						\mathcal{N}_2 &= \int_0^T \int_{\mathbb{R}^3 \times \mathbb{R}^3} (\rho_{\scriptscriptstyle F_{\varepsilon,\eta}} - \rho_{\scriptscriptstyle F}) \mu_{\varepsilon,\eta} \nabla_v Y \cdot \mathcal{P}^{\perp} u_{\varepsilon,\eta} \, \mathrm{d}v \mathrm{d}x \mathrm{d}t \\
						&\quad + \int_0^T \int_{\mathbb{R}^3 \times \mathbb{R}^3} \rho_{\scriptscriptstyle F} \mu_{\varepsilon,\eta} \nabla_v Y \cdot \mathcal{P}^{\perp} u_{\varepsilon,\eta} \, \mathrm{d}v \mathrm{d}x \mathrm{d}t \\
						&\quad + \int_0^T \int_{\mathbb{R}^3 \times \mathbb{R}^3} \sqrt{\mu_{\varepsilon,\eta}} \, (\mathbf{I - P}_{\scriptscriptstyle 1}) h_{\varepsilon,\eta} \nabla_v Y \cdot \mathcal{P}^{\perp} u_{\varepsilon,\eta} \, \mathrm{d}v \mathrm{d}x \mathrm{d}t.
					\end{aligned}
				\end{equation}
				For the first term on the right-hand side of \eqref{N-2}, the uniform bound \eqref{Conver,p.perp.u1} and the strong convergence \eqref{Str.rho.F} yield
				\begin{equation*}
					\begin{aligned}
						&\left| \int_0^T \int_{\mathbb{R}^3 \times \mathbb{R}^3} (\rho_{\scriptscriptstyle F_{\varepsilon,\eta}} - \rho_{\scriptscriptstyle F}) \mu_{\varepsilon,\eta} (\psi + 2\chi_2 v) \cdot \mathcal{P}^{\perp} u_{\varepsilon,\eta} \, \mathrm{d}v \mathrm{d}x \mathrm{d}t \right| \\
						&\quad \lesssim \|\rho_{\scriptscriptstyle F_{\varepsilon,\eta}} - \rho_{\scriptscriptstyle F}\|_{C([0,T]; L^2_x)} \|\mu_{\varepsilon,\eta} (1 + |v|)\|_{L^1_v} \|(\psi, \chi_2)\|_{L^1([0,T]; L^\infty_x)} \|\mathcal{P}^{\perp} u_{\varepsilon,\eta}\|_{L^\infty([0,T]; L^2_x)} \\
						&\quad \lesssim \|\rho_{\scriptscriptstyle F_{\varepsilon,\eta}} - \rho_{\scriptscriptstyle F}\|_{C([0,T]; L^2_x)} \xrightarrow{\varepsilon \to 0} 0,
					\end{aligned}
				\end{equation*}
				which implies
				\begin{equation}\label{N21}
					\int_0^T \int_{\mathbb{R}^3 \times \mathbb{R}^3} (\rho_{\scriptscriptstyle F_{\varepsilon,\eta}} - \rho_{\scriptscriptstyle F}) \mu_{\varepsilon,\eta} \nabla_v Y \cdot \mathcal{P}^{\perp} u_{\varepsilon,\eta} \, \mathrm{d}v \mathrm{d}x \mathrm{d}t \xrightarrow{\varepsilon \to 0} 0.
				\end{equation}
				For the second term on the right-hand side of \eqref{N-2}, integrating over $v$ gives
				\begin{equation*}
					\int_0^T \int_{\mathbb{R}^3 \times \mathbb{R}^3} \rho_{\scriptscriptstyle F} \mu_{\varepsilon,\eta} \nabla_v Y \cdot \mathcal{P}^{\perp} u_{\varepsilon,\eta} \, \mathrm{d}v \mathrm{d}x \mathrm{d}t = \int_0^T \int_{\mathbb{R}^3} \rho_{\scriptscriptstyle F} \psi \cdot \mathcal{P}^{\perp} u_{\varepsilon,\eta} \, \mathrm{d}x \mathrm{d}t.
				\end{equation*}
				Since $\psi \in C^1([0,T]; C_c^\infty(\mathbb{R}^3_x))$ and $\rho_{\scriptscriptstyle F} \in L^\infty([0,\infty); H^N_x)$, it holds that $\rho_{\scriptscriptstyle F} \psi \in L^1([0,T]; H^{-N}_x)$. Hence, from convergence \eqref{Conver,p.perp.u}, we deduce
				\begin{equation}\label{N22}
					\int_0^T \int_{\mathbb{R}^3 \times \mathbb{R}^3} \rho_{\scriptscriptstyle F} \mu_{\varepsilon,\eta} \nabla_v Y \cdot \mathcal{P}^{\perp} u_{\varepsilon,\eta} \, \mathrm{d}v \mathrm{d}x \mathrm{d}t \xrightarrow{\varepsilon \to 0} 0.
				\end{equation}
				For the third term on the right-hand side of \eqref{N-2}, applying the energy dissipation bound \eqref{global,dish} yields
				\begin{equation}\label{N23}
					\begin{aligned}
						&\left| \int_0^T \int_{\mathbb{R}^3 \times \mathbb{R}^3} \sqrt{\mu_{\varepsilon,\eta}} \, (\mathbf{I - P}_{\scriptscriptstyle 1}) h_{\varepsilon,\eta} (\psi + 2\chi_2 v) \cdot \mathcal{P}^{\perp} u_{\varepsilon,\eta} \, \mathrm{d}v \mathrm{d}x \mathrm{d}t \right| \\
						&\quad \lesssim \|(\mathbf{I - P}_{\scriptscriptstyle 1}) h_{\varepsilon,\eta}\|_{L^2([0,T]; L^2_{x,v})} \|(\psi, \chi_2)\|_{L^2([0,T]; L^\infty_x)} \|\mathcal{P}^{\perp} u_{\varepsilon,\eta}\|_{L^\infty([0,T]; L^2_x)} \left\| \sqrt{\mu_{\varepsilon,\eta}} (1 + |v|) \right\|_{L^2_v} \\
						&\quad \lesssim \|(\mathbf{I - P}_{\scriptscriptstyle 1}) h_{\varepsilon,\eta}\|_{L^2([0,T]; L^2_{x,v})} \lesssim \eta^{1/2} \xrightarrow{\varepsilon \to 0} 0.
					\end{aligned}
				\end{equation}
				Substituting \eqref{N21}, \eqref{N22}, and \eqref{N23} into \eqref{N-2}, we obtain
				\begin{equation}\label{N2}
					\mathcal{N}_2 \xrightarrow{\varepsilon \to 0} 0.
				\end{equation}
				\textbf{Estimation of $\mathcal{N}_3$.} By \eqref{expres-F} and \eqref{F.def}, we write $\mathcal{N}_3$ as
				\begin{equation}\label{N-3}
					\begin{aligned}
						\mathcal{N}_3 &= \int_0^T \int_{\mathbb{R}^3 \times \mathbb{R}^3} \rho_{\scriptscriptstyle F_{\varepsilon,\eta}} \nabla_v Y \cdot u \, (\mu_{\varepsilon,\eta} - \delta_{v=0}) \, \mathrm{d}v \mathrm{d}x \mathrm{d}t \\
						&\quad + \int_0^T \int_{\mathbb{R}^3 \times \mathbb{R}^3} (\rho_{\scriptscriptstyle F_{\varepsilon,\eta}} - \rho_{\scriptscriptstyle F}) \nabla_v Y \cdot u \, \delta_{v=0} \, \mathrm{d}v \mathrm{d}x \mathrm{d}t \\
						&\quad + \int_0^T \int_{\mathbb{R}^3 \times \mathbb{R}^3} \sqrt{\mu_{\varepsilon,\eta}} \, (\mathbf{I - P}_{\scriptscriptstyle 1}) h_{\varepsilon,\eta} \nabla_v Y \cdot u \, \mathrm{d}v \mathrm{d}x \mathrm{d}t.
					\end{aligned}
				\end{equation}
				Direct integration shows that the first term on the RHS of \eqref{N-3} vanishes identically:
				\begin{equation}\label{N31}
					\begin{aligned}
						&\int_0^T \int_{\mathbb{R}^3 \times \mathbb{R}^3} \rho_{\scriptscriptstyle F_{\varepsilon,\eta}} \nabla_v Y \cdot u \, (\mu_{\varepsilon,\eta} - \delta_{v=0}) \, \mathrm{d}v \mathrm{d}x \mathrm{d}t \\
						= &\int_0^T \int_{\mathbb{R}^3 \times \mathbb{R}^3} \rho_{\scriptscriptstyle F_{\varepsilon,\eta}} \psi \cdot u \, (\mu_{\varepsilon,\eta} - \delta_{v=0}) \, \mathrm{d}v \mathrm{d}x \mathrm{d}t \\
						&+ 2 \int_0^T \int_{\mathbb{R}^3 \times \mathbb{R}^3} \rho_{\scriptscriptstyle F_{\varepsilon,\eta}} \chi_2 v \cdot u \, (\mu_{\varepsilon,\eta} - \delta_{v=0}) \, \mathrm{d}v \mathrm{d}x \mathrm{d}t = 0.
					\end{aligned}
				\end{equation}
				For the second term, the strong convergence \eqref{Str.rho.F} implies
				\begin{equation}\label{N32}
					\int_0^T \int_{\mathbb{R}^3 \times \mathbb{R}^3} (\rho_{\scriptscriptstyle F_{\varepsilon,\eta}} - \rho_{\scriptscriptstyle F}) \nabla_v Y \cdot u \, \delta_{v=0} \, \mathrm{d}v \mathrm{d}x \mathrm{d}t \xrightarrow{\varepsilon \to 0} 0,
				\end{equation}
				and for the third term, the energy dissipation bound \eqref{global,dish} gives
				\begin{equation}\label{N33}
					\int_0^T \int_{\mathbb{R}^3 \times \mathbb{R}^3} \sqrt{\mu_{\varepsilon,\eta}} \, (\mathbf{I - P}_{\scriptscriptstyle 1}) h_{\varepsilon,\eta} \nabla_v Y \cdot u \, \mathrm{d}v \mathrm{d}x \mathrm{d}t \xrightarrow{\varepsilon \to 0} 0.
				\end{equation}
				Consequently, plugging \eqref{N31}--\eqref{N33} into \eqref{N-3} yields
				\begin{equation}\label{N3}
					\mathcal{N}_3 \xrightarrow{\varepsilon \to 0} 0.
				\end{equation}
				Combining \eqref{N1}, \eqref{N2}, and \eqref{N3} in \eqref{D.}, we conclude that
				\begin{equation*}
					\int_0^T \int_{\mathbb{R}^3 \times \mathbb{R}^3} \Big( F_{\varepsilon,\eta} \nabla_v Y \cdot u_{\varepsilon,\eta} - F \nabla_v Y \cdot u \Big) \, \mathrm{d}v \mathrm{d}x \mathrm{d}t \xrightarrow{\varepsilon \to 0} 0,
				\end{equation*}
				which completes the proof of Lemma \ref{lemma6.2}.
			\end{proof}
			We now turn to the proof of Lemma \ref{Vlasov equation.}. We claim that this proof is due to Bernard-Desvillettes-Golse-Ricci in \cite{MR3668954}.
			\begin{proof}
				Using the definition of the collision integral $\mathcal{D}(F,f)$ in \eqref{D,R,defa}, one has
				\begin{align*}
					&\int_{\mathbb{R}^3} \mathcal{D}(F_{\varepsilon,\eta},f_{\varepsilon,\eta})Y \mathrm{d}v \\
					=&\int_{(\mathbb{R}^3)^2\times\mathbb{S}^2} \{F_{\varepsilon,\eta}(v'')f_{\varepsilon,\eta}(w'') - F_{\varepsilon,\eta}(v)f_{\varepsilon,\eta}(w)\} |(\varepsilon v-w)\cdot \omega| Y(v) \mathrm{d}\omega \mathrm{d}w \mathrm{d}v \\
					=&\int_{(\mathbb{R}^3)^2\times \mathbb{S}^2} F_{\varepsilon,\eta}(v)f_{\varepsilon,\eta}(w) \left\{ Y(v'') - Y(v) \right\} |(\varepsilon v-w)\cdot \omega| \mathrm{d}\omega \mathrm{d}w \mathrm{d}v,
				\end{align*}
				where we have used the change of variables $(w'', v'') \rightarrow (w, v)$, $|(\varepsilon v''-w'')\cdot \omega| = |(\varepsilon v-w)\cdot \omega|$, and $\mathrm{d}w'' \mathrm{d}v'' = \mathrm{d}w \mathrm{d}v$.
				Then, by using the Taylor expansion at order 2 for the $C^2$ function $Y(v'')$, we further have
				\begin{equation}\label{limi-oprat-D}
					\begin{aligned}
						&\frac{1}{\eta}\int_0^T\int_{\mathbb{R}^3\times\mathbb{R}^3} \mathcal{D}(F_{\varepsilon,\eta},f_{\varepsilon,\eta})Y \mathrm{d}v \mathrm{d}x \mathrm{d}t \\
						=&\frac{1}{\eta}\int_0^T\int_{(\mathbb{R}^3)^3\times \mathbb{S}^2} F_{\varepsilon,\eta}(v)f_{\varepsilon,\eta}(w) \nabla_v Y(v)\cdot (v''-v) |(\varepsilon v-w)\cdot \omega| \mathrm{d}\omega \mathrm{d}w \mathrm{d}v \mathrm{d}x \mathrm{d}t \\
					&+ \frac{1}{\eta}\int_0^T\int_{(\mathbb{R}^3)^3\times \mathbb{S}^2} F_{\varepsilon,\eta}(v)f_{\varepsilon,\eta}(w) H(v'',v) : (v''-v)^{\otimes^2} |(\varepsilon v-w)\cdot \omega| \mathrm{d}\omega \mathrm{d}w \mathrm{d}v \mathrm{d}x \mathrm{d}t \\
						=:& I_{\varepsilon,\eta} + H_{\varepsilon,\eta},
					\end{aligned}
				\end{equation}
				where
				\[
				H(v'',v) := \int_{0}^1 (1-t) \nabla^2_v Y\left((1-t)v + tv''\right) \mathrm{d}t.
				\]
				We first treat the term $I_{\varepsilon,\eta}$. From \eqref{velocitya}, which implies that
			\begin{equation}\label{V,rela}
					v'' - v = -\frac{2\eta}{1+\eta} \omega\left( \omega\cdot \left(v - \frac{1}{\varepsilon}w\right)\right)  ,
				\end{equation}
				it follows that
				\[
				I_{\varepsilon,\eta} = -\frac{1}{1+\eta}\int_0^T\int_{\mathbb{R}^3\times\mathbb{R}^3} F_{\varepsilon,\eta}(v) \nabla_v Y(v)\cdot K_{\varepsilon,\eta}(v) \mathrm{d}v \mathrm{d}x \mathrm{d}t,
				\]
				where
				\begin{align*}
					K_{\varepsilon,\eta}(v)
					&:= \frac{2}{\varepsilon}\int_{\mathbb{R}^3\times\mathbb{S}^2} f_{\varepsilon,\eta}(w) |(\varepsilon v-w)\cdot \omega| \{(\varepsilon v-w)\cdot \omega\}\omega \mathrm{d}\omega \mathrm{d}w \\
					&= \frac{2\pi}{\varepsilon}\int_{\mathbb{R}^3} f_{\varepsilon,\eta}(w)(\varepsilon v-w)|\varepsilon v-w| \mathrm{d}w \\
					&= \frac{2\pi}{\varepsilon}\int_{\mathbb{R}^3} \{\mu+\sqrt{\mu}g_{\varepsilon,\eta}\}(w)(\varepsilon v-w)|\varepsilon v-w| \mathrm{d}w.
				\end{align*}
				Hence,
				\begin{equation}\label{I0}
					I_{\varepsilon,\eta} = I_{\varepsilon,\eta}^1 + I_{\varepsilon,\eta}^2 + I_{\varepsilon,\eta}^3 + I_{\varepsilon,\eta}^4 + I_{\varepsilon,\eta}^5,
				\end{equation}
				with
				\begin{align*}
					I_{\varepsilon,\eta}^1 &= -\varepsilon\cdot\frac{2\pi}{1+\eta}\int_0^T\int_{\mathbb{R}^3\times\mathbb{R}^3} F_{\varepsilon,\eta}\nabla_v Y \cdot \int_{\mathbb{R}^3} v\sqrt{\mu}g_{\varepsilon,\eta}|\varepsilon v-w| \mathrm{d}w \mathrm{d}v \mathrm{d}x \mathrm{d}t, \\
					I_{\varepsilon,\eta}^2 &= \frac{2\pi}{1+\eta}\int_0^T\int_{\mathbb{R}^3\times\mathbb{R}^3} F_{\varepsilon,\eta}\nabla_v Y \cdot \int_{\mathbb{R}^3} w\sqrt{\mu}g_{\varepsilon,\eta}\left(|\varepsilon v-w|-|w|\right) \mathrm{d}w \mathrm{d}v \mathrm{d}x \mathrm{d}t, \\
					I_{\varepsilon,\eta}^3 &= \frac{2\pi}{1+\eta}\int_0^T\int_{\mathbb{R}^3\times\mathbb{R}^3} F_{\varepsilon,\eta}\nabla_v Y \cdot \int_{\mathbb{R}^3} w\sqrt{\mu}g_{\varepsilon,\eta}|w| \mathrm{d}w \mathrm{d}v \mathrm{d}x \mathrm{d}t, \\
					I_{\varepsilon,\eta}^4 &= -\frac{2\pi}{1+\eta}\int_0^T\int_{\mathbb{R}^3\times\mathbb{R}^3} F_{\varepsilon,\eta}\nabla_v Y \cdot \int_{\mathbb{R}^3} v \mu|\varepsilon v-w| \mathrm{d}w \mathrm{d}v \mathrm{d}x \mathrm{d}t, \\
					I_{\varepsilon,\eta}^5 &= \frac{1}{\varepsilon}\cdot\frac{2\pi}{1+\eta}\int_0^T\int_{\mathbb{R}^3\times\mathbb{R}^3} F_{\varepsilon,\eta}\nabla_v Y \cdot \int_{\mathbb{R}^3} w\mu\left(|\varepsilon v-w|-|w|\right) \mathrm{d}w \mathrm{d}v \mathrm{d}x \mathrm{d}t.
				\end{align*}	
				
				Using Hölder's inequality with the equivalence relation \eqref{mu}, and the global energy estimates \eqref{global,bdh}--\eqref{global,bdg}, we have	
				\begin{align*}
					|I_{\varepsilon,\eta}^1| &\lesssim \varepsilon \left| \int_0^T\int_{(\mathbb{R}^3)^3} \sqrt{\mu_{\varepsilon,\eta}}h_{\varepsilon,\eta} (\psi+2\chi_2 v)\cdot v \sqrt{\mu} g_{\varepsilon,\eta}|\varepsilon v-w| \mathrm{d}w \mathrm{d}v \mathrm{d}x \mathrm{d}t \right| \\
					&\lesssim \varepsilon \|\sqrt{\mu_{\varepsilon,\eta}}\cdot v\cdot(1+|v|)\cdot(1+|\varepsilon v|)\|_{L^2_v} \|h_{\varepsilon,\eta}\|_{L^\infty ([0,T];L^2_{x,v})} \\
					&\quad \times \|\sqrt{\mu} (1+|w|)\|_{L^2_w} \|g_{\varepsilon,\eta}\|_{L^\infty ([0,T];L^2_{x,w})} \|(\psi,\chi_2)\|_{L^1([0,T]; L^\infty_x)} \\
					&\lesssim \varepsilon\cdot\frac{\eta^{1/2}}{\varepsilon} \lesssim \eta^{1/2} \xrightarrow{\varepsilon \to 0} 0,
				\end{align*}
				which implies that
				\begin{equation}\label{I1}
					I_{\varepsilon,\eta}^1 \xrightarrow{\varepsilon \to 0} 0.
				\end{equation}
				Applying the triangle inequality, we obtain
				\begin{equation}\label{triangle ineq}
					\big||\varepsilon v-w|-|w|\big| = \big||w-\varepsilon v|-|w|\big| \leqslant \varepsilon|v|.
				\end{equation}
				Thus, similar arguments as above deduce that
				\begin{equation}\label{I2}
					I_{\varepsilon,\eta}^2 \xrightarrow{\varepsilon \to 0} 0.
				\end{equation}
				Putting the decomposition of $g_{\varepsilon,\eta}$ in \eqref{expres-g} into $I_{\varepsilon,\eta}^3$, we obtain
				\begin{equation}\label{I31}
					\begin{aligned}
						I_{\varepsilon,\eta}^3
						&= \frac{\kappa}{1+\eta}\int_0^T\int_{\mathbb{R}^3\times\mathbb{R}^3} F_{\varepsilon,\eta}\nabla_v Y \cdot u_{\varepsilon,\eta} \mathrm{d}v \mathrm{d}x \mathrm{d}t \\
						&\quad + \frac{2\pi}{1+\eta}\int_0^T\int_{\mathbb{R}^3\times\mathbb{R}^3} F_{\varepsilon,\eta}\nabla_v Y \cdot \int_{\mathbb{R}^3} w\sqrt{\mu}\mathbf{\{I-P\}}g_{\varepsilon,\eta}|w| \mathrm{d}w \mathrm{d}v \mathrm{d}x \mathrm{d}t.
					\end{aligned}
				\end{equation}
				The two terms on the RHS of \eqref{I31} can be estimated as follows:
				For the first term, from Lemma \ref{lemma6.2}, we have
				\begin{equation}\label{I32}
					\frac{\kappa}{1+\eta}\int_0^T\int_{\mathbb{R}^3\times\mathbb{R}^3} F_{\varepsilon,\eta}\nabla_v Y \cdot u_{\varepsilon,\eta} \mathrm{d}v \mathrm{d}x \mathrm{d}t
					\xrightarrow{\varepsilon \to 0} \kappa \int_0^T\int_{\mathbb{R}^3\times\mathbb{R}^3} F\nabla_v Y \cdot u \mathrm{d}v \mathrm{d}x \mathrm{d}t.
				\end{equation}
				For the second term, from the energy dissipation bound \eqref{global,disg}, we obtain
				\begin{equation}\label{I33}
					\frac{2\pi}{1+\eta}\int_0^T\int_{\mathbb{R}^3\times\mathbb{R}^3} F_{\varepsilon,\eta}\nabla_v Y \cdot \int_{\mathbb{R}^3} w\sqrt{\mu}\mathbf{\{I-P\}}g_{\varepsilon,\eta}|w| \mathrm{d}w \mathrm{d}v \mathrm{d}x \mathrm{d}t \xrightarrow{\varepsilon \to 0} 0.
				\end{equation}
				Thus, plugging the convergences \eqref{I32} and \eqref{I33} into \eqref{I31}, we have
				\begin{equation}\label{I3}
					I_{\varepsilon,\eta}^3 \xrightarrow{\varepsilon \to 0} \kappa\int_0^T\int_{\mathbb{R}^3\times\mathbb{R}^3} F\nabla_v Y \cdot u \mathrm{d}v \mathrm{d}x \mathrm{d}t.
				\end{equation}
				
				Similar to the estimates for $I_{\varepsilon,\eta}^3$, we can establish
				\begin{equation}\label{I4}
					I_{\varepsilon,\eta}^4 \xrightarrow{\varepsilon \to 0} -2\pi\int_0^T\int_{\mathbb{R}^3\times\mathbb{R}^3} F\nabla_v Y \cdot \int_{\mathbb{R}^3} v \mu(w)|w| \mathrm{d}w \mathrm{d}v \mathrm{d}x \mathrm{d}t.
				\end{equation}
				
				Note that
				\begin{equation}\label{I51}
					\begin{aligned}
						&I_{\varepsilon,\eta}^{5} + 2\pi\int_0^T\int_{\mathbb{R}^3\times\mathbb{R}^3} F\nabla_v Y \cdot \int_{\mathbb{R}^3} \mu(w)\frac{w}{|w|}w\cdot v \mathrm{d}w \mathrm{d}v \mathrm{d}x \mathrm{d}t \\
						&\quad=\frac{2\pi}{1+\eta}\int_0^T\int_{\mathbb{R}^3\times\mathbb{R}^3} F_{\varepsilon,\eta}\nabla_v Y \cdot \int_{\mathbb{R}^3} w\mu(w)\left(\frac{|\varepsilon v-w|-|w|}{\varepsilon} + \frac{w}{|w|}\cdot v \right) \mathrm{d}w \mathrm{d}v \mathrm{d}x \mathrm{d}t \\
						&\quad\quad+ \frac{2\pi}{1+\eta}\int_0^T\int_{\mathbb{R}^3\times\mathbb{R}^3} (F-F_{\varepsilon,\eta})\nabla_v Y \cdot \int_{\mathbb{R}^3} \mu(w)\frac{w}{|w|}w\cdot v \mathrm{d}w \mathrm{d}v \mathrm{d}x \mathrm{d}t \\
						&\quad=: I_{\varepsilon,\eta}^{5,1} + I_{\varepsilon,\eta}^{5,2}.
					\end{aligned}	
				\end{equation}
				For $k > 2$, $I_{\varepsilon,\eta}^{5,1}$ is bounded by 
				\begin{align*}
					&2\pi \|F_{\varepsilon,\eta} |v|^{3/2}(1+|v|)^k\|_{L^\infty ([0,T];L^2_{x,v})} \\
				&\quad\quad \times\left\| \frac{\nabla_v Y}{|v|^{3/2}(1+|v|)^k}\cdot \int_{\mathbb{R}^3} w\mu(w)\left(\frac{|\varepsilon v-w|-|w|}{\varepsilon} + \frac{w}{|w|}\cdot v \right) \mathrm{d}w \right\|_{L^1 ([0,T];L^2_{x,v})} \\
					&\quad\lesssim\left\| \frac{\nabla_v Y}{|v|^{3/2}(1+|v|)^k}\cdot \int_{\mathbb{R}^3} w\mu(w)\left(\frac{|\varepsilon v-w|-|w|}{\varepsilon} + \frac{w}{|w|}\cdot v \right) \mathrm{d}w \right\|_{L^1 ([0,T];L^2_{x,v})}.
				\end{align*}
				On the one hand, by using the triangle inequality \eqref{triangle ineq}, we obtain
				\begin{align*}
					&\left| \frac{\nabla_v Y}{|v|^{3/2}(1+|v|)^k}\cdot w\mu(w)\left(\frac{|\varepsilon v-w|-|w|}{\varepsilon} + \frac{w}{|w|}\cdot v \right) \right| \\
					&\quad\lesssim \frac{|\psi+2\chi_2 v|}{|v|^{3/2}(1+|v|)^k} |w|\mu(w) \left( \frac{\big||\varepsilon v-w|-|w|\big|}{\varepsilon} + |v| \right) \\
					&\quad\lesssim\frac{|\psi|+|\chi_2|}{|v|^{1/2}(1+|v|)^{k-1}} |w|\mu(w) \in L^1 ([0,T];L^2_{x,v}(L^1_w)).
				\end{align*}
				On the other hand, for all $t \in [0,T]$, $x \in \mathbb{R}^3$, $v \in \mathbb{R}^3$, $w \in \mathbb{R}^3$, it holds that
				\begin{align*}
					&\left| \frac{\nabla_v Y}{|v|^{3/2}(1+|v|)^k}\cdot w\mu(w)\left(\frac{|\varepsilon v-w|-|w|}{\varepsilon} + \frac{w}{|w|}\cdot v \right) \right| \\
					&\quad\lesssim \frac{|\psi+2\chi_2 v|}{|v|^{3/2}(1+|v|)^k} |w|\mu(w) \left| \int_0^1 \frac{w}{|w|}\cdot v - \frac{w-\varepsilon\theta v}{|w-\varepsilon\theta v|}\cdot v \mathrm{d}\theta \right| \xrightarrow{\varepsilon \to 0} 0.
				\end{align*}
				Therefore, by the dominated convergence theorem, we have
				\begin{equation}\label{I52}
					I_{\varepsilon,\eta}^{5,1} \xrightarrow{\varepsilon \to 0} 0.
				\end{equation}
				Using the same way to treat $I_{\varepsilon,\eta}^{5,2}$ as to treat $I_{\varepsilon,\eta}^{3}$, we can establish
				\begin{equation}\label{I53}
					I_{\varepsilon,\eta}^{5,2} \xrightarrow{\varepsilon \to 0} 0.
				\end{equation}
				Hence, plugging the convergences \eqref{I52} and \eqref{I53} into \eqref{I51}, we obtain
				\begin{equation}\label{I5}
					I_{\varepsilon,\eta}^{5} \xrightarrow{\varepsilon \to 0} -2\pi\int_0^T\int_{\mathbb{R}^3\times\mathbb{R}^3} F\nabla_v Y \cdot \int_{\mathbb{R}^3} \mu(w)\frac{w}{|w|}w\cdot v \mathrm{d}w \mathrm{d}v \mathrm{d}x \mathrm{d}t.
				\end{equation}
				By isotropy, one has
				\[
				\int_{\mathbb{R}^3} \mu(w)\frac{w}{|w|}w\cdot v \mathrm{d}w = \frac{1}{3}v\int_{\mathbb{R}^3} \mu(w)|w| \mathrm{d}w,
				\]
				so that combining \eqref{I4} and \eqref{I5} yields
				\[
				I_{\varepsilon,\eta}^4 + I_{\varepsilon,\eta}^5 \longrightarrow -2\pi\int_0^T\int_{\mathbb{R}^3\times\mathbb{R}^3} F\nabla_v Y \cdot v \int_{\mathbb{R}^3} \mu(w)\left(|w| + \frac{1}{3}|w|\right) \mathrm{d}w \mathrm{d}v \mathrm{d}x \mathrm{d}t.
				\]
				On the other hand, we observe that
				\[
				w\cdot\nabla\mu(w) = -|w|^2\mu(w),
				\]
				so that
				\begin{equation*}
					\begin{aligned}
						\int_{\mathbb{R}^3} \mu(w)(3|w|+|w|) \mathrm{d}w 
						&= \int_{\mathbb{R}^3} \mu(w)\nabla\cdot(w|w|) \mathrm{d}w \\
						&= -\int_{\mathbb{R}^3} w\cdot\nabla\mu(w)|w| \mathrm{d}w = \int_{\mathbb{R}^3} \mu(w)|w|^3 \mathrm{d}w.
					\end{aligned}
				\end{equation*}
				Hence,
				\begin{equation}\label{I45}
					\begin{aligned}
						I_{\varepsilon,\eta}^4 + I_{\varepsilon,\eta}^5 &\longrightarrow \int_0^T\int_{\mathbb{R}^3\times\mathbb{R}^3} F\nabla_v Y \cdot v \left( \int_{\mathbb{R}^3} \frac{2\pi}{3}\mu(w)|w|^3 \mathrm{d}w \right) \mathrm{d}v \mathrm{d}x \mathrm{d}t \\
						&= -\kappa\int_0^T\int_{\mathbb{R}^3\times\mathbb{R}^3} F\nabla_v Y \cdot v \mathrm{d}v \mathrm{d}x \mathrm{d}t.
					\end{aligned}
				\end{equation}
				Therefore, plugging the convergences \eqref{I1}, \eqref{I2}, \eqref{I3}, and \eqref{I45} into \eqref{I0}, we have
				\begin{equation}\label{I}
					I_{\varepsilon,\eta} \longrightarrow -\kappa\int_0^T\int_{\mathbb{R}^3\times\mathbb{R}^3} F\nabla_v Y \cdot (v-u) \mathrm{d}v \mathrm{d}x \mathrm{d}t.
				\end{equation}
				Next, we treat the term $H_{\varepsilon,\eta}$.
				From \eqref{V,rela}, one has
				\begin{align*}
					|H_{\varepsilon,\eta}| &\lesssim \left| \frac{1}{\eta}\int_0^T\int_{(\mathbb{R}^3)^3\times \mathbb{S}^2} F_{\varepsilon,\eta}(v)f_{\varepsilon,\eta}(w) H(v'',v) : \left\{ -\frac{2\eta}{1+\eta}\omega\left( \omega\cdot \left(v - \frac{1}{\varepsilon}w\right)\right) \right\}^{\otimes^2} \right. \\
					&\quad \times \left. |(\varepsilon v-w)\cdot \omega| \mathrm{d}\omega \mathrm{d}w \mathrm{d}v \mathrm{d}x \mathrm{d}t \right| \\
					&\lesssim \frac{\eta}{\varepsilon^2} \left| \int_0^T\int_{(\mathbb{R}^3)^3} \sqrt{\mu_{\varepsilon,\eta}(v)} h_{\varepsilon,\eta}(v) \{\mu+\varepsilon\sqrt{\mu} g_{\varepsilon,\eta}\}(w) \cdot |\chi_2| \cdot |\varepsilon v-w|^3 \mathrm{d}w \mathrm{d}v \mathrm{d}x \mathrm{d}t \right|.
				\end{align*}
				Then, by using the uniform energy bounds \eqref{global,bdg} and \eqref{global,bdh}, we further have
				\[
				|H_{\varepsilon,\eta}| \lesssim \frac{\eta}{\varepsilon^2} = o(1) \xrightarrow{\varepsilon \to 0} 0,
				\]
				which implies that
				\begin{equation}\label{H}
					H_{\varepsilon,\eta} \xrightarrow{\varepsilon \to 0} 0.
				\end{equation}
				Ultimately, plugging the convergences \eqref{I} and \eqref{H} into \eqref{limi-oprat-D}, we can establish 
				\[
				\frac{1}{\eta}\int_0^T\int_{\mathbb{R}^3\times\mathbb{R}^3} \mathcal{D}(F_{\varepsilon,\eta},f_{\varepsilon,\eta})Y \mathrm{d}v \mathrm{d}x \mathrm{d}t
				\xrightarrow{\varepsilon \to 0} -\kappa\int_0^T\int_{\mathbb{R}^3\times\mathbb{R}^3} F\nabla_v Y \cdot (v-u) \mathrm{d}v \mathrm{d}x \mathrm{d}t,
				\]
				which completes the proof.
			\end{proof}
		\end{appendices}


		\noindent{\bf Acknowledgements.} \   The research of {\sc Linjie Xiong}
		was partially supported by the National Natural Science Foundation of China under Grants 12231006. 

		\bigskip
		
		\noindent {\bf Conflict of interest} \  The authors declare that they have no conflict of interest.

		

{\footnotesize
			\begin{thebibliography}{99}\setlength{\itemsep}{0.35mm}
				
				\bibitem{MR1079189} Allaire, G.: \newblock{\href{https://doi.org/10.1007/BF00375065}{Homogenization of the Navier-Stokes equations in open sets perforated with tiny holes I. Abstract framework, a volume distribution of holes}}. {\it Arch. Rational Mech. Anal.}, {\bf 113} (3), 209--259, 1991.	
				
				\bibitem{Ano-Bou} Anoshchenko, O. and Boutet de Monvel-Berthier, A.: \newblock{\href{https://doi.org/10.1002/(SICI)1099-1476(199704)20:6<495::AID-MMA858>3.0.CO;2-O}{The existence of the global generalized solution of the system of equations describing suspension motion}}. {\it Math. Meth. Appl. Sci.}, {\bf 20} (6), 495--519, 1997.
				
				\bibitem{MR2904273} Bae, H.-O., Choi, Y.-P., Ha, S.-Y. and Kang, M.-J.: \newblock{\href{https://doi.org/10.1088/0951-7715/25/4/1155}{Time-asymptotic interaction of flocking particles and an incompressible viscous fluid}}. {\it Nonlinearity}, {\bf 25} (4), 1155, 2012.	
				
				\bibitem{MR1115587} Bardos, C., Golse, F. and Levermore, D.: \newblock{\href{ https://dx.doi.org/10.1007/BF01026608}{Fluid dynamic limits of kinetic equations. I. Formal derivations}}. {\it J. Statist. Phys.}, {\bf 63} (1--2), 323--344, 1991.	
				
				\bibitem{MR1213991} Bardos, C., Golse, F. and Levermore, D.: \newblock{\href{https://doi.org/10.1002/cpa.3160460503}{Fluid dynamic limits of kinetic equations. II. Convergence proofs for the Boltzmann equation}}. {\it Comm. Pure Appl. Math.}, {\bf 46} (5), 667--753, 1993.
				
				\bibitem{MR1115292} Bardos, C. and Ukai, S.: \newblock{\href{https://doi.org/10.1142/S0218202591000137}{The classical incompressible Navier-Stokes limit of the Boltzmann equation}}. {\it Math. Models Methods Appl. Sci.}, {\bf 1} (2), 235--257, 1991.	
				
				\bibitem{MR3668954} Bernard, E., Desvillettes, L., Golse, F. and Ricci, V.: \newblock{\href{https://doi.org/10.4310/CMS.2017.v15.n6.a11}{A derivation of the Vlasov-Navier-Stokes model for aerosol flows from kinetic theory}}. {\it Commun. Math. Sci.}, {\bf 15} (6), 1703--1741, 2017.
				
				\bibitem{MR2555647} Boudin, L., Desvillettes, L., Grandmont, C. and Moussa, A.: \newblock{\href{ISSN: 0893-4983}{Global existence of solutions for the coupled Vlasov and Navier-Stokes equations}}. {\it Differential Integral Equations}, {\bf 22} (11--12), 1247--1271, 2009.
				
				\bibitem{MR3582196} Boudin, L., Grandmont, C., Lorz, A. and Moussa, A.: \newblock{\href{https://doi.org/10.1016/j.jde.2016.10.012}{Global existence of solutions to the incompressible Navier-Stokes-Vlasov equations in a time-dependent domain}}. {\it J. Differential Equations}, {\bf 262} (3), 1317--1340, 2017.	
				
				\bibitem{MR2986590} Boyer, F. and Fabrie, P.: \newblock{\href{ https://doi.org/10.1007/978-1-4614-5975-0}{Mathematical tools for the study of the incompressible Navier-Stokes equations and related models}}. {\it Appl. Math. Sci.}, {\bf 183}, 2013.
				
				\bibitem{MR3894734} Briant, M., Merino-Aceituno, S. and Mouhot, C.: \newblock{\href{https://doi.org/10.1142/S021953051850015X}{From Boltzmann to incompressible Navier–Stokes in Sobolev spaces with polynomial weight}}. {\it Anal. Appl. (Singap.)}, {\bf 17} (1), 85--116, 2019.
				
				\bibitem{MR0709743} Caflisch, R. and Papanicolaou, G.C.: \newblock{\href{https://doi.org/10.1137/0143057}{Dynamic theory of suspensions with Brownian effects}}. {\it SIAM J. Appl. Math.}, {\bf 43} (4), 885--960, 1983.
				
				\bibitem{MR3465376} Carrillo, J.A., Choi, Y.-P. and Karper, T.-K.: \newblock{\href{https://doi.org/10.1016/j.anihpc.2014.10.002}{On the analysis of a coupled kinetic-fluid model with local alignment forces}}. {\it Ann. Inst. H. Poincar\'e{} C Anal. Non Lin\'eaire}, {\bf 33} (2), 273--307, 2016.	
				
				\bibitem{Cercignani-88} Cercignani, C.: \newblock{\href{https://doi.org/10.1007/978-1-4612-1039-9}{The Boltzmann Equation and Its Applications}}. {\it Appl. Math. Sci.}, {\bf 67}, 1988.
				
				\bibitem{MR1307620} Cercignani, C., Illner, R. and Pulvirenti, M.: \newblock{\href{https://doi.org/10.1007/978-1-4419-8524-8}{The mathematical theory of dilute gases}}. {\it Appl. Math. Sci.}, {\bf 106}, 1994.
				
				\bibitem{MR2825335} Chae, M., Kang, K. and Lee, J.: \newblock{\href{https://doi.org/10.1016/j.jde.2011.07.016}{Global existence of weak and classical solutions for the Navier-Stokes-Vlasov-Fokker-Planck equations}}. {\it J. Differential Equations}, {\bf 251} (9), 2431--2465, 2011.
				
				\bibitem{MR2564289} Charles, F. and Desvillettes, L.: \newblock{\href{https://doi.org/10.1007/s10955-009-9858-2}{Small mass ratio limit of Boltzmann equations in the context of the study of evolution of dust particles in a rarefied atmosphere}}. {\it J. Stat. Phys.}, {\bf 137} (3), 539--567, 2009.
				
				\bibitem{MR3723164} Choi, Y.-P.: \newblock{\href{https://doi.org/10.1016/j.matpur.2017.05.019}{Finite-time blow-up phenomena of Vlasov/Navier-Stokes equations and related systems}}. {\it J. Math. Pures Appl. (9)}, {\bf 108} (6), 991--1021, 2017.
				
				\bibitem{MR4344262} Choi, Y.-P. and Jung, J.: \newblock{\href{https://doi.org/10.1142/S0218202521500482}{Asymptotic analysis for a Vlasov–Fokker–Planck/Navier–Stokes system in a bounded domain}}. {\it Math. Models Methods Appl. Sci.}, {\bf 31} (11), 2213--2295, 2021.
				
				\bibitem{MR3403400} Choi, Y.-P. and Kwon, B.: \newblock{\href{https://doi.org/10.1088/0951-7715/28/9/3309}{Global well-posedness and large-time behavior for the inhomogeneous Vlasov-Navier-Stokes equations}}. {\it Nonlinearity}, {\bf 28} (9), 3309--3336, 2015.
				
				\bibitem{MR5041103} Danchin, R.: \newblock{\href{https://doi.org/10.1007/s00205-026-02171-x}{Fujita-Kato solutions and optimal time decay for the Vlasov-Navier–Stokes system in the whole space}}. {\it Arch. Ration. Mech. Anal.}, {\bf 250} (2), Paper No. 13, 2026.
				
				\bibitem{MR1029125} De Masi, A., Esposito, R. and Lebowitz, J. L.:  \newblock{\href{https://doi.org/10.1002/cpa.3160420810}{Incompressible Navier-Stokes and Euler limits of the Boltzmann equation}}. {\it Comm. Pure Appl. Math.}, {\bf 42} (8), 1189--1214, 1989.
				
				\bibitem{Springerlink-ews} Desjardins, B. and Esteban, M.J.: \newblock{\href{https://doi.org/10.1007/s002050050136}{Existence of weak solutions for the motion of rigid bodies in a viscous fluid}}. {\it Arch. Ration. Mech. Anal.}, {\bf 146}, 59--71, 1999.
				
				\bibitem{MR2398959} Desvillettes, L., Golse, F. and Ricci, V.: \newblock{\href{https://doi.org/10.1007/s10955-008-9521-3}{The mean-field limit for solid particles in a Navier-Stokes flow}}. {\it J. Stat. Phys.}, {\bf 131} (5), 941--967, 2008.	
				
				\bibitem{MR1014927} DiPerna, R. J. and Lions, P.-L.: \newblock{\href{https://doi.org/10.2307/1971423}{On the Cauchy problem for Boltzmann equations: global existence and weak stability}}. {\it Ann. of Math. (2)}, {\bf 130} (2), 321--366, 1989.
				
				\bibitem{MR4591940} El Ghani, N. and Mejri, H.: \newblock{\href{https://doi.org/10.1142/S0129167X23500313}{The hydrodynamic limit for the inhomogeneous Vlasov-Navier-Stokes system}}. {\it Internat. J. Math.}, {\bf 34} (6), 2023.
				
				\bibitem{MR4698662} Flynn, P. and Guo, Y.: \newblock{\href{https://doi.org/10.1007/s00220-023-04901-8}{The massless electron limit of the Vlasov-Poisson-Landau system}}. {\it Comm. Math. Phys.}, {\bf 405} (2), 2024.
				
				\bibitem{MR1379589} Glassey, R.T.: \newblock{\href{ISBN: 0-89871-367-6 }{The Cauchy problem in kinetic theory}}. {\it Society for Industrial and Applied Mathematics (SIAM)}, 1996.
				
				\bibitem{MR3842054} Glass, O., Han-Kwan, D. and Moussa, A.: \newblock{\href{https://doi.org/10.1007/s00205-018-1253-1}{The Vlasov-Navier-Stokes system in a 2D pipe: existence and stability of regular equilibria}}. {\it Arch. Ration. Mech. Anal.}, {\bf 230} (2), 593--639, 2018.
				
				\bibitem{MR4172441} Golse, F.: \newblock{\href{https://doi.org/10.1088/1873-7005/abb5ef}{From the kinetic theory of gases to models for aerosol flows}}. {\it Fluid Dyn. Res.}, {\bf 52} (5), 051401, 2020.	
				
				\bibitem{MR2517786} Golse, F. and Saint-Raymond, L.: \newblock{\href{https://doi.org/10.1016/j.matpur.2009.01.013}{The incompressible Navier-Stokes limit of the Boltzmann equation for hard cutoff potentials}}. {\it J. Math. Pures Appl. (9)}, {\bf 91} (5), 508--552, 2009.	
				
				\bibitem{MR2025302} Golse, F. and Saint-Raymond, L.: \newblock{\href{https://doi.org/10.1007/s00222-003-0316-5}{The Navier-Stokes limit of the Boltzmann equation for bounded collision kernels}}. {\it Invent. Math.}, {\bf 155} (1), 81--161, 2004.
				
				\bibitem{MR1869641} Goudon, T.: \newblock{\href{https://doi.org/10.1017/S030821050000144X}{Asymptotic problems for a kinetic model of two-phase flow}}. {\it Proc. Roy. Soc. Edinburgh Sect. A}, {\bf 131} (6), 1371--1384, 2001.
				
				\bibitem{MR2106333} Goudon, T., Jabin, P.-E. and Vasseur, A.: \newblock{\href{https://doi.org/10.1512/iumj.2004.53.2508}{Hydrodynamic limit for the Vlasov-Navier-Stokes equations. I. Light particles regime}}. {\it Indiana Univ. Math. J.}, {\bf 53} (6), 1495--1515, 2004.
				
				\bibitem{MR2106334} Goudon, T., Jabin, P.-E. and Vasseur, A.: \newblock{\href{https://doi.org/10.1512/iumj.2004.53.2509}{Hydrodynamic limit for the Vlasov-Navier-Stokes equations. II. Fine particles regime}}. {\it Indiana Univ. Math. J.}, {\bf 53} (6), 1517--1536, 2004.
				
				\bibitem{MR2095473} Guo, Y.: \newblock{\href{https://doi.org/10.1512/iumj.2004.53.2574}{The Boltzmann equation in the whole space}}. {\it Indiana Univ. Math. J.}, {\bf 53} (4), 1081--1094, 2004.	
				
				\bibitem{MR1908664} Guo, Y.: \newblock{\href{ https://doi.org/10.1002/cpa.10040}{The Vlasov-Poisson-Boltzmann system near Maxwellians}}. {\it Comm. Pure Appl. Math.}, {\bf 55} (9), 1104--1135, 2002.
				
				\bibitem{MR1610309} Hamdache, K.: \newblock{\href{https://doi.org/10.1007/BF03167396}{Global existence and large time behaviour of solutions for the Vlasov-Stokes equations}}. {\it Japan J. Indust. Appl. Math.}, {\bf 15} (1), 51--74, 1998.
				
				\bibitem{MR4420295} Han-Kwan, D.: \newblock{\href{https://doi.org/10.2140/pmp.2022.3.35}{Large-time behavior of small-data solutions to the Vlasov-Navier-Stokes system on the whole space}}. {\it Probab. Math. Phys.}, {\bf 3} (1), 35--67, 2022. 
				
				\bibitem{Han-Kwan-Michel-23} Han-Kwan, D. and Michel, D.: 
				\newblock{\href{https://doi.org/10.1090/memo/1516}{On hydrodynamic limits of the Vlasov-Navier-Stokes system}}. {\it Mem. Amer. Math. Soc.}, {\bf 302} (1516), 115pp, 2024. 
				
				\bibitem{MR4061981} Han-Kwan, D., Miot, É., Moussa, A. and Moyano, I.: \newblock{\href{https://doi.org/10.4171/rmi/1120}{Uniqueness of the solution to the 2D Vlasov-Navier-Stokes system}}. {\it Rev. Mat. Iberoam.}, {\bf 36} (1), 37--60, 2020.	
				
				\bibitem{MR4076066} Han-Kwan, D., Moussa, A. and Moyano, I.: \newblock{\href{https://doi.org/10.1007/s00205-020-01491-w}{Large time behavior of the Vlasov-Navier-Stokes system on the torus}}. {\it Arch. Ration. Mech. Anal.}, {\bf 236} (3), 1273--1323, 2020.
				
				\bibitem{MR3479195} Herda, M.: \newblock{\href{https://doi.org/10.1016/j.jde.2016.02.005}{On massless electron limit for a multispecies kinetic system with external magnetic field}}. {\it J. Differential Equations}, {\bf 260} (11), 7861--7891, 2016. 
				
				\bibitem{MR2094523} Jabin, P.-E. and Otto, F.: \newblock{\href{https://doi.org/10.1007/s00220-004-1126-3}{Identification of the Dilute Regime in Particle Sedimentation}}. {\it Comm. Math. Phys.}, {\bf 250} (2), 415--432, 2004.
				
				\bibitem{MR3875244} Jiang, N., Xu, C.J. and Zhao, H.J.: \newblock{\href{https://doi.org/10.1512/iumj.2018.67.5940}{Incompressible Navier-Stokes-Fourier limit from the Boltzmann equation: classical solutions}}. {\it Indiana Univ. Math. J.}, {\bf 67} (5), 1817--1855, 2018.
				
				\bibitem{MR0553964} Kawashima, S., Matsumura, A. and Nishida, T.: \newblock{\href{http://projecteuclid.org/euclid.cmp/1103907292}{On the fluid-dynamical approximation to the Boltzmann equation at the level of the Navier-Stokes equation}}. {\it Comm. Math. Phys.}, {\bf 70} (2), 97--124, 1979.
				
				\bibitem{MR2043729} Liu, T.-P., Yang, T. and Yu, S.-H.: \newblock{\href{https://doi.org/10.1016/j.physd.2003.07.011}{Energy method for Boltzmann equation}}. {\it Phys. D}, {\bf 188} (3--4), 178--192, 2004.
				
				\bibitem{MR0565234} L'vov, V.A. and Hruslov, E.Ja.:
				Perturbation of a viscous incompressible fluid by small particles. {\it ``Naukova Dumka'', Kiev}, pp. 173--177 and 267, 1978. 
				
				\bibitem{MR2415460} Mellet, A. and Vasseur, A.: \newblock{\href{https://doi.org/10.1007/s00220-008-0523-4}{Asymptotic analysis for a Vlasov-Fokker-Planck/compressible Navier-Stokes system of equations}}. {\it Comm. Math. Phys.}, {\bf 281} (3), 573--596, 2008.	
				
				\bibitem{MR0503305} Nishida, T.: \newblock{\href{http://projecteuclid.org/euclid.cmp/1103904211}{Fluid dynamical limit of the nonlinear Boltzmann equation to the level of the compressible Euler equation}}. {\it Comm. Math. Phys.}, {\bf 61} (2), 119--148, 1978.	
				
				\bibitem{MR4296180} Rachid, M.: \newblock{\href{ https://doi.org/10.3934/krm.2021017}{Incompressible Navier-Stokes-Fourier limit from the Landau equation}}. {\it Kinet. Relat. Models}, {\bf 14} (4), 599--638, 2021.		
				
				\bibitem{MR4493149} Su, Y., Wu, G., Yao, L. and Zhang, Y.: \newblock{\href{https://doi.org/10.1016/j.jde.2022.09.029}{Hydrodynamic limit for the inhomogeneous incompressible Navier-Stokes-Vlasov equations}}. {\it J. Differential Equations}, {\bf 342}, 193--238, 2023.
				
				\bibitem{MR3369269} Wang, D. and Yu, C.: \newblock{\href{https://doi.org/10.1016/j.jde.2015.05.016}{Global weak solution to the inhomogeneous Navier-Stokes-Vlasov equation}}. {\it J. Differential Equations}, {\bf 259} (8), 3976--4008, 2015.	
				
				\bibitem{Williams-85} Williams, F.A.: \newblock{\href{https://doi.org/10.1201/9780429494055}{Combustion theory, second edition}}, 1985.		
				
				\bibitem{MR3073216} Yu, C.: \newblock{\href{https://doi.org/10.1016/j.matpur.2013.01.001}{Global weak solutions to the incompressible Navier–Stokes–Vlasov equations}}. {\it J. Math. Pures Appl. (9)}, {\bf 100} (2), 275--293, 2013.
				
		\end{thebibliography} }
		
				
	\end{document}